\documentclass[11pt]{amsart}

\usepackage{amssymb,verbatim}
\usepackage{amsmath,amsfonts}
\usepackage[mathscr]{euscript}
\usepackage{amsthm}
\usepackage{url}
\usepackage{graphicx}
\usepackage{enumitem}
\usepackage{hyperref}
\hypersetup{colorlinks}
\usepackage{bbm}
\usepackage{bm}
\usepackage{mathrsfs}
\usepackage{cancel}
\usepackage[colorinlistoftodos]{todonotes}

\RequirePackage{cleveref}
\usepackage{hypcap}
\hypersetup{colorlinks=true, citecolor=darkblue, linkcolor=darkblue}
\definecolor{darkblue}{rgb}{0.0,0,0.7}
\newcommand{\darkblue}{\color{darkblue}}

\definecolor{darkred}{rgb}{0.68,0,0}
\newcommand{\darkred}{\color{darkred}}

\definecolor{darkgreen}{rgb}{0,.38,0}

\definecolor{magenta}{rgb}{.51, 0, .51}
\newcommand{\magenta}{\color{magenta}}

\newcommand{\defn}[1]{\emph{\darkblue #1}}
\newcommand{\defna}[1]{\emph{\darkred #1}}
\newcommand{\defnb}[1]{\emph{\darkblue #1}}
\newcommand{\defnm}[1]{\emph{\magenta #1}}
\newcommand{\thmb}[1]{\textnormal{\darkblue #1}}
\newcommand{\defng}[1]{\emph{#1}}

\setlist[enumerate]{
	label=\textnormal{({\roman*})},
	ref={\roman*}}

\makeatletter
\def\th@plain{%
	\thm@notefont{}% same as heading font
	\itshape % body font
}
\def\th@definition{%
	\thm@notefont{}% same as heading font
	\normalfont % body font
}
\makeatother

\newtheorem{thm}{Theorem}[section]
\newtheorem{lemma}[thm]{Lemma}

\newtheorem*{claim*}{Claim}
\newtheorem{cor}[thm]{Corollary}
\newtheorem{prop}[thm]{Proposition}
\newtheorem{conj}[thm]{Conjecture}
\newtheorem{conv}[thm]{Convention}

\theoremstyle{definition}
\newtheorem{ex}[thm]{Example}
\newtheorem{rem}[thm]{Remark}
\newtheorem{definition}[thm]{Definition}

\numberwithin{figure}{section}
\numberwithin{equation}{section}

\def\wh{\widehat}

\def\emp{\nothing}

\def\zz{\mathbb Z}

\def\nn{\mathbb N}
\def\cc{\mathbb C}

\def\rr{\mathbb R}

\def\la{\lambda}

\def\de{\delta}

\def\al{\alpha}
\def\be{\beta}

\def\cC{\mathcal C}
\def\cB{{\mathcal{B}}}

\def\cA{\mathcal A}

\def\cD{\mathcal D}
\def\cE{\mathscr E}
\def\cF{\mathscr F}

\def\cL{\mathcal L}

\def\cO{\mathcal O}
\def\cp{\mathcal P}
\def\cP{\mathcal P}

\def\cQ{\mathcal Q}
\def\cR{\mathcal R}

\def\cU{\mathcal U}

\def\bP{\mathbf{P}}

\def\<{\langle}
\def\>{\rangle}

\def\rB{\text{{\rm B}}}

\def\GL{ {\text {\rm GL} } }

\def\Cat{ {\text {\rm Cat} } }
\def\IT{ {\text {\rm IT} } }
\def\EF{ {\text {\rm EF} } }
\def\Sh{ \small {\text {\rm $\mathbb{Y}$} } }

\def\0{{\mathbf 0}}

\def\nothing{\varnothing}

\def\.{\hskip.06cm}
\def\ts{\hskip.03cm}

\def\lra{\leftrightarrow}

\def\bx{\xx}

\def\by{\yy}
\def\yb{\yy}

\def\La{\Lambda}

\newcommand{\SSYT}{\operatorname{SSYT}}

\DeclareMathOperator{\aH}{\textnormal{H}}

\def\.{\hskip.06cm}
\def\ts{\hskip.03cm}

\def\nin{\noindent}

\newcommand{\textsu}[1]{\textup{\textsf{#1}}}

\newcommand{\ComCla}[1]{\textup{\textsu{#1}}}
\newcommand{\sharpP}{\ComCla{\#P}}
\newcommand{\SP}{\ComCla{\#P}}

\def\SP{\sharpP}

\DeclareMathOperator{\Pb}{\mathbb{P}}

\def\xx{\textbf{\textit{x}}}
\def\yy{\textbf{\textit{y}}}
\def\zb{\textbf{\textit{z}}}

\def\aai{\textbf{{\textit{a}}}}
\def\bbi{\textbf{{\textit{b}}}}

\DeclareMathOperator{\ab}{\mathbf{a}}
\DeclareMathOperator{\bb}{\mathbf{b}}

\DeclareMathOperator{\eb}{\mathbf{e}}

\DeclareMathOperator{\cX}{\mathcal{X}}
\DeclareMathOperator{\xb}{\xx}

\DeclareMathOperator{\aL}{\textnormal{$\mathcal{L}$}}

\DeclareMathOperator{\cM}{\mathscr{M}}
\DeclareMathOperator{\cN}{\mathscr{N}}
\DeclareMathOperator{\fa}{\textnormal{a}}

\DeclareMathOperator{\fb}{\textnormal{b}}

\DeclareMathOperator{\fc}{\textnormal{c}}

\DeclareMathOperator{\fd}{\textnormal{d}}

\DeclareMathOperator{\ff}{\textnormal{f}}

\DeclareMathOperator{\fg}{\textnormal{g}}
\DeclareMathOperator{\fm}{\textnormal{m}}
\DeclareMathOperator{\fh}{\textnormal{h}}
\DeclareMathOperator{\fG}{\textnormal{G}}

\DeclareMathOperator{\fs}{\textnormal{s}}
\DeclareMathOperator{\fS}{\textnormal{S}}

\DeclareMathOperator{\cV}{\mathcal{V}}
\DeclareMathOperator{\cW}{\mathcal{W}}
\DeclareMathOperator{\Imm}{\textnormal{Imm}}

\DeclareMathOperator{\sort}{\textnormal{sort}}
\DeclareMathOperator{\tG}{\widetilde{\textnormal{G}}}
\DeclareMathOperator{\TL}{\textnormal{TL}}

\newcommand{\blue}[1]{\textcolor{blue}{#1}}
\newcommand{\red}[1]{\textcolor{red}{#1}}
\title{Correlation inequalities for Schur positivity}
\date{\today}

\author{Swee Hong Chan}
\address[Swee Hong Chan]{Department of Mathematics, Rutgers University,  Piscataway, NJ 08854.}
\email{\texttt{sweehong.chan@rutgers.edu}}

\author{Hong Chen}
\address[Hong Chen]{Department of Mathematics, Rutgers University,  Piscataway, NJ 08854.}
\email{\texttt{hc813@scarletmail.rutgers.edu}}

\author[\ts Igor Pak]{Igor Pak}
\address[Igor Pak]{Department of Mathematics, UCLA,  Los Angeles, CA 90095.}
\email{\texttt{pak@math.ucla.edu}}

\author[\ts Daniel Soskin]{Daniel Soskin}
\address[Daniel Soskin]{Department of Mathematics, UCLA,  Los Angeles, CA 90095.}
\email{\texttt{dsoskin@math.ucla.edu}}

\begin{document}

\begin{abstract}
We generalize the Ahlswede--Daykin inequality (1978) to a Schur positive
\emph{ADS inequality}, which also contains the Lam--Postnikov--Pylyavskyy
inequality (2007) as a special case.  We then present a number of
further generalizations and applications.  Notably, we resolve Mihalcea's
conjecture on log-supermodularity of stable Grothendieck polynomials.
\end{abstract}

\maketitle
	
\vskip-.5cm

% \newpage

\section{Introduction}\label{s:intro}

\subsection{Foreword} \label{ss:intro-foreword}
The extensive use of correlation inequalities in combinatorics, probability and across the sciences
lets one forget the remarkable story behind these inequalities.  The first
instances of such correlation inequalities were introduced in the 1960s,
independently by Harris \cite{Har60} and Kleitman \cite{Kle66}, to solve
very specific problems in percolation theory and in extremal combinatorics,
respectively.

The original proofs by induction had somewhat obscured the real depth
of the \defng{Harris--Kleitman {\rm (HK)} inequality}.  What followed was a long series of
generalizations, all extending the inequalities and the inductive arguments
in the increasingly clever ways.  Of these, the
\defng{Fortuin--Kasteleyn--Ginibre {\rm (FKG)} inequality} \cite{FKG71},
the \defng{Holley inequality}  \cite{Hol74}, the \defng{Daykin inequality} \cite{Day77},
and the \defna{Ahlswede--Daykin {\rm (AD)} inequality} \cite{AD78}, are
probably the best known.  Further generalizations include $q$- and multivariate
correlation inequalities \cite{Bjo11,Chr09,CP-multi,LP07},
the \defng{localization inequalities}~by Lov\'asz--Saks \cite{LS06},
multiple events versions \cite{AK96,Gla-FKG,RS93},
and most recently the
decision tree versions \cite{Gla24,Kern20}.

Let us emphasize a key feature of AD type generalizations of the HK inequality:
the \defna{local-to-global principle}.  The goal is no longer to establish positive
correlations between specific probability measures, but between measures
which satisfy certain natural ``local'' assumptions.  These local properties hold
trivially for the probabilistic and counting applications in the Harris and Kleitman
setting.  The ``global'' conclusions have a similar structure as local assumptions,
the proofs become streamlined, and the results turn out to be much stronger
due to flexibility of local assumptions.

In this paper we introduce a local-to-global principle for correlation inequalities
involving symmetric functions, where the inequalities are stated in terms
of \defng{Schur positivity}.  We replace the \defng{log-supermodularity} \ts assumption
in the AD inequality, with the celebrated
\defng{Lam--Postnikov--Pylyavskyy inequality} \ts \eqref{eq:LPP}, where the
meet and join operations on subsets are defined as intersection and union of Young
diagrams.

Our main result is the \defna{ADS inequality} \eqref{eq:ADS},
a Schur positive generalization of the AD inequality, which similarly
has a number of important applications, see below.
Since ADS inequality contains the LPP inequality as a trivial special case,
our proofs are no longer inductive and involve large refinement into
what we call \defng{Schur orchestra inequalities} (Theorem~\ref{thm:orch}).
The latter are then proved
by a technical arguments involving \defng{Temperley--Lieb immanants}, a
technique introduced by Rhoades--Skandera \cite{RS05,RS06},
and further extended in \cite{LPP07,SS25} (see also \cite{AKOS25,NP25,NP26,PS26}).

We now give some probabilistic and algebraic combinatorics
background, before proceeding to state the main results.  The
interested reader familiar with the subject is encouraged to jump
ahead to~$\S$\ref{ss:intro-main}.

\smallskip

%%%%%%%%%%%%%%%%%%%%%%%%%%%%%%%%%%%%%%%%%%%%%%%%%%%%%%%%%%%%%%%%%%%%%%%%%%%%%%%%%%%
%%%%%%%%%%%%%%%%%%%%%%%%%%%%%%%%%%%%%%%%%%%%%%%%%%%%%%%%%%%%%%%%%%%%%%%%%%%%%%%%%%%
%%%%%%%%%%%%%%%%%%%%%%%%%%%%%%%%%%%%%%%%%%%%%%%%%%%%%%%%%%%%%%%%%%%%%%%%%%%%%%%%%%%
%%%%%%%%%%%%%%%%%%%%%%%%%%%%%%%%%%%%%%%%%%%%%%%%%%%%%%%%%%%%%%%%%%%%%%%%%%%%%%%%%%%
%%%%%%%%%%%%%%%%%%%%%%%%%%%%%%%%%%%%%%%%%%%%%%%%%%%%%%%%%%%%%%%%%%%%%%%%%%%%%%%%%%%
%%%%%%%%%%%%%%%%%%%%%%%%%%%%%%%%%%%%%%%%%%%%%%%%%%%%%%%%%%%%%%%%%%%%%%%%%%%%%%%%%%%
%%%%%%%%%%%%%%%%%%%%%%%%%%%%%%%%%%%%%%%%%%%%%%%%%%%%%%%%%%%%%%%%%%%%%%%%%%%%%%%%%%%

\subsection{Harris--Kleitman inequality}
\label{ss:intro-HK}
We start with Kleitman's enumerative version of the
\defn{HK inequality} \cite{Kle66}, which is easiest to state.
A collection \. $\cU\subset 2^{[n]}$ \.  of subsets of \ts $[n]=\{1,\ldots,n\}$ \ts
is called \defn{up-closed}, if \ts $B \in \cU$ \ts for all \ts $B \subseteq A$, \ts $A\in \cU$.
Then:
\begin{equation}\label{eq:Kleitman}\tag{HK}
|\cU| \cdot |\cV| \, \le \, |\cU\cap \cV| \ts \cdot \ts 2^n\ts,
\end{equation}
for every two up-closed collections \ts $\cU,\cV\subset 2^{[n]}$.  Stated in
probabilistic terms, this is equivalent to
\begin{equation}\label{eq:HK}
\bP[A\in \cU \cap \cV] \, \ge \, \bP[A\in \cU] \. \cdot \. \bP[A\in \cV]\ts,
\end{equation}
where the probability is over uniform random subsets of $\cU$.
For example, for collections of edges of a complete graph~$K_r\ts$,
the result implies:
\begin{equation}\label{eq:Kleitman-sub}
\bP[G~\text{is triangle-free and planar}] \, \ge \, \bP[G~\text{is triangle-free}] \. \cdot \. \bP[G~\text{is planar}]\ts,
\end{equation}
where the probability is over uniform subgraphs $G$ of~$K_r\ts$.

In this notation, Harris's version \cite{Har60}, can be viewed as a probabilistic
generalization of Kleitman's version \eqref{eq:Kleitman}, where the probability measure on
collections is a product measure on~$2^{[n]}$.  For example, this implies
that \eqref{eq:Kleitman-sub} holds for the probability \ts $\bP_p$ \ts
in the Erd\H{o}s--R\'enyi \ts $G(n,p)$ \ts random graph model.\footnote{To
quote Harris's own characterization of the HK inequality,
``its truth seems obvious, but I have not been able to find a shorter
proof [..] nor have I seen the result stated elsewhere'' \cite{Har60}.}

In its basic application, the HK inequality implies that the \defn{critical probability}
$$
p_c \, := \, \sup\big\{p : \bP_p(x\lra \infty)=0\big\}
$$
for the $p$-percolation on an infinite connected graph \ts $G=(V,E)$,
is independent on the vertex $x\in V$, see e.g.\ \cite{BR06,Gri99}.
Here by \ts $\bP_p(x\lra \infty)$ \ts we mean the probability that $x$ lies in the
infinite cluster (connected component).
Now, the idea is that for every two vertices \ts $x,y\in V$,
we have:
$$\aligned
\bP_p(x\lra \infty) \. &> \. 0 \ \ \Longleftrightarrow \ \ \bP_p(y\lra \infty) \. > \. 0 \quad \ \ \text{for all \, $p\in (0,1]$, \, since} \\
\bP_p(x\lra \infty)  \. & \ge \. \bP_p(x\lra y \lra \infty)  \. \ge_{\eqref{eq:HK}} \. \bP_p(x\lra y) \cdot \bP_p(y\lra \infty).
\endaligned
$$
For the case when $G = \zz^2$ is a square lattice, Harris used the HK~inequality
to prove that \ts $p_c\ge \frac12$ \ts \cite{Har60}.
Famously, Kesten~\cite{Kes80} established the equality \ts $p_c= \frac12$ \ts
twenty years later.

\smallskip

\subsection{Ahlswede--Daykin inequality}
\label{ss:intro-AD}
Skipping over many intermediate generalizations, we now state the
\defna{AD inequality}, also called the \defng{four functions theorem},
see e.g.\ \cite[$\S$6.1]{AS16}. This result is central to this paper,
both as a technical theorem and as an inspiration.

Fix integers \ts $\ell,w\ge 1$, and let \ts
$\cX=\cX^{(\ell,w)}$ \ts denote the set of integer sequences \.
$\lambda = (\lambda_1, \ldots,\lambda_\ell)\in \{0,1,\ldots,w\}^\ell$. 
For \ts $\lambda,\mu \in \cX$, define the \defnb{join} ($\vee$) and \defnb{meet} ($\wedge$) operations as follows:
\[
\lambda \vee \mu \ := \  \big(\max\{\lambda_1,\mu_1\}, \ldots,\max\{\lambda_\ell,\mu_\ell\}\big), \qquad   \lambda \wedge \mu \ := \  \big(\min\{\lambda_1,\mu_1\}, \ldots, \min\{\lambda_\ell,\mu_\ell\}\big).
\]
Under these operations, \ts $(\cX,\vee,\wedge)$ \ts forms a distributive lattice, cf.\ Remark~\ref{rem:dist}.

\smallskip

\begin{thm}[{\rm\defn{AD inequality{}\ts}~\cite{AD78}}{}]\label{thm:AD}
 Fix integers \ts $\ell,w\ge 1$, and let \ts  \ts $\cX=\cX^{(\ell,w)}$ \ts be as above.
 Let  \.
 $\fa,\fb,\fc,\fd: \cX \to \rr_{\geq 0}$ \ts be four functions satisfying
 \begin{equation}\label{eq:AD-cond}\tag{\defn{\em AD-cond}}
 	\fa({\lambda}) \. \fb({\mu}) \ \leq \  \fc({\lambda \vee \mu}) \. \fd({\lambda \wedge \mu}) \qquad \text{ for all \. $\lambda,\mu \in \cX$.}
 \end{equation}
Then:
\begin{equation}
	\label{eq:AD}\tag{\defn{\em AD}}
	\bigg( \sum_{\lambda \in \cX}  \fa (\lambda)\bigg)
	 \bigg( \sum_{\lambda \in \cX}  \fb (\lambda)\bigg)
	 \ \leq \
	 	\bigg( \sum_{\lambda \in \cX}  \fc (\lambda)\bigg)
	 \bigg( \sum_{\lambda \in \cX}  \fd (\lambda)\bigg).
\end{equation}
\end{thm}

\smallskip

The assumption \eqref{eq:AD-cond} is a four function generalization of
\defn{log-supermodularity} \ts property
\begin{equation}\label{eq:lsm}
f(x) \. f(y) \, \le \, f(x \vee y) \. f(x\wedge y) \quad \text{for all} \quad x,y\in \cL,
\end{equation}
where $\cL$ is a distributive lattice.  See, e.g., \cite{Bjo11,SW75} for a discussion
of the significance of this property across combinatorics.

There are numerous combinatorial applications of the AD inequality, ranging
from extremal combinatorics to combinatorial number theory and order theory,
see e.g.\ monographs chapters \cite[$\S$6]{And87}, \cite[$\S$19]{Bol86},
and dedicated surveys \cite{FS00,Win86}.  Let us single out
the \defng{XYZ inequality} \ts for the numbers of linear extensions
\cite{She82} (see also \cite[$\S$6.4]{AS16} and \cite[Thm.~12.4.41]{West21}),
and other related correlation inequalities reviewed in \cite[$\S$14.4]{CP-LE}.

As we mentioned earlier, the FKG inequality is an important special
case used throughout percolation theory, see e.g.\ \cite{BR06,Gri99,Gri06,Wer09}
and~$\S$\ref{ss:finrem-more-corr}.
Let us emphasize that there are cases when the FKG is insufficient and
the full power of AD inequality is needed, see e.g.\ a discussion in
\cite{FS00}, and notable examples in \cite{FDS88,CP-multi}.
Finally, note that the HK inequality \eqref{eq:Kleitman} follows from \eqref{eq:AD}
by taking \ts $w=1$,
\ts $\ell=n$, \ts $\fc=\mathbf{1}$, and setting \ts $\fa$, $\fb$, $\fd$ \ts
to be indicator functions of \ts
$\cU$, \ts $\cV$ \ts and \ts $\cU\cap\cV$, respectively.

\smallskip

The following generalization of the \ts $w=1$ \ts case of the AD inequality
is both curious and essential to our work.  This result was first obtained
by Reuter \cite{Reu87}.  It was later rediscovered by Lov\'asz--Saks \cite[Cor.~4]{LS04}
as a special case of their family of localization inequalities, and most recently
rediscovered again in \cite[Claim~6.3]{CP-multi}.

\smallskip

\begin{thm}[{\rm \defn{RLS inequality}\ts}{}]\label{thm:RLS}\label{t:RLS} 
	Let \ts $k\geq 1$, 
	and let \. $\fa,\fb,\fc,\fd: 2^{[k]}\to \rr_{\geq 0}$ \. be functions satisfying \eqref{eq:AD-cond}.
	Then:
	\begin{equation}\label{eq:RLS}\tag{\defn{\em RLS}}
\sum_{S \ts\subseteq\ts [k]} \fa(S) \fb(\overline{S}) \ \leq \
		\sum_{S \subseteq [k]}   \fc(S)  \fd(\overline{S})\ts,
	\end{equation}
	where \ts $\overline{S} := [k] \setminus S$ \ts denotes the set complement.
\end{thm}

\smallskip

{\small
\begin{rem}\label{rem:dist}
By \defng{Birkhoff's representation theorem}~\cite{Bir}, every finite distributive lattice is isomorphic to a sublattice of a Boolean lattice \ts $\rB_\ell=\cX^{(\ell,1)}$.
While the AD inequality is typically stated for finite distributive lattices,
we use this equivalent formulation as it is better adapted to the content of this paper.
\end{rem}
}
\smallskip

\subsection{Lam--Postnikov--Pylyavskyy inequality} \label{ss:intro-LPP}
\defng{Schur polynomials} \ts are symmetric polynomials which play a central
role in algebraic combinatorics, representation theory and other areas.
They are defined as irreducible characters of \ts $\GL(N,\cc)$, and their
stable limits (\defng{Schur functions}) form
an orthonormal basis in the ring of symmetric polynomials with respect to the Hall inner product.
We refer to \cite{Ful97,Mac95,Sag01} and  \cite[Ch.~7]{EC} for
standard introductions to the area, connections to other area
and numerous applications.

Let \ts $\cP$ \ts denote the set of integer partitions \ts $\la=(\la_1, \la_2, \ldots)$,
and let \ts $\lambda \vee \mu$, \ts $\la\wedge\mu$ \ts be the operations on~$\cP$ given
by the union and intersection of the corresponding Young diagrams.
Denote by \. $\fs_{\la}:= \fs_{\la}(\xb)$ \.  the \defnb{Schur function} \ts
indexed by \ts $\la=(\lambda_1,\lambda_2,\ldots)$, in variables  \ts $\xb:=(x_1,x_2,\ldots)$.
For symmetric functions \ts $\ff$ \ts and~$\fg$, we write \. $\ff \leqslant_{\fs} \fg$ \.
if \ts $(\fg-\ff)$ \ts is \defn{Schur positive}, i.e.,
a nonnegative linear combination of Schur functions.  {Schur positivity} \ts is a
fundamental property which remains largely mysterious in many cases, see e.g.\
\cite{Oko03,StanPos}; see also \cite{Pat19} for a brief introduction.

\smallskip

\begin{thm}[\textnormal{\defnb{LPP inequality}}~{\cite[Thm.~5]{LPP07}}]\label{thm:LPP}
For all \ts $\lambda,\mu \in \cP$, we have:
\begin{equation}\label{eq:LPP}\tag{\defn{\em LPP}}
	\fs_{\lambda} \. \fs_{\mu}  \ \leqslant_{\fs} \ \fs_{\lambda \vee \mu} \.\fs_{\lambda \wedge \mu}
\end{equation}
\end{thm}

\smallskip

This is a remarkable inequality and an important tool in the area, see e.g.\ \cite[$\S$8.2]{KT21} and \cite[$\S$4.7]{PPY19}.
The LPP inequality is a part of the ``log-concavity phenomena'' championed in recent years
by June Huh and his coauthors, see \cite{Huh18,Kal23}. Let us single out a closely related
log-concavity property of structure constants, see \cite{HMMS22,Oko03} and~$\S$\ref{ss:finrem-Oko}.

We refer to \cite{LP07} for a bijective proof of the \defng{monomial positivity} \ts
version of \eqref{eq:LPP}, a weaker property which still implies the
\defng{evaluation positivity} \. LHS~$\le$~RHS,
for all \ts $x_1,x_2,\ldots \in \rr_+\ts$.  See \cite{CP-multi},
for a multivariate extension of the AD inequality which also implies
the monomial positivity version of \eqref{eq:LPP},
and the extensive discussion of \eqref{eq:LPP}
in the context of poset inequalities (cf.~\cite[$\S$4.1]{CP-corr}).
See also Speyer's recent breakthrough \cite{Spe26}, where he resolves
a conjecture from \cite{DP07} by using the AD inequality \eqref{eq:AD},
thus giving a new proof of \eqref{eq:LPP} along the way.

In \cite{LPP07}, Lam--Postnikov--Pylyavskyy proved also the Okounkov,
FFLP, and LLT inequalities, conjectured in \cite{Oko97,FFLP,LLT}, respectively.
Of these, the Okounkov inequality gives log-concavity of LR coefficients and
will play an important role later in the paper (see~$\S$\ref{ss:app-back}).
Lam--Postnikov--Pylyavskyy also proved the following important generalization of \eqref{eq:LPP}
for skew Schur functions.

\smallskip

\begin{thm}[{\rm {\defnb{skew LPP inequality}}, \cite[Thm.~5]{LPP07}}{}]\label{thm:LPP-skew}
For all skew shapes \ts $\lambda/\mu$ \ts and \ts $\nu/\rho$, we have:
\begin{equation}\label{eq:LPP-skew}\tag{\defn{\em skew LPP}}
	\fs_{\lambda/\mu} \. \fs_{\nu/\rho}  \ \leqslant_{\fs} \
\fs_{(\lambda\vee \nu)/(\mu\vee \rho)} \. \fs_{(\lambda\wedge \nu)/(\mu\wedge \rho)}
\end{equation}
\end{thm}

\smallskip

Let us emphasize that Schur positivity in \eqref{eq:LPP-skew} is defined
in terms of nonnegative linear combinations of the ordinary Schur functions,
rather than skew Schur functions, and that the inequality coincides with
\eqref{eq:LPP} on straight shapes.

\smallskip

\subsection{Main results} \label{ss:intro-main}
Our first main result is a Schur positivity version of the AD inequality.
We call it the \defn{Ahlswede--Daykin--Schur {\rm (ADS)} inequality}.

\smallskip

\begin{thm}[{\rm \defna{ADS inequality}\ts}{}]\label{thm:ADS}
	 Let \. $\fa,\fb,\fc,\fd: \cP \to \rr_{\geq 0}$ \ts be functions satisfying
\begin{equation}\label{eq:AD-cond-first}\tag{{\defn{\em ADS-cond}}}
 	\fa({\lambda}) \. \fb({\mu}) \ \leq \  \fc({\lambda \vee \mu}) \. \fd({\lambda \wedge \mu})
 \qquad \text{ for all \. $\lambda,\mu \in \cP$.}
 \end{equation}
\nin
\ Then we have:
	 	\begin{equation}\label{eq:ADS}\tag{\defn{\em ADS}}
	\bigg( \sum_{\lambda \in \cP} \fa({\lambda}) \fs_{\lambda}\bigg)
		\bigg( \sum_{\lambda \in \cP} \fb({\lambda}) \fs_{\lambda}\bigg)
		 \ \leqslant_{\fs} \
		 	\bigg( \sum_{\lambda \in \cP} \fc({\lambda}) \fs_{\lambda}\bigg)
		 		\bigg( \sum_{\lambda \in \cP} \fd({\lambda}) \fs_{\lambda}\bigg)
	 \end{equation}
\end{thm}

\smallskip

Note that the assumptions \eqref{eq:AD-cond} and \eqref{eq:AD-cond-first} coincide
on partitions, but the conclusions are much stronger.
We note that the evaluation positivity
\ts $x_1,x_2,\ldots\ge 0$ \ts version of \eqref{eq:ADS} follows immediately
from \eqref{eq:AD}.  The monomial positivity version of \eqref{eq:ADS}
would follow from the result by the first and third author \cite[Thm.~6.1]{CP-multi}.
This is an intermediate notion between evaluation positivity and Schur positivity
(see~$\S$\ref{ss:notation-basic}).

One way to see the power of the ADS inequality, is to consider the case when functions
\ts $\fa,\fb,\fc,\fd$ \ts are indicators of \ts $\la,\mu,\la\vee \mu,\la\wedge\mu$,
respectively.  The AD inequality becomes trivial in this case, while the ADS inequality
coincides with the LPP inequality.  Thus, one can think of \eqref{eq:ADS} as an
advanced extension of \eqref{eq:LPP}.

\smallskip
We now turn to the skew version of the ADS inequality.  As with the skew LPP, it coincides with the ADS inequality
on straight shapes:

\smallskip

\begin{thm}[{\rm \defna{skew ADS inequality}\ts}{}]\label{thm:ADS-skew}
	Let \. $\fa,\fb,\fc,\fd: \cP \times \cP \to \rr_{\geq 0}$ \ts be functions satisfying
{\small
	\begin{equation}\label{eq:AD-cond-skew}\tag{\defn{\em skew ADS-cond}}
	 \fa(\lambda,\mu) \. \fb(\nu,\rho) \ \leq \ \fc(\lambda \vee \nu, \mu \vee \rho) \. \fd(\lambda \wedge \nu,\mu \wedge \rho)
\quad \text{ for all } \quad \lambda,\mu,\nu,\rho \in \cP.
	 \end{equation}
}
\nin
Then we have:\footnote{In the expansion of \eqref{eq:ADS-skew} in the Schur basis,
some coefficients can be infinite. In these cases, the relation \ts $\leqslant_{\fs}$ \ts means
that the infinite coefficient of $s_\lambda$ on the LHS, implies infinite
coefficient on the RHS.}
{\small
		\begin{equation}\label{eq:ADS-skew}\tag{\defn{\em skew ADS}}
		\bigg(\sum_{\lambda,\mu  \in \cP} \fa({\lambda},\mu) \fs_{\lambda/\mu} \bigg)  \bigg(\sum_{\lambda,\mu  \in \cP} \fb({\lambda},\mu) \fs_{\lambda/\mu} \bigg) \, \leqslant_{\fs} \, \bigg(\sum_{\lambda,\mu  \in \cP} \fc({\lambda},\mu) \fs_{\lambda/\mu} \bigg) \bigg(\sum_{\lambda,\mu  \in \cP} \fd({\lambda},\mu) \fs_{\lambda/\mu} \bigg).
	\end{equation}
}
\end{thm}

\smallskip

When \ts $\mu=\rho=\varnothing$, the skew ADS inequality \eqref{eq:ADS-skew}
becomes the (usual) ADS inequality \eqref{eq:ADS}.
Similarly, when taking indicator functions as above, we obtain
the skew LPP inequality \eqref{eq:LPP-skew}.  As we show in~$\S$\ref{ss:recover},
the skew ADS inequality  implies the AD inequality.
Curiously, Theorem~\ref{thm:ADS} is proved as a corollary of Theorem~\ref{thm:ADS-skew},
which in turn uses Theorem~\ref{thm:RLS} as lemma, thus making the whole
story even more complicated.

For the proof of Theorem~\ref{thm:ADS-skew}, we follow the blueprint in \cite{CP-multi}
for the multivariate AD inequalities.  Namely, we isolate the core mechanism of
the AD inequalities by reducing them to a set of refined \defng{orchestra inequalities}
\eqref{eq:Orch}.  As can be shown using \cite[Thm~6.1]{CP-multi},
the monomial positivity version of these inequalities follow from \eqref{eq:AD-cond}.
The main step in our approach is to show that orchestra inequalities extend to
\defnm{Schur orchestra inequalities} (Theorem~\ref{thm:orch}),
i.e., hold for Schur positivity.
This is done by a technical argument involving Temperley--Lieb immanants,
a technique that was used by Lam--Postnikov--Pylyavskyy~\cite{LPP07}
to establish Theorems~\ref{thm:LPP} and~\ref{thm:LPP-skew}.

\smallskip

\subsection{Stable Grothendieck polynomials}\label{ss:intro-Groth}
We now present our main application,
to the celebrated \defng{symmetric Grothendieck polynomials}.
These are nonhomogeneous symmetric polynomials introduced by
Lascoux and Sch\"utzenberger \cite{LS83,Las90}, building on
their earlier work on \defng{Schubert polynomials}.  Originally
defined as representatives of $K$-theory classes of structure
sheaves of Schubert varieties, their geometric and algebraic properties
and significance are well understood by now, see e.g.\ \cite{Buch02,B+08}.

The \defng{stable Grothendieck polynomials} \ts $\{\fG_{\lambda}\}$ \ts
are obtained in the stable limit of Grothendieck polynomials.\footnote{The
name is standard in the literature and somewhat misleading as these
are symmetric functions, and thus power series rather than polynomials.}
They form a basis in the ring of symmetric functions, and were introduced
by Fomin and Kirillov in an algebraic context \cite{FK94,FK96}.
We use the term \defna{Grothendieck polynomials} \ts when the context is clear.

It is hard to overstate the importance of Grothendieck polynomials
in modern algebraic combinatorics.  We refer to \cite{Buch05}
for a brief introduction, to \cite{Kir16} for an in-depth overview
of algebraic generalizations, and to \cite{BKTY05,TY09a} for
combinatorial formulas and algorithms.   We postpone the definitions
until~$\S$\ref{ss:notation-Groth}, and proceed to state the main application.

First, we modify Grothendieck polynomials to make them monomial positive:
\begin{align}\label{eq:tG}
\tG_{\lambda}(\bx) \ := \ (-1)^{|\lambda|} \. G_{\lambda}(-\bx)\ts.
\end{align}
Our main application is the following \defn{LPPG inequality}, originally conjectured by Mihalcea~\cite{Mih}.

\smallskip

\begin{thm}[{\rm former \defn{Mihalcea's conjecture}\ts}{}]\label{thm:Gro-LPP}
	For all partitions \ts $\lambda, \mu \in \cP$, we have:
	\begin{equation}\label{eq:Gro-LPP}
	\tG_{\lambda} \. \tG_{\mu} \, \leqslant_{\fs} \, \tG_{\lambda \vee \mu} \. \tG_{\lambda \wedge \mu}\ts.
	\end{equation}
\end{thm}

\smallskip

Taking the lowest degree terms in \eqref{eq:Gro-LPP} gives \eqref{eq:LPP}.
Let us emphasize that even the evaluation positivity was open until now.
Note also that we only prove the Schur positivity, while the
\defng{Grothendieck positivity} \ts remains open:

\smallskip

\begin{conj}[{\rm Thomas--Yong~\cite[Conj.~9.2]{TY09b}, see also \cite{Mih}}{}]\label{conj:TY}
	For all partitions \ts $\lambda, \mu \in \cP$, we have:
	\[
	\tG_{\lambda} \. \tG_{\mu} \, \leqslant_{\fg} \, \tG_{\lambda \vee \mu} \. \tG_{\lambda \wedge \mu}\ts.
	\]
\end{conj}

\smallskip

Here by \. $P\leqslant_{\fg} Q$ \. we mean that \ts $(Q-P)$ \ts is a nonnegative sum of
positive Grothendieck polynomials \ts $\tG_\nu\ts$.  This is a stronger property than
Schur positivity since \ts $\tG_{\lambda}$ \ts are Schur positive,
so one can think of  Theorem~\ref{thm:Gro-LPP} as an evidence in
favor of the conjecture.

We show in Section~\ref{s:Gro}, that Theorem~\ref{thm:Gro-LPP} is a
corollary of both Theorems~\ref{thm:ADS} and~\ref{thm:ADS-skew},
with functions \ts $\fa,\fb,\fc,\fd$ \ts carefully chosen.  In fact, we derive
Theorem~\ref{thm:Gro-LPP} from a skew shape generalization given in
Theorem~\ref{thm:Gro-LPP-skew}.  Along the way, we prove log-supermodularity
of the numbers \ts $g_{\la,\mu}$ \ts of certain increasing  tableaux:
$$\fg_{\lambda, \ts\mu}  \.\cdot \. \fg_{\nu,\ts\rho}   \ \leq \
\fg_{\lambda\vee \nu, \ts \mu \vee \rho} \.\cdot \.\fg_{\lambda\wedge \nu, \ts \mu \wedge \rho}\.,
$$
see~$\S$\ref{ss:notation-Groth} for the definition and Lemma~\ref{lem:fg} for the proof.
This is a result of independent interest, similar to \defng{Bj\"orner's log-supermodularity}
\ts of the numbers \ts $f^\la$ \ts of standard Young tableaux
\cite[Prop.~6.1]{Bjo11}, see also a skew version in \cite[Cor.~3.3]{CP-multi}.

\smallskip

\subsection{Dual stable Grothendieck polynomials}
Recall that Schur functions are self-dual w.r.t.\ the
\defng{Hall inner product} \. $\<s_\la,s_\mu\>=\de_{\la\mu}\ts$, while the  stable
Grothendieck polynomials are not.  The \defna{dual stable Grothendieck polynomials}
\ts $\{\fG_{\lambda}^\ast\}$ \ts form a basis in the ring of symmetric
functions~$\La$ that is orthogonal dual to the stable Grothendieck polynomials,
and has its own interesting combinatorial properties.

This family of symmetric functions was defined by Lam--Pylyavskyy in \cite{LP}
from a Hopf algebraic point of view.  We refer to \cite{Gal,Yel17} for further
background and combinatorial formulas.  We postpone the definitions
until~$\S$\ref{ss:notation-dual}.  For now, we state a dual version of
Theorem~\ref{thm:Gro-LPP}.

\smallskip

\begin{thm}\label{thm:Gro-LPP-dual}
	For all partitions \ts $\lambda, \mu \in \cP$, we have:
	\[
	\fG_{\lambda}^* \. \fG_{\mu}^* \ \leqslant_{\fs} \ \fG_{\lambda \vee \mu}^* \. \fG_{\lambda \wedge \mu}^*.
	\]
\end{thm}

\smallskip

Again, note that we only prove Schur positivity.  The proof
is given in Section~\ref{s:dual}.  Along the way, we prove
log-supermodularity of the numbers of certain \defng{elegant fillings}
(Lemma~\ref{lem:fg*}).

\smallskip

\subsection{Shadow Schur functions}\label{ss:intro-shadow}
Fundamentally, the reason our skew ADS inequality (Theorem~\ref{thm:ADS-skew})
is applicable to the stable and dual stable Grothendieck polynomials, is because
there are well known combinatorial expansions of these families
in the basis of Schur functions.  Unfortunately, proofs of both
Theorem~\ref{thm:Gro-LPP} and~\ref{thm:Gro-LPP-dual} are somewhat
technical, obscuring the underlying philosophy.

To clarify these applications, we introduce a new family of
symmetric functions, where the applications of the
skew ADS inequality is straightforward.
Fix integers \ts $k\ge \ell\ge 1$ \ts and denote by \ts $\cP^\ell$ \ts
the set of partitions with at most~$\ell$ parts.
Define a \defn{shadow Schur function} \ts as a sum of Schur functions
of partitions with at most~$k$ rows, of which the first~$\ell$ rows are fixed:
\begin{equation}\label{eq:S_lambda}
	 \fS_{\lambda}^{(k,\ell)} \, := \,  \sum_{\mu \ts\in\ts \Sh(\la;k,\ell)} \. \fs_{\mu}\,, \quad \text{where} \quad \la\in \cP^\ell, \ \
\Sh(\la;k,\ell) \, := \, \{\mu \in \cP^k \. : \. \la_1=\mu_1,\ldots,\la_\ell=\mu_\ell\}.
\end{equation}
This family contains the (usual) Schur functions \ts for $k=\ell$, and is interesting in its own right.

In Section~\ref{s:app}, we state and prove both the usual and skew versions of log-supermodularity for shadow Schur functions, modeled after Theorems~\ref{thm:Gro-LPP} and~\ref{thm:Gro-LPP-dual} above.  We also present a roadmap for generalizations in this setting.

\smallskip

\subsection{Paper structure}\label{ss:intro-outline}
We start with standard definition and notation in algebraic combinatorics
(Section~\ref{s:def}), including definitions of both families of Grothendieck
polynomials discussed above.  In Section~\ref{s:app} we give applications of
the AD and ADS inequalities to shadow Schur functions.  This section can
be viewed as an accessible example illustrating the power of our results.

In a brief Section~\ref{s:under} we first relate the AD and ADS inequalities.
We then show the subtleties of the proofs of the ADS inequality and its
relatives, by giving a counterexample to the naive approach to \eqref{eq:ADS}.
This is another preliminary section aiming to help the reader absorb the results.
This discussion motivates the
\defnm{Schur duet inequalities} (Corollary~\ref{cor:duet-Schur}),
a special case of the \defnm{Schur orchestra inequalities} (Theorem~\ref{thm:orch}).
These are key technical results introduced in Section~\ref{s:Orch};  we need them
in lieu of the induction used in the standard proof of AD inequality.

In Section~\ref{s:TL} we present the algebraic approach to Schur positivity,
using the technology of \defnm{Temperley--Lieb immanants}.  These are used
in Section~\ref{s:orch-proof} to prove the Schur orchestra inequalities.
Then, in Section~\ref{s:ADS}, we prove both the ADS inequality (Theorem~\ref{thm:ADS})
and the skew ADS inequality (Theorem~\ref{thm:ADS}).

In Section~\ref{s:Gro}, we establish log-supermodularity and log-concavity
for Grothendieck polynomials (Theorem~\ref{thm:Gro-LPP} and Corollary~\ref{cor:Gro-LC}),
using the skew ADS inequality.  We also prove the \defnm{ADG inequality},
a generalization of the ADS inequality  to Grothendieck polynomials (Theorem~\ref{thm:ADS-Groth}).
This is followed by a closely related Section~\ref{s:dual} where we prove versions
of these inequalities for the dual stable Grothendieck polynomials
(Theorem~\ref{thm:Gro-LPP-dual} and Corollary~\ref{cor:Gro-LC-dual}).

In Section~\ref{s:further}, we outline four different directions in which
our results can be generalized, including an extension of ADS inequality
to \emph{supersymmetric functions} (Theorem~\ref{thm:ADS-Super}).
We also present two interesting open problems
(Conjectures~\ref{conj:SHKS} and~\ref{conj:SLC}).
We conclude with final remarks in Section~\ref{s:finrem}.
\medskip

%\newpage

%%%%%%%%%%%%%%%%%%%%%%%%%%%%%%%%%%%%%%%%%%%%%%%%%%%%%%%%%%%%%%%%%%%%%%%%%%%%%%%%%%%%%%
%%%%%%%%%%%%%%%%%%%%%%%%%%%%%%%%%%%%%%%%%%%%%%%%%%%%%%%%%%%%%%%%%%%%%%%%%%%%%%%%%%%%%%
%%%%%%%%%%%%%%%%%%%%%%%%%%%%%%%%%%%%%%%%%%%%%%%%%%%%%%%%%%%%%%%%%%%%%%%%%%%%%%%%%%%%%%
%%%%%%%%%%%%%%%%%%%%%%%%%%%%%%%%%%%%%%%%%%%%%%%%%%%%%%%%%%%%%%%%%%%%%%%%%%%%%%%%%%%%%%
%%%%%%%%%%%%%%%%%%%%%%%%%%%%%%%%%%%%%%%%%%%%%%%%%%%%%%%%%%%%%%%%%%%%%%%%%%%%%%%%%%%%%%
%%%%%%%%%%%%%%%%%%%%%%%%%%%%%%%%%%%%%%%%%%%%%%%%%%%%%%%%%%%%%%%%%%%%%%%%%%%%%%%%%%%%%%
%%%%%%%%%%%%%%%%%%%%%%%%%%%%%%%%%%%%%%%%%%%%%%%%%%%%%%%%%%%%%%%%%%%%%%%%%%%%%%%%%%%%%%

%\newpage

\section{Definitions and notation}\label{s:def}

\subsection{Basic notations}\label{ss:notation-basic}
We use \ts $\nn=\{0,1,2,\ldots\}$ \ts and \ts $\nn_{\ge 1} = \{1,2,\ldots\}$
\ts to denote the sets of nonnegative and positive integers.  Let
\ts $[n]=\{1,2,\ldots,n\}$ \ts and \ts $\rr_+ = \{x\ge 0\}$.
We use \ts $2^X$ \ts to denote the set of subsets of~$X$.
The indicator function of a subset \ts $A$ \ts is denoted~$\mathbf{1}_A\ts$.

For an inequality \ts $a \ge b$, the difference \ts $(a-b)$ \ts
is called the \defn{defect}.  For polynomials \ts $f,g \in \rr[x_1,x_2,\ldots]$,
we write \ts $f\geqslant g$ \ts if the defect \ts $(f-g)\ge 0$ \ts for all \ts
$x_1,x_2,\ldots \ge 0$.  We write \ts $f\geqslant_{\fm} g$ \ts if \ts
$(f-g)\in \rr_+[x_1,x_2,\ldots]$.  These properties are called the \defn{evaluation}
and the \defn{monomial positivity}, respectively.

\subsection{Partitions}\label{ss:notation-part}
Throughout the paper, a \defnb{partition} \ts is an infinite nonincreasing sequence
\ts $\la=(\la_1, \la_2, \ldots)$ \ts with bounded support.  We use \ts $\ell(\la)$ \ts
to denote the support size, i.e., the number of nonzero parts in~$\la$.
Let \ts $\cP$ \ts denote the set of partitions, and let \ts $\cP^\ell$ \ts
denote the set of partitions with at most \ts $\ell$ \ts nonzero parts.
These are not to be confused with \defn{set partitions} \. $X= B_1\sqcup B_2 \sqcup \ldots$

When writing partitions, we omit the infinite
tail of zeros, and write only the prefix with the support of the sequence,
so e.g.\ \ts $(4,3,1,0,0,\ldots)$ \ts is written as \ts $(4,3,1)$.
For convenience, we use both the sequence and word notation,
so for example \ts $(4,3,1)$ \ts and \ts $431$ \ts correspond to the same
partition.

Denote by \ts $|\la|:=\la_1+\la_2+\ldots$ \ts the \defn{size} \ts of the partition~$\la$.
A \defn{conjugate partition} \ts $\la'=(\la_1',\la_2',\ldots)$ \ts is defined by \.
$\la_i':=|\{j\.: \. \la_j\ge i\}|$.
As in the introduction, let
\[
\lambda \vee \mu \ := \  \big(\max\{\lambda_1,\mu_1\},\max\{\lambda_2,\mu_2\}, \ldots\big), \qquad   \lambda \wedge \mu \ := \  \big(\min\{\lambda_1,\mu_1\},\min\{\lambda_2,\mu_2\},\ldots\big).
\]
Clearly, these operations coincide with union and intersection of the corresponding
Young diagram, as defined in the introduction.  Note that under these operations,
\ts $(\cP,\vee,\wedge)$ \ts forms an infinite distributive lattice.

Next, we define two operations on all partitions $\la,\mu\in \cp$:
$$
\big\lceil \tfrac{\la+\mu}{2}\big\rceil \ := \ \big(\big\lceil \tfrac{\la_1+\mu_1}{2}\big\rceil,
\big\lceil \tfrac{\la_2+\mu_2}{2}\big\rceil, \ldots\big), \qquad
\big\lfloor \tfrac{\la+\mu}{2}\big\rfloor \ := \ \big(\big\lfloor \tfrac{\la_1+\mu_1}{2}\big\rfloor,
\big\lfloor \tfrac{\la_2+\mu_2}{2}\big\rfloor, \ldots\big).
$$
Clearly, when all parts \ts $(\la_i+\mu_i)$ \ts are even, these two partitions coincide.

Finally, for partitions \ts $\lambda,\mu\in \cp^\ell$,
let \ts $\nu \in \cP^{2\ell}$ \ts be a partition obtained by sorting the
combined parts of $\lambda$ and $\mu$ in weakly decreasing order.
We define:
$$
\sort_1(\lambda,\mu) \, := \, (\nu_1,\nu_3,\ldots), \qquad
\sort_2(\lambda,\mu) \, := \, (\nu_2,\nu_4,\ldots).
$$
Note also that this sorting operation is conjugate to the averaging operation
above.

\subsection{Young diagrams and Young tableaux} \label{ss:notation-SYT}
We refer to \cite{Mac95} and  \cite[Ch.~7]{EC} for standard definitions and notation
in algebraic combinatorics.
A \defn{Young diagram} \ts corresponding to partition \ts $\la=(\la_1,\la_2,\ldots)$ \ts
is the set of squares
$$
\big\{(i,j)\in \nn^2 \. : \. 1\le j \leq  \la_i, \ts 1\le i \le \ell(\la)\big\}.
$$
By a mild abuse of notation, we use \ts $\la$ \ts to also
denote the corresponding Young diagram, and refer to it as the \defn{straight shape}.
Note that the Young diagram of a conjugate partition $\la'$ is obtained by the \ts
$x\lra y$ \ts reflection.
We use $\emp$ to denote empty Young diagram corresponding to the $\mathbf{0}$ sequence.

Let \ts $\mu=(\mu_1,\mu_2,\ldots)$ \ts be a partition such that \.
$\mu_i \le \la_i$ \. for all \ts $i\ge 1$.  The difference of Young diagrams
is denoted by \. $\la/\mu$ \. and called the \defn{skew Young diagram} \ts
of shape \ts $\la/\mu$, or simply the \defn{skew shape}~$\la/\mu$.
We use \ts $|\la/\mu|:=|\la|-|\mu|$ \ts to denote the number of squares in~$\la/\mu$.

Let \ts $A: \la/\mu \to \nn_{\ge 1}$ \ts be a function which increases in rows and strictly
increases in columns.  We think of \ts $A$ \ts as a Young tableau with integers written
in squares of~$\la/\mu$.  Such~$A$ is called a \defn{semistandard Young tableau}.
The set of such tableaux is denoted \ts $\SSYT(\la/\mu)$.  We use \ts $\SSYT(\la/\mu,t)$ \ts
to denote semistandard Young tableaux with entries $\le t$.

The \defn{weight} \ts of a tableau \ts $A\in \SSYT(\la/\mu)$ \ts is a sequence \.
$\big(m_1(A),m_2(A),\ldots\big)$, where \ts $m_i(A):=|A^{-1}(i)|$ \ts is the number
of $i$'s in~$A$.  \defn{Kostka number} \ts $K_{\la,\mu}$ \ts is the number of
$A\in \SSYT(\la)$ \ts of weight~$\mu$.

\subsection{Schur functions} \label{ss:notation-Schur}
\defn{Skew Schur polynomial} \ts is a symmetric polynomial associated with the skew shape \ts $\la/\mu$ \ts
and can be defined as
\begin{equation}\label{eq:Schur-def}
\fs_{\la/\mu}(z_1,\ldots,z_n) \, = \, \sum_{A \ts \in \ts \SSYT(\la/\mu, \ts n)} \. z_1^{m_1(A)} \ts \cdots \. z_n^{m_n(A)}\..
\end{equation}
They form a linear basis
in the space $\La_n$ of all symmetric polynomials in~$n$ variables.  Here we include
only combinatorial definitions, rather than the (more standard) determinantal definition.

Similarly,
\defn{skew Schur functions} \ts are defined as
\begin{equation}\label{eq:Schur-def-inf}
\fs_{\la/\mu}(z_1,z_2,\ldots) \, = \, \sum_{A \ts \in \ts \SSYT(\la/\mu)} \. z_1^{m_1(A)} \. z_2^{m_2(A)} \ts \cdots
\end{equation}
They are the stable limits of Schur polynomials as \ts $n\to \infty$. Note that Schur functions
\ts $\fs_{\la/\mu}$ \ts are homogenous of degree \ts $|\la/\mu|$.

Recall that Schur functions \ts $\fs_\la$ \ts corresponding to straight shape \ts $\mu=\emp$ \ts
form a linear basis in the \defn{ring of symmetric functions} \.
$\La=\varprojlim \La_n$\ts, where \. $\La_n=\cc[x_1,\ldots,x_n]^{S_n}$.
It is convenient  to extend the definition of functions \ts $\fs_\lambda$ \ts
to arbitrary infinite integer sequences \ts $\lambda \in \nn^\infty$.
We adopt the convention that \. $\fs_{\lambda}=0$ \. when $\lambda$ is not a partition.

The \defn{Littlewood--Richardson coefficients} \ts are the structure constants
of multiplication of Schur functions:
$$
\fs_\mu \. \fs_\nu \ = \ \sum_{\la\in \cP} \. c^\la_{\mu,\nu} \. \fs_\la\ts.
$$
Taking the degrees shows that \ts $c^\la_{\mu,\nu}=0$ \ts unless \ts $|\la|=|\mu|\ts+\ts |\nu|$.

\subsection{Grothendieck polynomials} \label{ss:notation-Groth}
An \defn{increasing tableau} of skew shape $\lambda/\mu$ is a function
\ts $A: \lambda/\mu \to \nn_{\ge 1}$ \ts which is strictly increasing in
rows and down columns, and such that every entry in the $i$-th row is
restricted to $\{1, 2, \ldots, i-1\}$.  Denote by \ts $\IT(\la/\mu)$ \ts
the set of such tableaux, and let \ts $\fg_{\mu,\lambda}:=|\IT(\la/\mu)|$ \ts be
the number of such tableaux (note the order reversal in the index).

Denote by \ts $\widehat{\lambda}$ \ts the unique maximal partition satisfying \ts
$\widehat{\lambda}_i \le \lambda_i + i - 1$ \ts for all \ts $i \ge 1$.
Denote \ts
$\cR(\la):=\big\{\mu \.:\. \lambda \subseteq \mu \subseteq \widehat{\lambda}\big\}$.
For a partition \ts $\lambda \in \cP$, the \ts \defnb{stable Grothendieck polynomial} \ts is defined as
\begin{equation}\label{eq:fG}
\fG_{\lambda}(\bx) \, := \,  \sum_{\mu \in \cR(\la)} (-1)^{|\mu|-|\lambda|} \. \fg_{\lambda,\mu} \ts \fs_\mu(\bx)\ts.
\end{equation}
These functions form a basis in the ring of symmetric functions~$\La$.
Note that just like Schur functions, the stable Grothendieck polynomials are formal power series in
variables \ts $\bx=(x_1,x_2,\ldots)$.

There are several other equivalent definitions of $\fG_{\lambda}$, and the definition
presented here is due to Lenart \cite[Thm.~2.8]{Len}.  For an alternative combinatorial
formula for \ts $\fG_{\lambda}$ \ts in terms of set-valued tableaux,  see \cite[$\S$3]{Buch02}
and  \cite[Cor.~3.11]{MPS21}.  Yet another combinatorial definition in terms of standardized
increasing tableaux is given in \cite{McN06} in a more general context of \defng{factorial
Grothendieck polynomials}.

As in \eqref{eq:tG}, it will be convenient for us to define the following positive
version of stable Grothendieck polynomials:
\begin{align}\label{eq:tG-later}
\tG_{\lambda}(\bx) \, :=  \, (-1)^{|\lambda|}G_{\lambda}(-\bx) \, = \,  \sum_{\mu \in \cR(\la)} \. \fg_{\lambda,\mu} \ts \fs_\mu(\bx)\ts.	
\end{align}

\subsection{Dual stable Grothendieck polynomials} \label{ss:notation-dual}
An \defnb{elegant filling} \ts of the skew shape \ts $\lambda/\mu$ \ts
is a semistandard Young tableau \ts $A \in \SSYT(\lambda/\mu)$ \ts
subject to the condition that every entry in row $i$ belongs to the set
$\{1,2,\ldots, i-1\}$, for all $i \geq 1$.  We use \ts $\EF(\la/\mu)$ \ts
to denote the set of elegant fillings, and let \ts
$\fg^*_{\lambda,\mu}:=|\EF(\la/\mu)|$ \ts denotes the number of elegant
fillings of the skew shape~$\lambda/\mu$.

For partition $\lambda \in \cP$,
the \defnb{dual stable Grothendieck polynomial} \ts is defined as
\[
\fG^{*}_{\lambda}(\bx) \ := \  \sum_{\mu \subseteq \lambda} \. \fg^*_{\lambda,\mu} \. \fs_{\mu}(\bx).
\]
These functions also form a basis in the ring of symmetric functions~$\La(\bx)$.

\medskip

%%%%%%%%%%%%%%%%%%%%%%%%%%%%%%%%%%%%%%%%%%%%%%%%%%%%%%%%%%%%%%%%%%%%%%%%%%%%%%%%%%%%%%
%%%%%%%%%%%%%%%%%%%%%%%%%%%%%%%%%%%%%%%%%%%%%%%%%%%%%%%%%%%%%%%%%%%%%%%%%%%%%%%%%%%%%%
%%%%%%%%%%%%%%%%%%%%%%%%%%%%%%%%%%%%%%%%%%%%%%%%%%%%%%%%%%%%%%%%%%%%%%%%%%%%%%%%%%%%%%
%%%%%%%%%%%%%%%%%%%%%%%%%%%%%%%%%%%%%%%%%%%%%%%%%%%%%%%%%%%%%%%%%%%%%%%%%%%%%%%%%%%%%%
%%%%%%%%%%%%%%%%%%%%%%%%%%%%%%%%%%%%%%%%%%%%%%%%%%%%%%%%%%%%%%%%%%%%%%%%%%%%%%%%%%%%%%
%%%%%%%%%%%%%%%%%%%%%%%%%%%%%%%%%%%%%%%%%%%%%%%%%%%%%%%%%%%%%%%%%%%%%%%%%%%%%%%%%%%%%%
%%%%%%%%%%%%%%%%%%%%%%%%%%%%%%%%%%%%%%%%%%%%%%%%%%%%%%%%%%%%%%%%%%%%%%%%%%%%%%%%%%%%%%

%\newpage

\section{Shadow Schur functions} \label{s:app}

In this section we introduce a new family of \defng{shadow Schur functions} \ts
which generalize the ordinary Schur functions.  We then present extensions
of LPP and skew LPP inequalities, which follow easily from the ADS inequalities.
We then state without proof several other related inequalities for shadow Schur
functions.

\smallskip

\subsection{Shadow LPP inequality}\label{ss:app-LPP}
Fix integers \ts $k\ge \ell\ge 1$ \ts and note that \ts $\Sh(\la;k,\ell)=\{\la\}$ \ts
when \ts $k=\ell$, or when \ts $k>\ell$ \ts and \ts $\la_\ell=0$.  Recall the definition
\eqref{eq:S_lambda} of shadow Schur functions given in the introduction.

\smallskip

\begin{thm}[{\rm \defn{shadow LPP inequality}\ts}{}]\label{thm:LPP-global}
For all \. $k \ge \ell\ge 1$ \ts and \ts  $\lambda,\mu \in \cP^\ell$, we have:
\begin{equation}\label{eq:sh-LPP}\tag{sLPP}
\fS_{\lambda}^{(k,\ell)} \. \fS_{\mu}^{(k,\ell)} \ \leqslant_{\fs} \
\fS_{\lambda \vee \mu}^{(k,\ell)} \.\fS_{\lambda \wedge \mu}^{(k,\ell)}.
\end{equation}
\end{thm}

\smallskip
Note that the original LPP inequality \eqref{eq:LPP} is a special case of
\eqref{eq:sh-LPP}, where $k=\ell$.
We define \defn{skew shadow Schur functions} \ts as follows:
\begin{equation}\label{eq:S_skew}
	\fS^{(k,\ell)}_{\lambda/\mu} \ := \ \sum_{\nu \ts \in \ts \Sh(\la;k,\ell), \. \rho \in \Sh(\mu;k,\ell)} \fs_{\nu/\rho}\.,
\end{equation}
the sum  over all skew shapes whose first $\ell$ row lengths are specified by $\lambda$ and $\mu$.
The following theorem extends \eqref{eq:sh-LPP} to the skew setting:

\smallskip

\begin{thm}[{\rm \defn{shadow skew LPP inequality}\ts}{}]\label{thm:LPP-global-skew}
	For all \, $\lambda,\mu, \nu,\rho \in \cP^\ell$, we have:
	\[  \fS_{\lambda/\mu}^{(k,\ell)} \. \fS_{\nu/\rho}^{(k,\ell)} \ \leqslant_{\fs} \  \fS_{(\lambda \vee \nu)/(\mu\vee \rho)}^{(k,\ell)} \. \fS_{(\lambda \wedge \nu)/(\mu\wedge \rho)}^{(k,\ell)}.   \]
\end{thm}

\smallskip

\begin{proof}
Let \. $\fa,\fb,\fc,\fd: \cP^k \times \cP^k \to \rr_{\geq 0}$ \.
be given by
\begin{align*}
	\fa(\alpha,\beta) \ &:= \  \mathbf{1}_{\{\al_i =\lambda_i, \. \be_i =\mu_i \ \text{ for all } i \in [\ell]  \}}\\
	\fb(\alpha,\beta) \ &:= \  \mathbf{1}_{\{\al_i =\nu_i, \. \be_i =\rho_i \ \text{ for all } i \in [\ell]  \}}\\
	\fc(\alpha,\beta) \ &:= \  \mathbf{1}_{\{\al_i =\max\{\lambda_i,\nu_i\}, \. \be_i =\max\{\mu_i,\rho_i\} \ \text{ for all } i \in [\ell]  \}}\\
	\fd(\alpha,\beta) \ &:= \  \mathbf{1}_{\{\al_i =\min\{\lambda_i,\nu_i\}, \. \be_i =\min\{\mu_i,\rho_i\} \ \text{ for all } i \in [\ell]  \}}
\end{align*}
It is straightforward to verify that these functions satisfy \eqref{eq:AD-cond-skew}, and that
\begin{alignat*}{2}
&  \fS^{(k,\ell)}_{\la/\mu} \ = \ 	\sum_{\alpha,\beta\in \cP^k} \fa(\alpha,\beta) \fs_{\alpha/\beta}\,, \quad
&& \fS^{(k,\ell)}_{\nu/\rho} \ = \ \sum_{\alpha,\beta\in \cP^k} \fb(\alpha,\beta) \fs_{\alpha/\beta}\,,\\
& \fS^{(k,\ell)}_{(\la\vee \nu)/(\mu \vee \rho)} \ = \ \sum_{\alpha,\beta\in \cP^k} \fc(\alpha,\beta) \fs_{\alpha/\beta}\,, \quad
&& \fS^{(k,\ell)}_{(\la\wedge \nu)/(\mu \wedge \rho)} \ = \ \sum_{\alpha,\beta\in \cP^k} \fd(\alpha,\beta) \fs_{\alpha/\beta}\,.
\end{alignat*}
The  theorem now follows from the skew ADS inequality (Theorem~\ref{thm:ADS-skew}).
\end{proof}

\smallskip
\subsection{Further Schur positive inequalities} \label{ss:app-back}
The following two inequalities were conjectured by Okounkov \cite[p.~269]{Oko01}
and Fomin--Fulton--Li--Poon \cite[Conj.~2.7]{FFLP}.  They were also proved
by Lam--Postnikov--Pylyavskyy \cite[Thm.~4]{LPP07}.

\smallskip

\begin{thm}[{\rm \defn{Okounkov inequality}\ts}{}]\label{thm:Oko}
For all \ts  $\lambda,\mu \in \cP$, we have:
\begin{equation}\label{eq:Oko}\tag{Ok}
\fs_{\lambda} \. \fs_{\mu} \ \leqslant_{\fs} \   \fs_{\left \lceil (\lambda+\mu)/2 \right\rceil}
\. \fs_{\left \lfloor (\lambda+\mu)/2 \right\rfloor}.
\end{equation}
\end{thm}

\smallskip

\begin{thm}[{\rm \defn{FFLP inequality}\ts}{}]\label{thm:FFLP}
For all \ts  $\lambda,\mu \in \cP$, we have:
\begin{equation}\label{eq:FFLP}\tag{FFLP}
\fs_{\lambda} \. \fs_{\mu} \ \leqslant_{\fs} \   \fs_{\sort_1(\lambda,\mu)}
\. \fs_{\sort_2(\lambda,\mu)}.
\end{equation}
\end{thm}

\smallskip

Lam--Postnikov--Pylyavskyy proved further Schur positive inequalities
which fall outside the scope of this paper.  They also stated the
\defng{LPP conjecture}, see \cite[Conj.~1]{DP07}, which was
recently proved by Speyer \cite{Spe26}.  Somewhat surprisingly,
Speyer observed that his proof implies a general statement about
Schur positivity:

\smallskip

\begin{thm}[{\rm Speyer \cite[Rem.~4.19]{Spe26}}{}]\label{t:speyer}
Let \ts $f: \cP^\ell \to \La$ \ts such that \ts $f(\lambda+\mathbf{1})=f(\lambda)$ \ts for all \ts $\lambda \in \cP^\ell$,
and
$$
f(\la) \ts f(\mu) \, \leqslant_{\fs} \, f(\la\vee \mu) \ts f(\la\wedge \mu) \quad \forall \. \la,\mu \in \cp^\ell.
$$
Then:
$$
f(\la) \ts f(\mu) \, \leqslant_{\fs} \, f\big(\big\lceil\tfrac{\la+ \mu}{2}\big\rceil\big) \.
f\big(\big\lfloor\tfrac{\la+ \mu}{2}\big\rfloor\big) \quad \forall \. \la,\mu \in \cp^\ell.
$$
\end{thm}

\smallskip

In particular, Speyer's theorem shows that the LPP inequality \eqref{eq:LPP}
implies both  \eqref{eq:Oko}  and, after conjugation, the inequality \eqref{eq:FFLP}.
We conclude with a skew version of \eqref{eq:Oko} which was also proved by
Lam--Postnikov--Pylyavskyy.

\smallskip

\begin{thm}[{\rm \defn{skew Okounkov inequality}~\cite[Thm~11]{LPP07}}{}]\label{thm:Oko-skew}
	For all \ts $\lambda,\mu, \nu,\rho \in \cP$, we have:
	\[  \fs_{\lambda/\mu} \. \fs_{\nu/\rho} \ \leqslant_{\fs} \
	\fs_{\left \lceil (\lambda+\nu)/2 \right\rceil/  \left \lceil (\mu+\rho)/2 \right\rceil}
	\. \fs_{\left \lfloor (\lambda+\nu)/2 \right\rfloor /
		\left \lfloor (\mu+\rho)/2 \right\rfloor
	}.   \]
\end{thm}

\smallskip

\subsection{New inequalities for shadow Schur functions} \label{ss:app-further-shadow}
As a consequence  of Theorem~\ref{thm:LPP-global-skew}, we obtain the shadow generalization of
the skew Okounkov inequality:

\smallskip

\begin{thm}[{\rm \defn{shadow skew Okounkov inequality}\ts}{}]\label{thm:Oko-skew-global}
	For all \ts $k \ge \ell \ge 1$ \ts and \ts $\lambda,\mu, \nu,\rho \in \cP^\ell$, we have:
	\[  \fS^{(k,\ell)}_{\lambda/\mu} \. \fS_{\nu/\rho}^{(k,\ell)} \ \leqslant_{\fs} \
	\fS^{(k,\ell)}_{\left \lceil (\lambda+\nu)/2 \right\rceil/  \left \lceil (\mu+\rho)/2 \right\rceil}
	\. \fS^{(k,\ell)}_{\left \lfloor (\lambda+\nu)/2 \right\rfloor /
		\left \lfloor (\mu+\rho)/2 \right\rfloor
	}.   \]
\end{thm}

\smallskip

We state Theorem~\ref{thm:Oko-skew-global} without proof, since it closely
parallels that of Theorem~\ref{thm:Oko-skew} given in \cite{LPP07},
substituting the shadow skew LPP inequality (Theorem~\ref{thm:LPP-global-skew})
for the original version (Theorem~\ref{thm:LPP-skew}).
Note that this theorem does not follow directly from  Speyer's Theorem~\ref{t:speyer},
because $S^{(k,\ell)}_{\lambda/\mu}$ fails to satisfy the first condition of Speyer's theorem.
In particular, we obtain the following straight version as corollary.

\smallskip

\begin{cor}[{\rm \defn{shadow Okounkov inequality}\ts}{}]\label{cor:sh-Oko}
For all \ts $k \ge \ell \ge 1$ \ts and
\ts  $\lambda,\mu \in \cP^\ell$, we have:
\[ \fS_{\lambda}^{(k,\ell)} \. \fS^{(k,\ell)}_{\mu} \ \leqslant_{\fs} \
\fS^{(k,\ell)}_{\left \lceil (\lambda+\mu)/2 \right\rceil}
\. \fS^{(k,\ell)}_{\left \lfloor (\lambda+\mu)/2 \right\rfloor}.
\]
\end{cor}

\smallskip

Let us also mention that the FFLP inequality also holds for shadow Schur functions:

\smallskip

\begin{thm}[{\rm \defn{shadow FFLP inequality}\ts}{}]\label{thm:FFLP-global}
For all \ts $k \ge \ell \ge 1$ \ts and
\ts  $\lambda,\mu \in \cP^\ell$, such that \ts $\lambda_\ell=\mu_\ell$, we have:
\[
\fS_{\lambda}^{(k,\ell)} \.  \fS_{\mu}^{(k,\ell)}  \ \leqslant_{\fs} \   \fS_{\sort_1(\lambda,\mu)}^{(k,\ell)} \.
\fS_{\sort_2(\lambda,\mu)}^{(k,\ell)}.
\]
\end{thm}

\smallskip

By the argument as above, this inequality extends \eqref{eq:FFLP}.
We state this result without proof, since it closely parallels that
of a skew version of \eqref{eq:FFLP} given in \cite[Cor.~12]{LPP07}.
Here we substitute the shadow Okounkov inequality (Corollary~\ref{cor:sh-Oko})
for the original version (Theorem~\ref{thm:Oko}).
\smallskip

\begin{rem}
We note that the condition $\lambda_\ell=\mu_\ell$ in Theorem~\ref{thm:FFLP-global},
which is absent in \eqref{eq:FFLP}, is necessary for this generalization.
To see this, let \ts $k=3$, \ts $\ell=2$, \ts $\lambda=(2,2)$, and \ts $\mu=(1,1)$.
In this case, $\sort_1(\lambda,\mu)=\sort_2(\lambda,\mu)=(2,1)$, yet
\[
\fS^{(3,2)}_{22} \fS^{(3,2)}_{11}  \not\leqslant_{\fs} \fS_{21}^{(3,2)} \fS_{2,1}^{(3,2)}.
\]
\end{rem}

\smallskip

\subsection{Lower Schur functions}\label{ss:app-lower}
For a partition \ts $\lambda\in\cP$, define \defn{lower Schur functions}
\begin{align*}
\fS_{\subseteq \ts\lambda} \ := \ \sum_{\mu \in \cP, \, \mu \subseteq \lambda} \. \fs_{\mu}\ts.
\end{align*}
It is easy to see that all results in this section continue to hold when
shadow Schur functions are replaced with lower Schur functions.  The proofs are nearly identical;
we omit the details.

Note that lower Schur functions are perhaps more natural than shadow Schur functions.
However, in contrast with the latter, the former do not contain the ordinary Schur
functions as a special case.  On the other hand, the corresponding log-supermodularity,
i.e.\ version of Theorem~\ref{thm:LPP-global} in this case, does again generalize the LPP inequality by taking the maximal degree terms.

\medskip
%\newpage

%%%%%%%%%%%%%%%%%%%%%%%%%%%%%%%%%%%%%%%%%%%%%%%%%%%%%%%%%%%%%%%%%%%%%%%%%%%
%%%%%%%%%%%%%%%%%%%%%%%%%%%%%%%%%%%%%%%%%%%%%%%%%%%%%%%%%%%%%%%%%%%%%%%%%%%
%%%%%%%%%%%%%%%%%%%%%%%%%%%%%%%%%%%%%%%%%%%%%%%%%%%%%%%%%%%%%%%%%%%%%%%%%%%
%%%%%%%%%%%%%%%%%%%%%%%%%%%%%%%%%%%%%%%%%%%%%%%%%%%%%%%%%%%%%%%%%%%%%%%%%%%
%%%%%%%%%%%%%%%%%%%%%%%%%%%%%%%%%%%%%%%%%%%%%%%%%%%%%%%%%%%%%%%%%%%%%%%%%%%
%%%%%%%%%%%%%%%%%%%%%%%%%%%%%%%%%%%%%%%%%%%%%%%%%%%%%%%%%%%%%%%%%%%%%%%%%%%
%%%%%%%%%%%%%%%%%%%%%%%%%%%%%%%%%%%%%%%%%%%%%%%%%%%%%%%%%%%%%%%%%%%%%%%%%%%

\section{Understanding the ADS inequality} \label{s:under}

\subsection{The role of Schur functions}\label{ss:under-subtle}
It is natural to ask if Theorem~\ref{thm:ADS} can be strengthened to the
following formal algebraic statement which omits Schur functions altogether.

Let \ts $\cA = \rr\<\ts\ff_\lambda \.: \. \lambda \in \cP\ts\>$ \ts be a commutative
ring. Suppose that there exists a partial order  \ts ``$\preccurlyeq$'' \ts on~$\cA$
that satisfies the log-supermodularity property
\[
\ff_\lambda \. \ff_{\mu}  \, \preccurlyeq \, \ff_{\lambda \vee \mu} \. \ff_{\lambda \wedge \mu}
\quad \text{for all } \quad \lambda,\mu \in \cP,
\]
and is closed under  addition and multiplication by nonnegative scalars.
It is tempting to conjecture that these assumptions suffice to imply the inequality
\begin{equation}\label{eq:NAD}\tag{naive AD}
	\bigg( \sum_{\lambda \in \cP} \fa({\lambda}) \ff_{\lambda}\bigg)
\bigg( \sum_{\lambda \in \cP} \fb({\lambda}) \ff_{\lambda}\bigg)
\ \preccurlyeq \
\bigg( \sum_{\lambda \in \cP} \fc({\lambda}) \ff_{\lambda}\bigg)
\bigg( \sum_{\lambda \in \cP} \fd({\lambda}) \ff_{\lambda}\bigg).
\end{equation}
Indeed, for Schur functions, \eqref{eq:NAD} and \eqref{eq:LPP} immediately implies \eqref{eq:ADS}.

However, this naive generalization is false, as demonstrated by the following counterexample.
Let \. $\fa,\fb,\fc,\fd:\cP\to \{0,1\}$ \. be given by the indicator functions
$$
	\fa \, := \, \mathbf{1}_{\{221,  22 \}}\., \quad
	\fb \, := \, \mathbf{1}_{\{311, 31 \}}\., \quad
	\fc \, := \, \mathbf{1}_{\{321,  32 \}}\., \quad
	\fd \, := \, \mathbf{1}_{\{211, 21 \}}\..
$$
It is easy to verify that these functions satisfy \eqref{eq:AD-cond}.
The \eqref{eq:NAD} then states
\begin{align*}
(\ast) \qquad
\ff_{221}\ts \ff_{311}  \, + \, {\textcolor{blue}{\ff_{221}\ts \ff_{31}}}
\, + \, {\textcolor{blue}{\ff_{22}\ts \ff_{311}}} \, + \, \ff_{22}\ts \ff_{31}
\ \preccurlyeq \
\ff_{321}\ts \ff_{211}  \, + \, {\textcolor{red}{\ff_{321}\ts \ff_{21}}}
\, + \, {\ff_{32}\ts \ff_{211}}\, + \,\ff_{32}\ts \ff_{21}\ts.
\end{align*}
Observe that the LHS contains the terms \ts
$\blue{\ff_{221}\ff_{31}}$ \ts and \ts $\blue{\ff_{22}\ff_{311}}$.
In contrast, the only term on the RHS that dominates these under \ts $\preccurlyeq$ \ts
is \ts $\red{\ff_{321}\ff_{21}}$.  This provides a counterexample to \eqref{eq:NAD}.

On the other hand, the inequality $(\ast)$ does hold for Schur functions \ts
$\ff_{\lambda}=\fs_{\lambda}$.  This is justified by the following relations:
\begin{align*}
	\fs_{221} \. \fs_{311} \,  \leqslant_{\fs} \,
	\fs_{321} \. \fs_{211}\ts, \ \ \quad
	\fs_{22} \. \fs_{31} \,  \leqslant_{\fs} \,
	\fs_{32} \. \fs_{21}\ts, \ \ \quad
	\blue{\fs_{221} \. \fs_{31}} \. + \. \blue{\fs_{22} \. \fs_{311}}
	\, \leqslant_{\fs} \, \red{\fs_{321} \. \fs_{21}}  \. + \.
	\fs_{32} \. \fs_{211}.
\end{align*}
Here the first two inequalities are immediate consequences of \eqref{eq:LPP}.
By contrast, the third inequality is an special case of a new family of
Schur inequalities, which we call the \defng{Schur duet inequalities}
(Corollary~\ref{cor:duet-Schur}).

\smallskip

\subsection{Losing Schur functions}\label{ss:recover}
It is natural to ask whether \eqref{eq:AD}
is a special case of \eqref{eq:ADS}.
While we are not aware of such implication, we can deduce
it from the more general \eqref{eq:ADS-skew},
as follows.

By Birkhoff's representation theorem~\cite{Bir}, every finite
distributive lattice is isomorphic to a sublattice of a Boolean lattice, cf.\ Remark~\ref{rem:dist}.
Therefore, to prove the full
AD inequality, it suffices to establish it for the special case where the
underlying lattice is the Boolean lattice \ts $\rB_\ell = 2^{[\ell]}$.
For a subset \. $E \in  2^{[\ell]}$, \ts let $\la:=\la(E)=(\lambda_1,\ldots, \lambda_\ell)$ \ts
be a partition given by
\begin{align*}
	\lambda_i \ := \
	\begin{cases}
		\ell-i+1 & \text{ if } \ i \in E,\\
		\ell-i & \text{ if } \ i \notin E.
	\end{cases}
\end{align*}
Let \. $\mu:= (\ell-1,\ell-2,\dots,1,0)$.
Note that $\lambda/\mu$ consists of a collection of disjoint boxes.
Therefore, \. $\fs_{\lambda/\mu} (1,0,0,\ldots) = 1.$

Let \. $\fa,\fb,\fc,\fd:\rB_\ell \to \rr_{\geq 0}$ \. be four functions
satisfying the assumptions \eqref{eq:AD-cond} in the AD inequality.
Define function \. $\fa':\cP \times \cP \to \rr_{\geq 0}$ \. by
\[ \fa'(\nu,\rho) \ := \
\begin{cases}
\. \fa(\lambda,\mu)
& \text{ if } \ \nu=\lambda(E) \ \text{ for some } \ E \in 2^{[\ell]} \ \text{ and }  \ \rho=\mu,\\
\. 0 &\text{ otherwise}.
\end{cases}
\]
Let functions \. $\fb',\fc',\fd':\cP \times \cP \to \rr_{\geq 0}$ \. be defined analogously.
It is straightforward to verify that these four functions $\fa', \fb',\fc',\fd'$ satisfy \eqref{eq:AD-cond-skew}.
Also note that,
\begin{align*}
	\sum_{\nu,\rho \ts \in \ts \cP} \. \fa'(\nu,\rho) \. \fs_{\nu/\rho}(1,0,0,\ldots) \ = \  \sum_{E \ts \in \ts 2^{[\ell]}} \.\fa(E),
\end{align*}
with analogous identities holding for \ts $\fb', \ts \fc'$ \ts and \ts $\fd'$.
By the Schur positivity in \eqref{eq:ADS-skew}, evaluating both sides
at \ts $(1,0,0,\dots)$ \ts gives
\[	\bigg( \sum_{E \ts \in \ts 2^{[\ell]} }  \fa (E)\bigg)
	\bigg( \sum_{E \ts \in \ts 2^{[\ell]} }  \fb(E)\bigg)
\ \leq \
	\bigg( \sum_{E \ts \in \ts 2^{[\ell]} }  \fc (E)\bigg)
	\bigg( \sum_{E \ts \in \ts 2^{[\ell]} }  \fd (E)\bigg).
\]
This is the desired AD inequality on~$\rB_\ell\ts$.

%%%%%%%%%%%%%%%%%%%%%%%%%%%%%%%%%%%%%%%%%%%%%%%%%%%%%%%%%%%%%%%%%%%%%%%%%%%%%%%%%%%%
%%%%%%%%%%%%%%%%%%%%%%%%%%%%%%%%%%%%%%%%%%%%%%%%%%%%%%%%%%%%%%%%%%%%%%%%%%%%%%%%%%%%
%%%%%%%%%%%%%%%%%%%%%%%%%%%%%%%%%%%%%%%%%%%%%%%%%%%%%%%%%%%%%%%%%%%%%%%%%%%%%%%%%%%%
%%%%%%%%%%%%%%%%%%%%%%%%%%%%%%%%%%%%%%%%%%%%%%%%%%%%%%%%%%%%%%%%%%%%%%%%%%%%%%%%%%%%
%%%%%%%%%%%%%%%%%%%%%%%%%%%%%%%%%%%%%%%%%%%%%%%%%%%%%%%%%%%%%%%%%%%%%%%%%%%%%%%%%%%%
%%%%%%%%%%%%%%%%%%%%%%%%%%%%%%%%%%%%%%%%%%%%%%%%%%%%%%%%%%%%%%%%%%%%%%%%%%%%%%%%%%%%

\medskip

\section{Orchestra inequalities}\label{s:Orch}

In this section we introduce the \defng{AD cone} \ts and the \defng{orchestra inequalities},
generalizing the \defng{duet inequalities}.  This is the main technical tool we need to
prove Schur positive inequalities.

\subsection{Ahlswede--Daykin cone}\label{ss:AD-cone}

Fix a positive integer $\ell$.
Let $\cE$ and $\cF$ be two disjoint copies of $[\ell] := \{1, 2, \ldots, \ell\}$.
Formally, we may regard $\cE := [\ell] \times \{0\}$ and $\cF := [\ell] \times \{1\}$ to distinguish them.
However, to avoid making the notation overly heavy, we will refer to elements of $\cE$ simply by the index $i$ and elements of $\cF$ by the index $j$.
Throughout this paper, we use the disjoint union symbol $\sqcup$ to emphasize that elements are distinguished by their origin, even when they share the same numerical value.

Consider a function \. $\cO: 2^{\cE \sqcup \cF}\to \rr$.
For every set partition  \. $\{B_1, \ldots, B_t\}$ \. of $\cE \sqcup  \cF$,
we say that \ts $\cO$ \ts satisfies an \defnb{orchestra inequality}, if
	\begin{equation}\label{eq:Orch}\tag{Orch}
		\sum_{H \ts \subseteq \ts [t]} \. \cO \bigg(\bigsqcup_{h \ts \in \ts H} \.  B_h \bigg) \ \geq \ 0.
	\end{equation}
Here, the sum runs over all \ts $2^t$ \ts subsets of \ts $[t]=\{1, \dots, t\}$.\footnote{We
use the term ``orchestra'' because it aggregates the value of the function $\cO$
over every possible subset of partition part, akin to how an orchestra's collective
output is defined by the combination of all its individual sections.}
Define \defnb{Ahlswede--Daykin {\rm (AD)} cone} \.
 $\cC^{\ell}$ \. as the set of functions satisfying
orchestra inequalities for all set partitions of \ts $\cE \sqcup  \cF$.
\smallskip

\begin{ex}\label{ex:orchestra}
	For $\ell=2$, the set $\mathcal{E} \sqcup \mathcal{F}$ has 4 elements,
	resulting in 15 possible partitions. The structure of the orchestra
	inequality depends entirely on the number of blocks in the chosen
	partition. There are the four possible cases:
	
\smallskip

	\noindent \textit{Case 1. One-block partition:} $\{\mathcal{E} \sqcup \mathcal{F}\}$: \\
	There is 1 such partition. The inequality aggregates $2$ terms:
	\begin{equation*}
		\mathcal{O}(\varnothing) + \mathcal{O}(\mathcal{E} \sqcup \mathcal{F}) \ge 0.
	\end{equation*}
	
\smallskip
	\noindent \textit{Case 2. Two-block partitions:} $\{B_1, B_2\}$: \\
	There are 7 such partitions. For each, the inequality aggregates
	$2^2$ terms:
	\begin{equation*}
		\mathcal{O}(\varnothing) + \mathcal{O}(B_1) + \mathcal{O}(B_2)
		+ \mathcal{O}(\mathcal{E} \sqcup \mathcal{F}) \ge 0.
	\end{equation*}
	
\smallskip
	\noindent \textit{Case 3. Three-block partitions:} $\{B_1, B_2, B_3\}$: \\
	There are 6 such partitions. For each, the inequality aggregates
	$2^3$ terms:
{\small
	\[
\cO(\varnothing) + \cO(B_1) + \cO(B_2) + \cO(B_3) + \cO(B_1\sqcup B_2) + \cO(B_1\sqcup B_3) + \cO(B_2\sqcup B_3) + \cO(\cE\sqcup\cF) \ \geq \ 0.
\]
}
	
\smallskip
	\noindent \textit{Case 4. Four-block partition:} $\{B_1, B_2, B_3, B_4\}$: \\
	There is 1 such partition, where every element is its own block.
	The inequality aggregates $2^4$ terms:
{\small
	\begin{align*}
	&\mathrel{\phantom{=}} \cO(\varnothing) + \cO(B_1) + \cO(B_2) + \cO(B_3) + \cO(B_4) \\
	&+\cO(B_1\sqcup B_2) + \cO(B_1\sqcup B_3) + \cO(B_1\sqcup B_4) + \cO(B_2\sqcup B_3) + \cO(B_2\sqcup B_4) + \cO(B_3\sqcup B_4)\\
	&+\cO(B_1\sqcup B_2\sqcup B_3) + \cO(B_1\sqcup B_2\sqcup B_4) + \cO(B_1\sqcup B_3\sqcup B_4) + \cO(B_2\sqcup B_3\sqcup B_4)
	+\cO(\cE\sqcup\cF)
	\ \geq \ 0.
\end{align*}
}
\end{ex}

\smallskip

Intuitively, the definition of this cone captures the essential inequalities
derivable from the constraints \eqref{eq:AD-cond} in the Ahlswede--Daykin Theorem~\ref{thm:AD}.
Note that the cone definition incorporates both $\cE$ and $\cF$ to accommodate
the \emph{skew} \ts version of the ADS inequality.
If we are only aiming for the straight version,
the cone definition can be restricted only to~$\cE$.

\medskip

\subsection{Schur orchestra inequality} \label{ss:orch-Schur}
For convenience, we extend the definition of  skew Schur functions \. $\fs_{\la/\mu}$ \.
to arbitrary  vectors $\lambda,\mu \in \nn^\ell$.
We adopt the convention that \. $\fs_{\la/\mu}:=0$ \. if either $\lambda$ or $\mu$
is not a partition, or if \. $\mu \not\subseteq \lambda$.

For \ts $E \subseteq \cE$ \ts and \ts $F \subseteq \cF$,
let
\begin{equation}\label{eq:complement}
\overline{E} \ := \ \cE \setminus E,  \quad \overline{F} \ := \ \cF \setminus F  \quad \text{and} \quad \overline{E \sqcup F} \ts := \ts \overline{E} \ts \sqcup \ts \overline{F}\ts.
\end{equation}
The following result gives a mechanism to lift the real-valued inequalities in \eqref{eq:Orch}
to statements about Schur positivity.

\smallskip

\begin{thm}[{\rm \defn{Schur orchestra inequality}\ts}{}]\label{thm:orch}
Let $\alpha, \beta, \gamma,\delta \in \nn^\ell$ and let $\cO$ be a function that satisfies \eqref{eq:Orch}.
For subsets \ts $E \subseteq \cE$ \ts and \ts $F \subseteq \cF$, let \.
$\lambda(E)=(\lambda_1,\ldots, \lambda_\ell)$ \. and \. $ \mu(F)=(\mu_1,\ldots, \mu_\ell)$ \ts in \ts
$\nn^\ell$, where
\begin{align}\label{eq:lambda-E}
	\lambda_i
	\ := \  \begin{cases}
		\alpha_i+\gamma_i & \text{ if } \ i \in E,\\
		\alpha_i & \text{ if } \ i \notin E,
	\end{cases}
\end{align}
and
\begin{align}\label{eq:mu-F}
	\mu_j
	\ := \  \begin{cases}
		\beta_j+\delta_j & \text{ if } \  j \in F,\\
		\beta_j & \text{ if } \ j \notin F.
	\end{cases}
\end{align}
Then
\begin{equation}\label{eq:S-Orch}\tag{S-Orch}
	\sum_{E \subseteq \cE, \. F \subseteq \cF}  \cO \big(E \sqcup F\big)  \.  \fs_{\lambda(E)/\mu(F)}  \. \fs_{\lambda(\overline{E})/\mu(\overline{F})} \ \geqslant_{\fs} \ 0.
\end{equation}
\end{thm}

\smallskip

The sum in \eqref{eq:S-Orch} can be viewed as a sum over the finest partition of $\cE\cup\cF$ as in Example~\ref{ex:orchestra}, Case~4. A generalization of \eqref{eq:S-Orch} involving other partitions of $\cE\cup\cF$ is discussed in~$\S$\ref{ss:Schur-orchestra-finrem}. The present formulation suffices for all subsequent applications in this paper.
In particular, Theorem~\ref{thm:orch} implies the following corollary.

\smallskip

\begin{cor}[\thmb{Schur duet inequality}]\label{cor:duet-Schur}
Let $\lambda:=(\lambda_1,\lambda_2,\lambda_3)$ and $\mu:=(\mu_1,\mu_2,\mu_3)$ be nonnegative integer vectors of length $3$ such that $\lambda \supseteq \mu$.
Then
\[ \fs_{(\mu_1,\lambda_2,\lambda_3)} \.
\fs_{(\lambda_1,\mu_2,\mu_3)}  \, + \, \fs_{(\lambda_1,\mu_2,\lambda_3)}\.
\fs_{(\mu_1,\lambda_2,\mu_3)}     \ \leqslant_{\fs} \  \fs_{(\lambda_1,\lambda_2,\lambda_3)}\.
\fs_{(\mu_1,\mu_2,\mu_3)} \, + \, \fs_{(\lambda_1,\lambda_2,\mu_3)} \.
\fs_{(\mu_1,\mu_2,\lambda_3)}.
\]
\end{cor}

\smallskip

\begin{proof}
	We apply Theorem~\ref{thm:orch} as follows. Let \ts $\ell=3$, and
\ts $\alpha \gets \mu$\., \. $\beta \gets (0,0,0)$\., \. $\gamma \gets \lambda-\mu$\., \. $\delta \gets (0,0,0)$.
Then, for \ts $E \subseteq \cE$ \ts and \ts $F \subseteq \cF$, we have \ts $\la(E)=(\la_1,\la_2,\la_3)$ \ts where
	\begin{align*}
		\lambda_i
		\ := \  \begin{cases}
			\lambda_i & \text{ if } \ i \in E,\\
			\mu_i & \text{ if } \ i \notin E,
		\end{cases}
	\end{align*}
	and \ts $\mu(F)=(0,0,0)$.
	Define a function \ts $\cO:2^{\cE \sqcup \cF}\to \rr$ \ts by
	\[ \cO(E \sqcup F) \ := \
	\begin{cases}
		1 & \text{ if }  E=\{1,2,3\} \text{ or } \ E=\{1,2\},\\
		-1 & \text{ if }  E=\{2,3\} \text{ or }  \ E=\{1,3\},\\
		0 & \text{ otherwise.}
	\end{cases}  \]
	Note that function \ts $\cO$ \ts does not depend on~$F$ in this case.
	It is straightforward to verify that the function \ts $\cO$ \ts satisfies
 the orchestra inequalities \eqref{eq:Orch}, and that \eqref{eq:S-Orch} gives the desired inequality.
\end{proof}

%%%%%%%%%%%%%%%%%%%%%%%%%%%%%%%%%%%%%%%%%%%%%%%%%%%%%%%%%%%%%%%%%%%%%%%%%%%%%%%%%%%%
%%%%%%%%%%%%%%%%%%%%%%%%%%%%%%%%%%%%%%%%%%%%%%%%%%%%%%%%%%%%%%%%%%%%%%%%%%%%%%%%%%%%
%%%%%%%%%%%%%%%%%%%%%%%%%%%%%%%%%%%%%%%%%%%%%%%%%%%%%%%%%%%%%%%%%%%%%%%%%%%%%%%%%%%%
%%%%%%%%%%%%%%%%%%%%%%%%%%%%%%%%%%%%%%%%%%%%%%%%%%%%%%%%%%%%%%%%%%%%%%%%%%%%%%%%%%%%
%%%%%%%%%%%%%%%%%%%%%%%%%%%%%%%%%%%%%%%%%%%%%%%%%%%%%%%%%%%%%%%%%%%%%%%%%%%%%%%%%%%%
%%%%%%%%%%%%%%%%%%%%%%%%%%%%%%%%%%%%%%%%%%%%%%%%%%%%%%%%%%%%%%%%%%%%%%%%%%%%%%%%%%%%

\medskip

%\newpage

\section{Temperley--Lieb algebra and immanants}\label{s:TL}

In this section we give a quick review of \defng{Temperley--Lieb immanants},
which we use to prove orchestra inequalities in the next section.

\subsection{Temperley--Lieb algebra}\label{ss:TL-alg}
For a fixed \ts $\xi \in \cc$, define \defnb{Temperley--Lieb algebra} \.
$\mathrm{TL}_k(\xi)$ \. by the basis \ts $t_1,\dotsc,t_{k-1}$,  subject
to the relations
\begin{alignat*}{2}
	t_i^2 &= \xi t_i, &\qquad &\text{for } \ i=1,\dotsc,k-1, \\
	t_i t_j t_i &= t_i,   &\qquad &\text{if } \  |i-j|=1,\\
	t_i t_j &= t_j t_i,   &\qquad &\text{if } \ |i-j| \geq 2.
\end{alignat*}
We refer to \cite{DG26} for a friendly historical survey and many helpful references.

The basis elements of $\mathrm{TL}_k(\xi)$ can be represented by planar diagrams on the vertex set $\cV \sqcup \cW$,
\begin{align}\label{eq:VW}
 \cV := \{1,2,\ldots,k\} \quad \text{and} \quad \cW := \{k+1,k+2,\ldots,2k\},
\end{align}
where we use the symbol $\sqcup$ to emphasize that the union is disjoint.
Each basis element $\kappa$ corresponds to a perfect matching of these $2k$ points via $k$ non-intersecting strands. From now on these $2k$ points form two rows of $k$ points in each. The points in the first row are increasing from $1$ to $k$, and points in the second row are decreasing from $2k$ to $(k+1)$.
Equivalently, this says that there are no edge pairs \ts $\{a,b\}$ \ts and \ts $\{c,d\}$ \ts satisfying the crossing condition $a<c<b<d$,
see e.g.~\cite{TemLieb71}. We denote the set of these basis elements (or equivalently, \defng{planar diagrams}) by $\mathcal{B}_k$.
Note that the  cardinality of $\mathcal{B}_k$ is given by the \defnb{Catalan number} \.
$\Cat(k)=\frac{1}{k+1}\binom{2k}{k}$.

An example of generators of $\mathrm{TL}_{4}(2)$ is shown in Figure~\ref{fig:TL-1}, and remaining basis elements in $\cB_4$ are shown in Figure~\ref{fig:TL-1-rest}, giving \ts $\Cat(4)=14$ \ts elements in total.

\begin{figure}[hbt]
\begin{center}
	\includegraphics[scale=1]{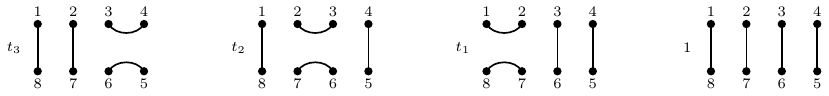}
\end{center}
\caption{Multiplicative generators for $\TL_4(2)$.}
\label{fig:TL-1}
\end{figure}

\begin{figure}[hbt]
\begin{center}
	\includegraphics[scale=1]{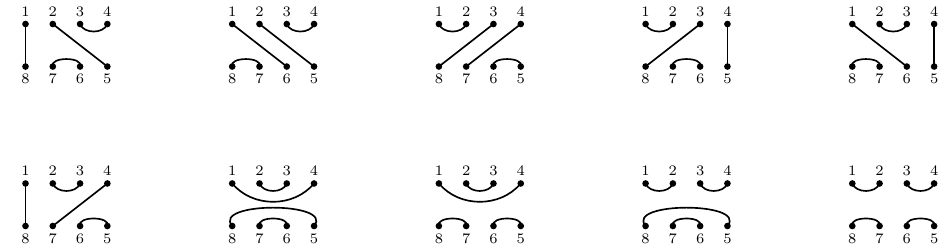}
\end{center}
\caption{Remaining basis elements in $\mathcal{B}_4$.}
\label{fig:TL-1-rest}
\end{figure}

\subsection{Temperley--Lieb immanants}\label{ss:imm}
Following Littlewood \cite{LittlewoodTGC} and Stanley \cite{StanPos},
the \defnb{immanants} \ts are polynomials in matrix entries defined
as follows.

Fix \ts $k\geq1$ \ts and let \ts $S_k$ \ts be symmetric group on \ts $k$ \ts elements.
Given a function $f: S_k \rightarrow \mathbb C$ and a $k\times k$ matrix $X=(x_{ij})_{1\leq i,j\leq k}$, the \defnb{$f$-immanant} of $X$ is the polynomial
\begin{equation}\label{eq:immdef}
	\Imm_f(X) \ := \ \sum_{w \in S_k} \. f(w) \. x_{1,w_1} \cdots x_{k,w_k}\..
\end{equation}

From now on, let \ts $\xi=2$.  In this case, there is a natural surjection
	\begin{equation}\label{eq:sntotn}
		\begin{aligned}
			\sigma \. : \.  \mathbb{C}[S_k] \rightarrow \TL_k(2) \ \ \ \mbox{ with } \ \ s_i \mapsto t_i - 1, \quad i\in[k-1],
		\end{aligned}
	\end{equation}
	where \ts $s_i=(i,i+1)$ \ts is an adjacent transposition.
	Since \. $(\kappa)_{\kappa\in \cB_k}$ \. forms a basis of \ts $\TL_k(2)$, for each \ts $\kappa\in\cB_k$,
define a function \ts $f_\kappa : S_k\to \mathbb R$ \ts as
	\begin{align}
		\sigma(w) \. := \. \sum_{\kappa\ts\in\ts \cB_k} \. f_\kappa(w) \. \kappa\ts.
	\end{align}
	Following \cite{RS05,RS06}, the \defnb{Temperley--Lieb immanant} \ts of a matrix~$X$,
denoted \. $\Imm_{\kappa}(X)$,  is defined as
	\begin{equation}
		\Imm_{\kappa}(X) \ := \  \Imm_{f_{\kappa}}(X)
		\ = \ \sum_{w \in S_k} \. f_{\kappa}(w) \. x_{1,w_1} \. \cdots \. x_{k,w_k}\ts.
	\end{equation}

\smallskip

\subsection{From diagrams to multisets}\label{ss:TL-multi}
Fix \ts $\ell\ge 1$ \ts and set \ts $k := 2\ell$. Let \.
$\cM = \{m_1 \leq \dots \leq m_k\}$ \. and \. $\cN = \{n_1 \leq \dots \leq n_k\}$ \.
be multisets of size \ts $k$ \ts with maximum multiplicity~$2$. For $M \subset \cM$ and $N \subset \cN$, let $\overline{M} := \cM \setminus M$ and $\overline{N} := \cN \setminus N$.  We extend this notation to unions by  \. $\overline{M \sqcup N} := \overline{M} \sqcup \overline{N}$. Throughout this paper, we restrict our attention to subsets $M$ and $N$ of size~$\ell$, such that \ts
$M, N, \overline{M}$ and \ts $\overline{N}$ \ts are \defn{simple}, i.e., contain no repeated elements.

Let $\kappa \in \mathcal{B}_k$ be a Temperley--Lieb diagram on the vertex set $\cV \sqcup \cW$.
We identify the multiset $\cM$ with $\cV$, and $\cN$ with $\cW$,
via a bijection \. $\psi \colon \cM \sqcup \cN \to \cV \sqcup \cW$.
This bijection is defined by
\begin{equation}\label{eq:psi}
	\psi(x) \ := \ \begin{cases}
		i & \text{if } x = m_i \in \cM, \\
		2k+1-j & \text{if } x = n_j \in \cN.
	\end{cases}
\end{equation}
For any subset $V \sqcup W \subseteq \cV \sqcup \cW$, we define its complement by $\overline{V \sqcup W} := (\cV \setminus V) \sqcup (\cW \setminus W)$.
Note that the map $\psi$ commutes with the complement operation; that is, for any subset $M \sqcup N \subseteq \cM \sqcup \cN$,
\[
\psi(\overline{M\sqcup N}) \ = \ \overline{\psi(M \sqcup N)}.
\]

\smallskip

\begin{ex}\label{ex:1}
Our running example is  by the following multisets of size 6:
\[  \cM \ = \ \{2,4,7,8,10,11\}, \qquad \cN \ = \ \{1,2,3,3,5,6\}.  \]	
Now let $M\subseteq \cM$ and $N\subseteq \cN$ be given by
\[ M \ = \  \{2,10,11\} \ = \  \{m_1,m_5,m_6\}, \qquad N \ = \ \{3,5,6\}  \ = \  \{n_4,n_5,n_6\}. \]
Notice that
\[ \overline{M} \ = \ \{4,7,8\} \  = \  \{m_2,m_3,m_4\}, \qquad \overline{N} \ = \ \{1,2,3\}  \ = \  \{n_1,n_2,n_3\},  \]
so $M$, $N$, $\overline{M}$, $\overline{N}$ are all simple subsets.
Then, applying  $\psi$  gives:
\begin{alignat*}{2}
  \psi(M \sqcup N) \ &= \ \{1,5,6\} \sqcup \{13-6,13-5,13-4\} \ &&= \ \{1,5,6,7,8,9\},   \\
   \psi(\overline{M \sqcup N}) \ &= \ \{2,3,4\} \sqcup \{13-3,13-2,13-1\} \ &&= \ \{2,3,4,10,11,12\}.
\end{alignat*}
\end{ex}

\smallskip

\subsection{Compatible diagrams}\label{ss:TL-comp}
We say that a diagram \ts $\kappa$ \ts is \defnb{compatible} \ts with  \ts $(\cM,\cN)$,
if it satisfies the following conditions:
\begin{enumerate}
	\item All duplicate labels in $\cM$ are paired:
	for every index $i$ such that $m_i = m_{i+1}$, the vertices $i$ and $(i+1)$ are connected by an edge.
	
	\item All duplicate labels in $\cN$ are paired:
	for every index $j$ such that $n_j = n_{j+1}$, the vertices $(2k+1-j)$ and $(2k-j)$ are connected by an edge.
\end{enumerate}

Let \ts $K \subseteq \cM \sqcup \cN$ \ts be a subset such that the intersections \ts $K \cap \cM$ \ts and \ts
$K \cap \cN$ \ts are both simple subsets of size~$\ell$.
We say that $\kappa$ is \defnb{compatible} with~$K$, if it is compatible with \ts $(\cM,\cN)$ \ts
and satisfies the following bipartite condition:
\begin{enumerate}
	\setcounter{enumi}{2}
	\item The perfect matching pairs the set $\psi(K)$ with its complement $\psi(\overline{K})$.
	In other words, every edge of $\kappa$ connects a vertex in $\psi(K)$ to a vertex not in $\psi(K)$.
\end{enumerate}
Denote by \. $\Theta(K):=\Theta_{\cM,\cN}(K)$ \. the set of diagrams from \ts $\cB_{k}$ \ts
that are compatible with~$K$.
Note that \ts $\Theta(K)=\Theta(\overline K)$ \ts by the symmetry.

\smallskip
\begin{ex}\label{ex:2}
Let $\cM $, $\cN$, $M$, $N$ be as in Example \ref{ex:1}, and  $K=M\sqcup N$.
Note that  $\psi(K)=\{1,5,6,7,8,9\}$ and $\overline{\psi(K)}=\{2,3,4,10,11,12\}$. Since $n_3=n_4$, the vertices 9 and 10 must be connected by an edge, as shown in Figure~\ref{fig:prematching}.
Then $\Theta(K)$ consists of the two diagrams shown in Figure~\ref{fig:compatible}.

\begin{figure}[hbt]
\begin{center}
	\includegraphics[scale=1]{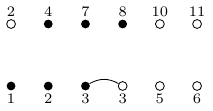}
\end{center}
\caption{Diagrams compatible with $(\cM,\cN)$ in Example~\ref{ex:1}.
The top and bottom rows are labeled by the elements of $\cM$ and $\cN$, respectively.
White nodes indicate elements in $\psi(K)$.}
\label{fig:prematching}
\end{figure}

\begin{figure}[hbt]
\begin{center}
	\includegraphics[scale=1]{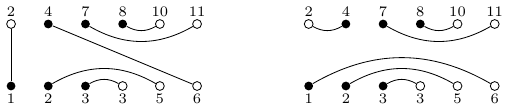}
\end{center}
\caption{Two  diagrams in $\Theta(K)$, where $K$ is as in Example~\ref{ex:2}. Note that white nodes are paired with black nodes.}
\label{fig:compatible}
\end{figure}
\end{ex}

\smallskip

\subsection{Rhoades--Skandera approach}\label{ss:TL-RS}
We now present a  result of Rhoades and Skandera expressing products of Schur
polynomials in terms of immanants.
Let \ts $X = (x_{i,j})_{i,j \geq 1}$ \ts be an \ts $\nn \times \nn$ \ts matrix.
We denote by \ts $X_{M,N} := (x_{m_i, n_j})_{1 \le i, j \le k}$ \ts the \ts
$k \times k$ \ts submatrix of~$X$ indexed by the multisets $M$ and~$N$.
Note that this submatrix may contain repeated rows or columns.

\smallskip

\begin{lemma}[{\cite[Prop.~4.4]{RS05}}]\label{lem:RS-1}
	Let $\ell \geq 1$.
	Let $\cM$ and $\cN$ be multisets of size $2\ell$ with maximum multiplicity $2$.
	For any  subsets $M \subseteq \cM$ and $N \subseteq \cN$ of size $\ell$
	such that $M, N, \overline{M}$, and $\overline{N}$ are all simple, and any $\nn \times \nn$ matrix $X$,
	\begin{equation}\label{eq:RS-1}
		\det(X_{M,N})\. \det(X_{\overline{M},\overline{N}}) \ =
\ \sum_{\kappa \in \Theta (M \sqcup \overline{N})} \. \Imm_{\kappa}(X_{\cM, \cN}).
	\end{equation}
\end{lemma}

\smallskip

In this paper, we apply Lemma~\ref{lem:RS-1} to the
\defnb{generalized Jacobi--Trudi matrix}~$\ts\aH$, defined as follows.
Let \ts $\fh_r(\bx):= \fs_r(\bx)$ \ts denote the $r$-th \defng{complete
homogeneous symmetric function} \ts
in \ts $\bx = (x_1, x_2, \dots)$.  We adopt the convention that \ts
$\fh_0 = 1$ \ts and \ts $\fh_r = 0$ \ts for \ts $r < 0$.
The matrix \ts $\aH$ \ts is the \. $\nn \times \nn$ \. matrix
with $(i,j)$-th entry given by \ts $\fh_{j-i}$, for all \ts $i,j \ge 1$.

\smallskip
\begin{lemma}[{\cite[Prop.~3.5]{RS06}}]\label{lem:RS-2}
	Let $\ell \geq 1$.
	Let $\cM$ and $\cN$ be multisets of size $2\ell$ with maximum multiplicity $2$.
	Then for any $\kappa \in \cB_{2\ell}$, we have:
	\begin{equation}\label{eq:RS-2}
		\Imm_{\kappa}(\aH_{\cN,\.\cM}) \, \geqslant_{\fs} \, 0.
	\end{equation}
	In other words, the Temperley--Lieb immanants of generalized Jacobi--Trudi matrices are Schur positive.
\end{lemma}

\smallskip
Let us now apply Lemmas~\ref{lem:RS-1} and~\ref{lem:RS-2} in conjunction with the
\defnb{Jacobi--Trudi identity}, see e.g.\ \cite[$\S$7.19]{EC}.
For partitions \. $\lambda=(\lambda_1, \ldots, \lambda_\ell)$ \. and \.
$\mu=(\mu_1, \ldots,\mu_\ell) \in \cP^{\ell}$, we have:
\begin{equation}\label{eq:JT}\tag{JT}
	\det(\aH_{N_{\mu},M_{\lambda}}) \ = s_{\lambda / \mu}\ts,
\end{equation}
where
\begin{equation}\label{eq:M-lambda}
M_{\lambda} \ := \ \{\la_{1}+\ell,\la_{2}+\ell-1, \ldots, \la_{\ell}+1\}\,, \quad N_{\mu} \ = \ \{\mu_1+\ell,
\mu_2+\ell-1, \ldots, \mu_{\ell}+1\}.
\end{equation}

We note that this identity holds even if $\mu_i > \lambda_i$ for some $i \in [\ell]$,
since both sides of the equation are equal to zero in that case.
However, this identity is not directly applicable if \ts $\lambda$ \ts or \ts $\mu$ \ts
are not valid partitions, i.e., if \. $\lambda_i < \lambda_{i+1}$ \. or \. $\mu_i < \mu_{i+1}$ \. for some~$i$. In such instances, the (combinatorial) Schur function \ts $\fs_{\lambda/\mu}$ \ts in the RHS vanishes by the convention adopted in this paper, whereas the determinant (the LHS) is not necessarily zero.

\smallskip
Finally, we have the following proposition as a direct consequence of
\eqref{eq:JT} and Lemma~\ref{lem:RS-1}.
We denote by \ts $\uplus$ \ts the \defng{multiset union}.

\begin{prop}\label{prop:RS}
	For all partitions \. $\la,\mu,\nu,\rho \in \cP^{\ell}$, we have:
 \begin{equation}\label{eq:RS}\tag{RS06}
	\begin{split}
		&	s_{\lambda / \mu} \. s_{\nu / \rho}
		\ = \ \sum_{\kappa \ts\in\ts \Theta(M_{\lambda} \sqcup N_{\rho})} \. \Imm_{\kappa}(\aH_{N_{\mu} \ts \uplus \ts N_{\rho}, \ts M_{\lambda} \ts \uplus \ts M_{\nu}}).
	\end{split}
\end{equation}
\end{prop}

\begin{proof}
We have:
\begin{equation*}
	\begin{split}
		&	s_{\lambda / \mu} \. s_{\nu / \rho} \ = \ \det(\aH_{N_{\mu},M_{\lambda}}) \. \det(\aH_{N_{\rho},M_{\nu}})
		\ = \ \sum_{\kappa \in  \Theta(M_{\lambda} \sqcup N_{\rho})} \. \Imm_{\kappa}(\aH_{N_{\mu} \uplus N_{\rho}, M_{\lambda} \uplus M_{\nu}}),
	\end{split}
\end{equation*}
where the first equality is due to \eqref{eq:JT}, and the second equality is due to
Lemma~\ref{lem:RS-1}.
\end{proof}

\smallskip

We emphasize that the equation above applies only when $\lambda, \mu, \nu$, and $\rho$ are partitions.
Otherwise, the product $\fs_{\lambda / \mu}\fs_{\nu / \rho}$ vanishes by our adopted convention.
We also note that one can replace the index set $\Theta(M_{\lambda}\sqcup N_{\rho})$ in
\eqref{eq:RS} with $\Theta(M_{\nu} \sqcup N_{\mu})$ due to the symmetry of $\Theta$,
so the apparent asymmetry in \eqref{eq:RS} is merely notational.

\smallskip

\begin{ex}\label{ex:3}
Let $\lambda=(8,8,1),~\mu=(0,0,0),~\nu=(5,5,3),~\rho=(3,3,2)$. Then
\begin{align*}
	M_{\lambda} \ = \ \{2,10,11\}, \quad M_{\nu} \ = \ \{4,7,8\}, \quad N_{\mu} \ = \ \{1,2,3\}, \quad N_{\rho} \ = \ \{3,5,6\}.
\end{align*}
Note that $M_{\lambda} \cup M_{\nu}= \cM$ and $N_{\mu} \cup N_{\rho}=\cN$,
where $\cM,\cN$ is as in Example~\ref{ex:1}.
Also note that $M_{\lambda}=M$, $M_{\nu}=\overline{M}$, $N_{\mu}=\overline{N}$, and $N_{\rho}=N$, where $M,N, \overline{M},\overline{N}$ is as in Example~\ref{ex:1}.
Finally, note that $M_{\lambda} \sqcup N_{\rho}=K$, where $K$ is as in Example~\ref{ex:2}.
Then  Proposition \ref{prop:RS} says  that
\begin{align*}
	\fs_{881} \. \fs_{553/332} \ &= \ \det\big(\aH_{\{1,2,3\},\{2,10,11\}}\big)\. \det\big(H_{\{3,5,6\},\{4,7,8\}}\big) \\
	&= \  \Imm_{\kappa_1}(\aH_{\cN,\cM})\. + \. \Imm_{\kappa_2} (\aH_{\cN,\cM}),
\end{align*}
where $\kappa_1$ and $\kappa_2$ are shown in the Figure~\ref{fig:compatible-2}.
\end{ex}

\begin{figure}[htb]
\begin{center}
	\includegraphics[scale=1]{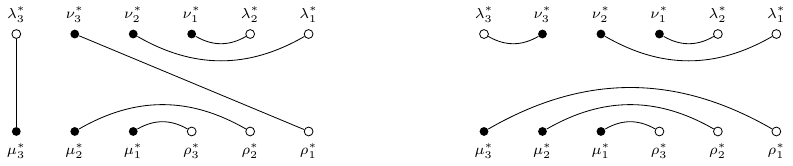}
\end{center}
\caption{Two compatible diagrams in \. $\Theta(M_{\lambda} \sqcup N_{\rho})$ \.
as in Example~\ref{ex:3}. Vertices are labeled by \. $\lambda_i^*, \mu_i^*, \nu_i^*$, and \ts $\rho_i^*$, where \. $\lambda_i^* = \lambda_i + \ell + 1  - i$ (with $\mu_i^*, \nu_i^*$, and $\rho_i^*$ defined analogously). These diagrams are equivalent to those in Figure~\ref{fig:compatible} up to the relabeling of nodes, following the correspondence in Example~\ref{ex:3}.
}
\label{fig:compatible-2}
\end{figure}

\medskip

%\newpage

%%%%%%%%%%%%%%%%%%%%%%%%%%%%%%%%%%%%%%%%%%%%%%%%%%%%%%%%%%%%%%%%%%%%%%%%%%%%%%%%%%%%%%
%%%%%%%%%%%%%%%%%%%%%%%%%%%%%%%%%%%%%%%%%%%%%%%%%%%%%%%%%%%%%%%%%%%%%%%%%%%%%%%%%%%%%%
%%%%%%%%%%%%%%%%%%%%%%%%%%%%%%%%%%%%%%%%%%%%%%%%%%%%%%%%%%%%%%%%%%%%%%%%%%%%%%%%%%%%%%
%%%%%%%%%%%%%%%%%%%%%%%%%%%%%%%%%%%%%%%%%%%%%%%%%%%%%%%%%%%%%%%%%%%%%%%%%%%%%%%%%%%%%%
%%%%%%%%%%%%%%%%%%%%%%%%%%%%%%%%%%%%%%%%%%%%%%%%%%%%%%%%%%%%%%%%%%%%%%%%%%%%%%%%%%%%%%

\section{Proof of orchestra inequalities}\label{s:orch-proof}
In this section, we introduce the definitions and auxiliary lemmas needed for the proof of the Schur orchestra inequality, and we conclude by proving the inequality itself. The main idea of the proof is to show that the left-hand side of \eqref{eq:S-Orch} can be expressed as a positive linear combination of TL immanants.
\subsection{Admissible subsets}\label{ss:lemmas}
In this subsection we collect various lemmas that will be used in the proof of Theorem~\ref{thm:orch}.
We need to set up the following notations.
Let $\ell\geq 1$  and
\. $\alpha, \beta, \gamma,\delta \in \nn^\ell$  \. be
as in the proposition, and are fixed throughout the rest of this section.

Recall that $\cE$ and $\cF$ are two disjoint copies of $[\ell]$.
We say that $E \subseteq \cE$ is \defnb{admissible} if both \.
$\lambda(E)=(\la_1,\ldots,\la_\ell)$ \. and \.
$\lambda(\overline{E})=(\la^\circ_1,\ldots,\la^\circ_\ell)$ \.
defined in \eqref{eq:lambda-E}, are partitions.
Similarly, we say that  $F \subseteq \cF$ is \defnb{admissible} if both \.
$\mu(F)=(\mu_1,\ldots,\mu_\ell)$ \. and \.
$\mu(\overline{F})=(\mu^\circ_1,\ldots,\mu^\circ_\ell)$ \.
defined in \eqref{eq:mu-F}, are partitions.
Finally, we call a union $E \sqcup F \subseteq \cE \sqcup \cF$ \defnb{admissible} \ts
if its components $E$ and $F$ are individually admissible.
We denote by $\cA$ the collection of all admissible subsets of $\cE \sqcup \cF$.

It is straightforward to verify that $(\cA,\vee,\wedge)$ forms a distributive lattice, where the lattice operations $\vee$ and $\wedge$ are given by set union and intersection, respectively.
We assume without loss of generality that there exists at least one admissible set, as otherwise the LHS of \eqref{eq:S-Orch} vanishes and the proposition holds vacuously.

\smallskip

\begin{lemma}\label{lem:lat-top}
	The empty set $\varnothing$ and the total set $\cE \sqcup \cF$ are both admissible.
	In other words, the vectors $\lambda(\cE)=\alpha+\gamma$, $\lambda(\varnothing)=\alpha$, $\mu(\cF)=\beta+\delta$, $\mu(\varnothing)=\beta$ are all partitions.
\end{lemma}
\smallskip

\begin{proof}
	Let $E \sqcup F$ be an admissible set, which exists by assumption.
	By the definition of admissibility, the complement $\overline{E} \sqcup \overline{F}$ is also admissible.
	Observe that
	\begin{align*}
		(E \sqcup F) \vee (\overline{E} \sqcup \overline{F}) &= (E \cup \overline{E}) \sqcup (F \cup \overline{F}) = \cE \sqcup \cF, \\
		(E \sqcup F) \wedge (\overline{E} \sqcup \overline{F}) &= (E \cap \overline{E}) \sqcup (F \cap \overline{F}) = \varnothing.
	\end{align*}
	Since the family of admissible sets is a lattice (and thus closed under join and meet), it follows that $\cE \sqcup \cF$ and $\varnothing$ are admissible.
\end{proof}

 \smallskip

 We now characterize admissibility in terms of an equivalence relation on elements of $\cE \sqcup \cF$.

 \begin{definition}[{\rm \defn{Admissibility equivalence relation}\ts}{}]\label{def:equiv-admis}
 We first define an equivalence relation $\sim_{\cE}$ on the set $\cE$,  by setting $i \sim_{\cE} i+1$.,
  for every \. $i \in \{1, \ldots, \ell-1\}$ satisfying
\[
\alpha_i < \alpha_{i+1} + \gamma_{i+1},
\]
and then taking the reflexive and transitive closure of this condition.
Note that  this condition identifies indices $i$ and $i+1$ that must share the same membership status in $E$ to prevent the partition condition from failing in either $\lambda(E)$ or $\lambda(\overline{E})$.
Similarly, we define an equivalence relation $\sim_{\cF}$ on the set $\cF$  by setting $j \sim_{\cF} j+1$ for any $j \in \{1, \ldots, \ell-1\}$ satisfying
\[
\beta_j < \beta_{j+1} + \delta_{j+1},
\]
and taking the reflexive and transitive closure.
Finally, we extend these to a single equivalence relation $\sim$ on $\cE \sqcup \cF$ by taking the disjoint union of $\sim_{\cE}$ and $\sim_{\cF}$.
 \end{definition}

 \smallskip

 \begin{lemma}\label{lem:admin}
 A subset $E \sqcup F \subseteq \cE \sqcup \cF$ is admissible if and only if it is a union of equivalence classes of $\sim$.
 \end{lemma}

\smallskip

\begin{proof}
	The forward implication ($\Rightarrow$) is immediate. We focus on the reverse implication ($\Leftarrow$).
	Let $E \subseteq \cE$ be a union of equivalence classes of $\sim_{\cE}$.
    We will show that $\lambda(E)$ is a partition, i.e., \. $\lambda_i \geq \lambda_{i+1}$ \ts for all \.
    $i \in \{1, \ldots, \ell-1\}$.
	Clearly, the argument for \ts $\lambda(\overline{E})$ \ts is identical by the symmetry.
	
	First, suppose that $i \sim_{\cE} i+1$. By the definition of $E$ as a union of equivalence classes, it follows that either $\{i, i+1\} \subseteq E$ or $\{i, i+1\} \cap E = \varnothing$.  We have:
	\begin{itemize}
		\item If \. $\{i, i+1\} \subseteq E$, then
		\[ \lambda_{i} = \alpha_i + \gamma_i \geq \alpha_{i+1} + \gamma_{i+1} = \lambda_{i+1}\., \]
		where the inequality holds because \. $\lambda(\cE) = \alpha + \gamma$ \. is a partition (by Lemma~\ref{lem:lat-top}).
		\item If \. $\{i, i+1\} \cap E = \varnothing$, then
		\[ \lambda_{i} = \alpha_i \geq \alpha_{i+1} = \lambda_{i+1}\., \]
		where the inequality holds because \. $\lambda(\varnothing) = \alpha$ \. is a partition (by Lemma~\ref{lem:lat-top}).
	\end{itemize}
	
	Next, suppose that $i \not\sim_{\cE} i+1$. By the definition of the relation $\sim_{\cE}$, the failure of the relation implies the strict inequality condition does not hold, so we must have:
	\begin{equation*}
		\alpha_i \geq \alpha_{i+1} + \gamma_{i+1}.
	\end{equation*}
	Consequently, $\lambda_i \geq \lambda_{i+1}$ regardless of whether $i$ or $i+1$ are in $E$, as the minimum possible value of \ts $\lambda_i$, which is $\alpha_i$, is at least the maximum possible value of \ts $\lambda_{i+1}$, which is  \ts $\alpha_{i+1} + \gamma_{i+1}$.
	
	This confirms that $E$ is admissible. By an analogous argument, if $F \subseteq \cF$ is a union of equivalence classes of $\sim_{\cF}$, then $\mu_j \geq \mu_{j+1}$ for all $j$, ensuring that $F$ is admissible.
	This completes the proof of the lemma.
\end{proof}

\smallskip

\subsection{From vectors to TL diagrams} \label{ss:orch-proof-diagrams}
We now relate the vectors \. $\alpha,\beta,\gamma,\delta$ \.
to TL diagrams and multisets discussed
in Section~\ref{s:TL}.
Let $k := 2\ell$, and let   $\cM$ and $\cN$ be the multisets of size $k$ given by
\begin{align}\label{eq:MN}
	\cM \ &:= \ M_{\alpha} \sqcup M_{\alpha+\gamma} \quad \text{and} \quad \cN \ := \ N_{\beta} \sqcup N_{\beta+\delta},
\end{align}
where $M_{\lambda}$ and $N_{\mu}$ are as defined in \eqref{eq:M-lambda}.
We reiterate that $\cM$ and $\cN$ are treated as disjoint sets.
Recall the set of vertices $\cV \sqcup \cW$ from \eqref{eq:VW} and the map $\psi$ from \eqref{eq:psi}.

\smallskip

\begin{definition}
	Let $p_1, \dots, p_{2\ell}$ and $\overline{p_1}, \dots, \overline{p_{2\ell}}$ be vertices in $\cV \sqcup \cW$ defined as follows:
\begin{equation}\label{eq:p_i}
\begin{aligned}
	p_i \, &:= \, \psi(\alpha_{\ell-i+1} + \gamma_{\ell-i+1} + i), & \overline{p_i} \. &:= \, \psi(\alpha_{\ell-i+1} + i) &&\text{for} \, \ i \in [\ell], \\
	p_{\ell+j} \, &:= \,\psi(\beta_{j} + \ell-j+1), & \overline{p_{\ell+j}} \, &:= \, \psi(\beta_{j} +\delta_j + \ell-j+1) &&\text{for} \, \ j \in [\ell].
\end{aligned}
\end{equation}
	To resolve potential ambiguities given by \. $\alpha_{\ell-i+1} + \gamma_{\ell-i+1} + i = \alpha_{\ell-i'+1} + i'$ \.  for some $i, i' \in [\ell]$, we impose the tie-breaking condition $\overline{p_{i'}} < p_i\ts$.
	Similarly, when \. $\beta_{j} + \ell-j+1= \beta_{j'} +\delta_{j'} + \ell-j'+1$\. for some $j,j' \in [\ell]$, we impose the tie-breaking condition $\overline{p_{\ell+j'}} < p_{\ell+j}\ts$.
	\end{definition}

\smallskip

By construction, these vertices satisfy the following inequalities:
\begin{align*}
	p_1 \ < \  p_2 \ < \ \cdots \ < \ p_{2\ell}\., \quad
	\overline{p_1} \ < \  \overline{p_2} \ < \ \cdots \ < \ \overline{p_{2\ell}}\., \quad \text{and} \quad
	\overline{p_i} \ < \ p_i \quad \text{ for all } \ i \in [2\ell].
\end{align*}

\smallskip

\begin{ex}\label{ex:4}
Let $\ell=3$, and let
\[\alpha \ = \ (5,5,1), \quad \beta \ = \ (0,0,0), \quad  \gamma \ = \ (3,3,2), \quad \delta \ = \ (3,3,2).   \]
It then follows by direction calculation that
{
$$
\alpha_{3}+1=2 < \alpha_{3}+\gamma_3+1=4<\alpha_{2}+2=7 < \alpha_{1}+3=8<\alpha_{2}+\gamma_2+2=10<\alpha_{1}+\gamma_1+3=11
$$
$$\overline{p_1}=1 \, < \, p_1=2 \, < \, \overline{p_2}=3\, < \,\overline{p_3}=4\, < \,p_2=5\, < \,p_3=6.
$$
}
\nin
Similarly, we have:
{
$$\beta_{3}+1=1<\beta_{2}+2=2< \beta_{1}+3=3\leq \beta_{3}+\delta_3+1=3<\beta_{2}+\delta_2+2=5<\beta_1+\delta_1+3=6
$$
$$
{p_6}=12\, > \,{p_5}=11 \, > \,{p_4}=10\, > \,\overline{p_6}=9\, > \,\overline{p_5}=8\, > \,\overline{p_4}=7.
$$
}

\nin
These  uniquely  determine \, $p_1, \dots, p_{6}$ \, and \, $\overline{p_1}, \dots, \overline{p_{6}}$ \, for this example.
\end{ex}

\smallskip

\subsection{From admissible sets to vertex coloring} \label{ss:orch-proof-Psi}
We introduce map~$\Psi$ to transition from a choice of admissible sets to a corresponding coloring of $2k$ vertices of the diagram. We start with the following:

\smallskip

\begin{definition}\label{defLp_i}
Let  $\Psi$ be the function that maps
an admissible set \. $E \sqcup F \subseteq \cE \sqcup \cF$ \.  to the subset \. $\{u_1, u_2, \dots, u_{2\ell}\}$ \.  of $\cV \sqcup \cW$ given by
\begin{equation}\label{eq:u_i}
	u_i \ = \ \begin{cases}
		p_i & \text{if } \ell-i+1 \in E \\
		\overline{p_i} & \text{if } \ell-i+1 \notin E
	\end{cases}
	\qquad \text{and} \qquad
	u_{\ell+j} \ = \ \begin{cases}
		p_{\ell+j} & \text{if } j \in F \\
		\overline{p_{\ell+j}} & \text{if } j \notin F
	\end{cases}
\end{equation}
for all \ts $i,j \in [\ell]$.
\end{definition}

\smallskip
It is straightforward to check that the admissibility assumption on $E \sqcup F$ implies that
\[ u_1 \, < \, u_2 \,  < \, \cdots \, < \, u_{2\ell} \.,
\]
and that
\[
\Psi(\cE \sqcup \cF) \, = \, \{p_1, p_2, \ldots, p_{2\ell} \}, \qquad \Psi(\varnothing) \, = \, \{\overline{p_1}, \overline{p_2}, \ldots, \overline{p_{2\ell}} \}.
\]
We write \. $P:=\{p_1,\ldots, p_{2\ell}\}$ \., and
\. $\overline{P}:=\{\overline{p_1},\ldots, \overline{p_{2\ell}}\}$\..
We extend the complement notation to any selected vertex $u_i \in \{p_i, \overline{p_i}\}$  $(i \in [2\ell])$ by
\[
\overline{u_i} \ := \ \begin{cases} \overline{p_i} & \text{if } u_i = p_i, \\ p_i & \text{if } u_i = \overline{p_i}. \end{cases}
\]

It is straightforward to verify that the image of the complement set is simply the set of complementary vertices:
\[
\Psi\big( \overline{E \sqcup F} \big) \, = \, \overline{\Psi\big(E \sqcup F\big)}  \, = \, \big\{ \overline{u_1} < \overline{u_2} < \cdots < \overline{u_{2\ell}} \big\}\..
\]

\smallskip

\begin{ex}\label{ex:5}
Let $\ell=3$ and let $\alpha,\beta,\gamma,\delta$ be as in Example~\ref{ex:4}.
Take \. $E = \{1,2\}$ \. and \. $F =  \varnothing$. It then follows from the definition that
\begin{align*}
	\Psi(E \sqcup F) \ &= \ \{\overline{p_1}, p_2, p_3, \overline{p_4}, \overline{p_5}, \overline{p_6}\} =\{u_1,\dots,u_6\}= \{1,5,6,7,8,9\},	
	\\
	\Psi(\overline{E \sqcup F}) \ &= \ \{{p_1}, \overline{p_2}, \overline{p_3}, {p_4}, {p_5}, {p_6}\} = \{\overline{u_1},\dots,\overline{u_6}\} = \{2,3,4,10,11,12\},
\end{align*}
which agree with the sets $\psi(M\sqcup N)$ and $\psi(\overline{M\sqcup N})$ in Example~\ref{ex:1}, respectively.

Note that the map $\Psi$ is defined so that its image $\Phi(E\sqcup F) = \{u_1,\dots,u_{2\ell}\}$ corresponds to white nodes and $\Phi(\overline{E\sqcup F}) =\{\overline{u_1},\dots,\overline{u_{2\ell}}\}$ to black nodes, as shown in Figure~\ref{fig:uandubar}.
\begin{figure}[hbt]
	\begin{center}
		\includegraphics[scale=1]{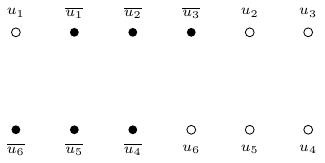}
	\end{center}
	\caption{White nodes indicate elements in $\Psi(E\sqcup F)=\{u_1,\dots,u_6\}=\{1,5,6,7,8,9\}$.}
	\label{fig:uandubar}
\end{figure}
\end{ex}

\smallskip

The next lemma justifies the definitions of \ts $p_i$'s \ts and \ts $\Psi:$

\smallskip

\begin{lemma}\label{lem:Psi}
	For any admissible set $E \sqcup F \subseteq \cE \sqcup \cF$,
	we have
	\begin{align*}
		\Psi(E \sqcup F) \ = \  \psi\big(M_{\lambda(E)} \sqcup \overline{N_{\mu({F})}}\big).
	\end{align*}
	\end{lemma}

\smallskip

\begin{proof}
For any $i \in [\ell]$, we have
		\begin{align*}
	& p_i \in \Psi(E\sqcup F) \quad \ \Longleftrightarrow \quad	\ell-i+1 \in E  \quad \ \Longleftrightarrow \quad   \alpha_{\ell-i+1} + \gamma_{\ell-i+1} + i \in  M_{\lambda(E)}   \\
	& \quad \ \Longleftrightarrow \quad  \psi(\alpha_{\ell-i+1} + \gamma_{\ell-i+1} + i) \in  \psi\big(M_{\lambda(E)} \sqcup \overline{N_{\mu({F})}}\big) \\
	& \quad \ \Longleftrightarrow \quad p_i \in  \psi\big(M_{\lambda(E)} \sqcup \overline{N_{\mu({F})}}\big)
	\end{align*}
where the first equivalence is by \eqref{eq:u_i}, the second equivalence is by \eqref{eq:lambda-E} and \eqref{eq:M-lambda}, 	
the third equivalence is because $\psi$ is a bijection, and the fourth equivalence is by \eqref{eq:p_i}.
By an analogous argument, we have for all $j \in [\ell]$,
		\begin{align*}
	& p_{\ell+j} \in \Psi(E\sqcup F) \quad \ \Longleftrightarrow \quad	j \in F \quad \ \Longleftrightarrow \quad   \beta_{j} + \delta_{j} + \ell-j+1 \in  N_{\mu(F)}   \\
		& \quad \ \Longleftrightarrow \quad  \beta_{j} + \ell-j+1 \in   \overline{N_{\mu({F})}} \\
	& \quad \ \Longleftrightarrow \quad  \psi(\beta_{j} + \ell-j+1) \in  \psi\big(M_{\lambda(E)} \sqcup \overline{N_{\mu({F})}}\big) \\
	& \quad \ \Longleftrightarrow \quad p_{\ell+j} \in  \psi\big(M_{\lambda(E)} \sqcup \overline{N_{\mu({F})}}\big),
\end{align*}
which completes the proof.
\end{proof}

\smallskip

\subsection{Compatible TL diagrams and admissible sets} \label{ss:orch-proof-comp}
We now extend the definition of compatible diagrams to admissible sets.

For an admissible set \. $E \sqcup F \subseteq \cE \sqcup \cF$, we say that a
Temperley--Lieb diagram \ts $\kappa \in \cB_{2\ell}$ \ts is
\defnb{compatible} \ts with \ts $E \sqcup F$, if \ts $\kappa$ \ts
is compatible with \. $M_{\lambda(E)} \sqcup \overline{N_{\mu({F})}}$.
We denote by \ts $\Theta(E \sqcup F)$ \ts the set of diagrams compatible
with \ts $E \sqcup F$. Formally, let
\begin{equation}\label{eq:Theta}
	\Theta(E \sqcup F) \ := \ \Theta_{\cM,\cN}\big(M_{\lambda(E)} \sqcup \overline{N_{\mu({F})}}\big),
\end{equation}
where \ts $\Theta_{\cM,\cN}$ \ts is defined as in~$\S$\ref{ss:imm}.

\smallskip

\begin{lemma}\label{lem:comp}
	Let \ts $\kappa \in \cB_{2\ell}$ \ts be a Temperley--Lieb diagram
compatible with \ts $(\cM,\cN)$, and let \ts $E \sqcup F \subseteq \cE \sqcup \cF$ \ts
be an admissible set.  Then \ts $\kappa$ \ts is compatible with \ts $E \sqcup F$ \ts
if and only if every edge of \ts $\kappa$ \ts connects a vertex in \ts $\Psi(E \sqcup F)$
\ts to a vertex in \ts $\overline{\Psi(E \sqcup F)}$.
\end{lemma}

\begin{proof}	
	The lemma is an immediate consequence of Lemma~\ref{lem:Psi} and the definition
of compatibility given in~$\S$\ref{ss:imm}.
\end{proof}
\smallskip

\begin{lemma}\label{lem:Theta}
	Every Temperley--Lieb diagram compatible with an admissible set $E \sqcup F$ is also compatible with $\cE \sqcup \cF$, i.e.
	\[	\Theta(E \sqcup F)  \ \subseteq  \ \Theta(\cE \sqcup \cF). \]
\end{lemma}

 \smallskip

The proof of this lemma is adapted from the argument in \cite[Thm.~11]{LPP07}.

\smallskip

 \begin{proof}
Let $\kappa \in \cB_{2\ell}$ be any  diagram compatible with $E \sqcup F$, and let $\{a,b\}$ be any edge of $\kappa$, with vertices $a < b$ in the set $\cV \cup \cW = [4\ell]$.
Let \. $U:=\{u_1< \cdots < u_{2\ell}\}:= \Psi(E \sqcup F) $\..
It follows from Lemma~\ref{lem:comp} that we can without loss of generality assume that \. $a =u_i$ \. and  $b:=\overline{u_j}$ for some $i,j \in [2\ell]$.
Since no edges of $\kappa$ can cross the edge $\{a,b\}$,
the elements of $U\cap [a+1,b-1]$ are matched with the elements of $\overline{U}\cap [a+1,b-1]$.
Let
\[ t  \ := \  |U \cap [a+1,b-1]| \ = \  |\overline{U} \cap [a+1,b-1]|. \]

By Lemma~\ref{lem:comp}, it suffices to show that one of $a,b$ lies in $P=\{p_1, \ldots, p_{2\ell}\}$ and the other lies in $\overline{P}=\{\overline{p_1}, \ldots, \overline{p_{2\ell}}\}$. Suppose to the contrary that this is not the case.
There are now two possibilities:
First suppose that both  $a,b$ lie in $P$.
This implies that
\[ a= p_i \quad \text{ and } \quad b=p_j \..  \]
By using the inequalities \. $\overline{p_i}< p_i$ \. and \. $\overline{p_j}< p_j$\., this implies
\begin{equation}\label{eq:fou}
\overline{u_i} \ = \  \overline{a} \ < \ a \ = \ u_i \qquad \text{ and } \qquad
u_j \ = \  \overline{b} \ < \ b \ = \  \overline{u_j}\ts.
\end{equation}

 	We now estimate $t$ in two ways.
First note that
\begin{align*}
	\big|U \cap [a+1,\overline{b}-1] \big|  \ \leq \  t-1,\end{align*}
because the interval \ts $[a+1, \overline{b}-1]$ \ts does not contain \ts
$\overline{b}=u_j$, but the interval \ts $[a+1, {b}-1]$ \ts does, as a consequence of \eqref{eq:fou}.
On the other hand, we also have
\[   \big|U \cap [a+1,\overline{b}-1] \big|  \ = \ \{u_{i+1},u_{i+2},\ldots, u_{j-1}\} \ = \  j-i-1. \]
Hence we have
\begin{equation}\label{eq:fou-1}
t -1 \ \geq \ j-i-1.
\end{equation}
Now note that
\begin{align*}
	\big|\overline{U} \cap [\overline{a}+1,{b}-1] \big|  \ \geq \  t,\end{align*}
as a consequence of \eqref{eq:fou}.
On the other hand, we also have
\[   \big|\overline{U} \cap [\overline{a}+1,{b}-1] \big|  \ = \ \{\overline{u_{i+1}},\overline{u_{i+2}},\ldots, \overline{u_{j-1}}\} \ = \  j-i-1. \]
Hence we have
\begin{equation}\label{eq:fou-2}
	t \ \leq \ j-i-1,
\end{equation}
which contradicts \eqref{eq:fou-1}, as desired.
The case that both $a,b$ lie in $\overline{P}$ is analogous.
This completes the proof.
 \end{proof}

 \smallskip

\subsection{Compatibility equivalence relation} \label{ss:orch-proof-equiv}
Let $\kappa \in \cB_{2\ell}$ be a Temperley--Lieb diagram compatible with $\cE \sqcup \cF$.
We define an equivalence relation $\sim_\kappa$ on the set $\cE \sqcup \cF$ via the following construction.

\smallskip
\begin{definition}[{\rm \defn{Compatibility equivalence relation}\ts}{}]\label{def:equiv-comp}
Let $G_{\kappa}$ be the graph obtained by adding the $2\ell$  edges \.
$\{p_i, \overline{p_i}\}$ \. for all \. $1\le i \le 2\ell$, to the diagram~$\ts \kappa$.
Since the original edges of $\kappa$ also form a perfect matching between $P$ and $\overline{P}$ (by Lemma~\ref{lem:comp}), the graph $G_{\kappa}$ is the union of two perfect matchings.
Consequently, $G_{\kappa}$
decomposes into a collection of disjoint even cycles.

The equivalence relation $\sim_{\kappa}$ is defined by whether the vertices in $P$ corresponding to specific indices belong to the same connected component (cycle) in $G_{\kappa}$. Explicitly:

$\circ$ \. for \ts $i, i' \in \cE$, let \ts $i \sim_{\kappa} i'$ \. if and only if \. $p_{\ell+1-i}$ \ts and \ts $p_{\ell+1-i'}$ \ts
belong to the same component,

$\circ$ \.  for \ts $j, j' \in \cF$, let \ts $j \sim_{\kappa} j'$ \. if and only if \. $p_{\ell+j}$ \ts and \ts $p_{\ell+j'}$ belong to the same component.

$\circ$ \. for mixed elements \ts $i \in \cE$ \ts and \ts $j \in \cF$, \ts $i \sim_{\kappa} j$ \.
if and only if \. $p_{\ell+1-i}$ \ts and \ts $p_{\ell+j}$ \ts belong to the same component.
\end{definition}

\smallskip

\begin{ex}\label{ex:6}
Let $\ell=3$, let $\alpha,\beta,\gamma,\delta$ be as in Example~\ref{ex:4}, and let $E,F$ be as in Example~\ref{ex:5}.
Consider the diagram $\kappa \in \Theta(E \sqcup F)$ in the LHS of Figure~\ref{fig:cycle}.
In this example, the relation $\sim_\kappa$ has three equivalence classes, namely
\[  \{(1,0), (2,0) \}, \qquad  \{(3,0), (1,1), (3,1)\} \qquad \text{ and } \qquad   \{ (2,1) \},  \]
where recall our convention that $\cE=[\ell]\times\{0\}$ and $\cF=[\ell] \times \{1\}$.

\begin{figure}[hbt]
	\begin{center}
		\includegraphics[scale=1.1]{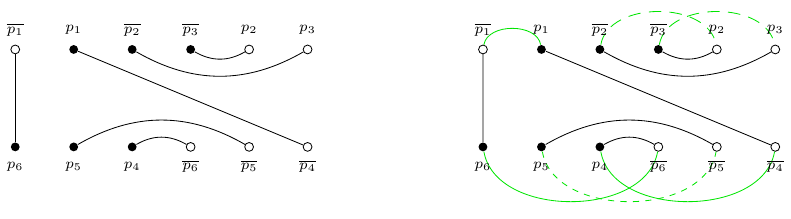}
	\end{center}
	\caption{(a) A diagram $\kappa$ that is compatible with $E\sqcup F$, where $E,F$ are as in Example~\ref{ex:6}. White nodes indicate elements in $\Psi(E \sqcup F)$. (b) the graph $G_{\kappa}$, where the additional edges are colored in green. }
	\label{fig:cycle}
\end{figure}
\end{ex}

 \smallskip

\begin{lemma}\label{lem:comp-kappa}
	Let $\kappa \in \cB_{2\ell}$ be a Temperley--Lieb diagram compatible with $\cE \sqcup \cF$, and let $E \sqcup F \subseteq \cE \sqcup \cF$ be an admissible set.
	Then $\kappa$ is compatible with $E \sqcup F$ if and only if the set $E \sqcup F$ is a union of equivalence classes of $\sim_{\kappa}$.
\end{lemma}

 \smallskip

 \begin{proof}
Let \. $U:=\{u_1 < \cdots < u_{2\ell}\} :=\Psi(E \sqcup F)$.
We first prove the \ts $(\Leftarrow)$ \ts direction.
Since $\kappa$ is compatible with $\cE \sqcup \cF$,
it follows from Lemma~\ref{lem:comp} that
every edge of $\kappa$ connects $p_i$ to $\overline{p_j}$ for some $i,j \in [2\ell]$.
It then follows from the definition of $G_{\kappa}$ that $p_i$ and $p_j$
belongs to the same connected component of $G_{\kappa}$.
On the other hand, since $E \sqcup F$ is a union of equivalence classes of $\sim_{\kappa}$,
it follows that either both $p_i$ and $p_j$ lie in $U$, or both  $p_i$ and $p_j$ lie in $\overline{U}$.

In the first case, it follows from \eqref{eq:u_i} that $p_i$ lies in $U$, while
$\overline{p_j}$ lies in~$\overline{U}$.
In the second  case, it follows from \eqref{eq:u_i} that $p_i$ lies in \ts $\overline{U}$,
while \ts $\overline{p_j}$ \ts lies in~${U}$.
Since the edge was chosen arbitrarily, it follows that every edge of \ts $\kappa$ \ts
connects a vertex in $U$ to a vertex in~$\overline{U}$.
By Lemma~\ref{lem:comp}, this implies that $\kappa$ is compatible with \ts $E \sqcup F$.

We now prove the \ts $(\Rightarrow)$ \ts direction:
We proceed by proving the contrapositive.
Suppose that  $E \sqcup F$ is not a union of equivalence classes of $\sim_{\kappa}$.
This  implies that there exists an equivalence class of $\sim_{\kappa}$ containing two elements  with different membership statuses in $E \sqcup F$.
It then follows  that there exist $i, j \in [2\ell]$ such that $p_i$ and $p_j$ belong to the same connected component of $G_{\kappa}$, yet $p_i \in U$ while $p_j \in \overline{U}$.

Since $G_{\kappa}$ is a bipartite graph with parts $P$ and $\overline{P}$, it follows that $p_i$ and $p_j$ must be connected by a path of even length in $G_{\kappa}$, where every vertex at an even position along the path belongs to $P$.
It follows that along this path, there exists a pair of vertices at distance 2 such that one belongs to $P \cap U$ and the other to $P \cap \overline{U}$. By substituting the original endpoints with this pair, we may assume without loss of generality that $p_i$ and $p_j$ are at distance $2$ in $G_{\kappa}$.
This implies that either $\{p_i, \overline{p_j}\}$ or $\{\overline{p_i}, p_j\}$ is an edge in $\kappa$. By symmetry, we assume the former.
On the other hand, since $p_i \in U$ and $p_j \in \overline{U}$, it follows from \eqref{eq:u_i} that $p_i$ and $\overline{p_j}$ both belong to~$U$.
Consequently, by Lemma~\ref{lem:comp}, we have $\kappa$ is not compatible with $E \sqcup F$, which gives us the desired contrapositive.
 \end{proof}

\smallskip

\subsection{Proof of Theorem~\ref{thm:orch}}
Let $\ell\geq 1$, $\alpha, \beta, \gamma,\delta$ and $\cO$ be as in the theorem.
Let $\cM$ and $\cN$ be the multisets defined in \eqref{eq:MN}.
We fix these parameters throughout  this proof.
Recall the definition of admissible sets and compatible diagrams from $\S$\ref{ss:lemmas}.
It follows from Proposition~\ref{prop:RS} that
\begin{align*}
	& \sum_{E \subseteq \cE, \. F \subseteq \cF}  \cO \big(E \sqcup F\big)  \.  \fs_{\lambda(E)/\mu(F)}  \. \fs_{\lambda(\overline{E})/\mu(\overline{F})} \ \ = \, \sum_{\substack{\text{admissible } \\ E \sqcup  F \. \subseteq \.\cE \sqcup \cF}}  \. \sum_{\kappa \ts\in\ts \Theta(E \sqcup F)}  \cO \big(E \sqcup F\big)  \Imm_{\kappa}(\aH_{\cN, \cM}).
\end{align*}
Let $c(E\sqcup F, \kappa)$ be the indicator function that $E \sqcup F$ is an admissible set and that $\kappa$ is compatible with  $E \sqcup F$.
It then follows from Lemma~\ref{lem:Theta}  that the sum above is equal to
\begin{align*}
\sum_{{E \sqcup  F \. \subseteq \. \cE \sqcup \cF}}  \. \sum_{\kappa \ts\in\ts \Theta(\cE \sqcup \cF)}
\. c(E\sqcup F, \kappa) \. \cO \big(E \sqcup F\big) \. \Imm_{\kappa}(\aH_{\cN, \cM}).
\end{align*}
By exchanging the order of summation, the sum above is equal to
\begin{align*}
	 \. \sum_{\kappa \ts\in\ts \Theta(\cE \sqcup \cF)} \Imm_{\kappa}(\aH_{\cN, \cM})  \.
\sum_{{E \sqcup  F \.\subseteq\. \cE \sqcup \cF}} \.  c(E\sqcup F, \kappa) \. \cO \big(E \sqcup F\big)\ts.
\end{align*}
Recall that \ts $\Imm_{\kappa}(\aH_{\cN, \cM})$ \ts is Schur positive by Lemma~\ref{lem:RS-2}.
Thus, it suffices to show that, for all \ts $\kappa \in \Theta(\cE \sqcup \cF)$,
we have
\begin{equation*}
  \sum_{{E \sqcup  F \. \subseteq \. \cE \sqcup \cF}} \. c(E\sqcup F, \kappa) \. \cO \big(E \sqcup F\big)    \ \geq \ 0\ts.
\end{equation*}

Let $\sim$ and $\sim_{\kappa}$ be the equivalence relations defined in
Definition~\ref{def:equiv-admis} and~\ref{def:equiv-comp}, respectively. Let \ts $\{B_1, \ldots, B_t\}$ \ts be the set partition of \ts
$\cE \sqcup \cF$ \ts into equivalence classes under the relation generated by \ts $\sim$ \ts and \ts $\sim_{\kappa}$.
It follows from Lemmas~\ref{lem:admin} and \ref{lem:comp-kappa},
that \ts $c(E\sqcup F, \kappa)=1$ \ts if and only if \ts $E \sqcup F$ \ts
is a union of parts of the partition \ts $\{B_1, \ldots, B_t\}$.
Hence, by the orchestra inequality~\eqref{eq:Orch}, we have:
\begin{equation*}
	\sum_{{E \sqcup  F \. \subseteq \. \cE \sqcup \cF}} \. c(E\sqcup F, \kappa)
\. \cO \big(E \sqcup F\big)    \ = \ \sum_{H \subseteq [t]}  \. \cO \bigg(\bigsqcup_{h \in H} B_h \bigg)\ \geq \ 0\ts.
\end{equation*}
This proves the proposition.
\qed

\subsection{Generalization of the Schur orchestra inequality}\label{ss:Schur-orchestra-finrem}
Here we present a generalization of Theorem~\ref{thm:orch}.
Recall the definitions of $\cE$, $\cF$, and the Ahlswede--Daykin cone.

\smallskip

\begin{thm}\label{thm:orch-ext}
	Let $\cO$ be a function in the Ahlswede--Daykin cone.
	Let $\alpha, \beta, \gamma, \delta \in \nn^\ell$ be nonnegative
	integer vectors of length $\ell$. Then, for every partition
	$\{B_1, \ldots, B_t\}$ of $\cE \sqcup \cF$,
	\begin{equation*}
		\sum_{H \subseteq [t]} \cO \big(B_H\big) \,
		\fs_{\lambda(B_H \cap \cE)/\mu(B_H \cap \cF)} \,
		\fs_{\lambda(\overline{B_H \cap \cE})/\mu(\overline{B_H \cap \cF})}
		\ \geqslant_{\fs} \ 0,
	\end{equation*}
	where $\lambda(\cdot)$ and $\mu(\cdot)$ are defined as in
	\eqref{eq:lambda-E} and \eqref{eq:mu-F}, respectively,
	and \. $B_H:=\bigsqcup_{h \in H} B_h$.
\end{thm}

\smallskip

In particular, Theorem~\ref{thm:orch} corresponds to the special
case of Theorem~\ref{thm:orch-ext} for the set partition \.
$\{\{1\}, \{2\}, \ldots, \{\ell\}\}$. The proof of the theorem
proceeds analogously to that of Theorem~\ref{thm:orch},
albeit with heavier notation.  Since we do not need this generalization,
we omit the details (cf.~$\S$\ref{ss:Schur measures}).

\medskip

\section{Proof of the ADS inequality}\label{s:ADS}

In this section, we establish the skew ADS inequality (Theorem~\ref{thm:ADS-skew}),
which we use to derive the ADS inequality (Theorem~\ref{thm:ADS}).

\subsection{Proof of Theorem~\ref{thm:ADS-skew}}
Expanding the product in \eqref{eq:ADS-skew} and grouping the summands, we have:
	\begin{align*}
		&	\bigg(\sum_{\lambda, \ts\mu \ts \in \ts\cP} \fc({\lambda},\mu) \fs_{\lambda/\mu} \bigg)  \bigg(\sum_{\lambda,\ts\mu  \ts\in \ts\cP} \fd({\lambda},\mu) \fs_{\lambda/\mu} \bigg) \ - \ \bigg(\sum_{\lambda, \ts\mu \ts \in \ts\cP} \fa({\lambda},\mu) \fs_{\lambda/\mu} \bigg) \bigg(\sum_{\lambda, \ts\mu \ts \in \ts\cP} \fb({\lambda},\mu) \fs_{\lambda/\mu} \bigg) \\
		&	\quad = \ \sum_{\lambda,\ts\mu,\ts\nu,\ts\rho  \ts\in \ts\cP} \big(\fc({\lambda},\mu) \fd(\nu,\rho) \ - \ \fa({\lambda},\mu)  \fb(\nu,\rho)\big)  \fs_{\lambda/\mu} \fs_{\nu/\rho}  \\
		& \quad = \ \sum_{\substack{(\alpha,\ts\beta,\ts\gamma,\ts\delta)}} \ \ \sum_{\substack{(\lambda,\ts\mu,\ts\nu,\ts\rho)}} \.  \big(\fc(\lambda,\mu)  \fd(\nu,\rho)  \. - \.  \fa(\lambda,\mu)  \fb(\nu,\rho)\big)   \fs_{\la/\mu} \. \fs_{\nu/\rho}.
	\end{align*}
	where the outer sum is over all \. $\alpha,\beta,\gamma,\delta \in \nn^\infty$ \. with finite support,
and the inner sum is over all \. $(\lambda,\mu,\nu,\rho) \in\cU(\alpha,\beta,\gamma,\delta)$, where
$$
\cU(\alpha,\beta,\gamma,\delta) \ := \
\big\{\lambda,\mu,\nu,\rho \in \cP \. : \.
\lambda \vee \nu=\alpha+\gamma, \.  \lambda \wedge \nu =\alpha, \.
				\mu \vee \rho = \beta+\delta, \. \mu \wedge \rho = \beta\big\}.
$$
Now, define \.
	$\Phi_{\fs}(\alpha,\beta,\gamma,\delta) \. := \. \Phi_{\fs}(\alpha,\beta,\gamma,\delta;\fa,\fb,\fc,\fd)$ \. to be the inner sum
	\begin{equation}\label{eq:AAD-skew}
		\Phi_{\fs}(\alpha,\beta,\gamma,\delta) \ := \ \sum_{\lambda,\ts\mu,\ts\nu,\ts\rho\ts \in\ts\cU(\alpha,\beta,\gamma,\delta)}  \big(\fc(\lambda,\mu)  \fd(\nu,\rho)  \. - \.  \fa(\lambda,\mu)  \fb(\nu,\rho)\big)  \. \fs_{\la/\mu} \. \fs_{\nu/\rho}.
	\end{equation}
	To prove Theorem~\ref{thm:ADS-skew}, it suffices to show that each \.
$\Phi_{\fs}(\alpha,\beta,\gamma,\delta)$ \. is Schur positive, which we state as Theorem~\ref{thm:SAAD} below.  This completes the proof. \qed

\smallskip

\begin{thm}[{\rm \defn{LSS inequality}{}\ts}{}]\label{thm:SAAD}
	Let \. $\ell\ge 1$, \ts $\alpha,\beta,\gamma,\delta \in \nn^{\ell}$,
let \. $\fa,\fb,\fc,\fd: \cP \times \cP \to \rr_{\geq 0}$ \.
be four functions satisfying \eqref{eq:AD-cond-skew}, and let \.
$\Phi_{\fs}(\alpha,\beta,\gamma,\delta)$ \. be defined by \eqref{eq:AAD-skew}.
	Then we have:
	\begin{align}\label{eq:LSS} \tag{LSS}
		\Phi_{\fs}(\alpha,\beta,\gamma,\delta) \ \geqslant_{\fs} \ 0.
	\end{align}
\end{thm}

\smallskip

Note that Theorem~\ref{thm:SAAD} is a Schur positive generalization
of the \defn{Lov\'asz--Saks inequality{}\ts}~\cite[Thm.~8]{LS06},
of which the RLS inequality (Theorem~\ref{t:RLS})
is an easy corollary.  Although somewhat technical to state,
the LSS inequality \eqref{eq:LSS} is the most powerful inequality
we have in this paper.

\smallskip

\subsection{Proof of Theorem~\ref{thm:SAAD}}
The following argument combines the RLS inequality
(Theorem~\ref{t:RLS}) and the Schur orchestra inequalities
(Theorem~\ref{thm:orch}), to obtain the result.

Let $\ell$, $\alpha,\beta,\gamma,\delta$ be  as in Theorem~\ref{thm:SAAD}, and set $k:=2\ell$.
Let $\cE$ and $\cF$ be two disjoint copies of $[\ell]$.
Note that $2^{\cE \sqcup \cF}$ forms a finite   distributive lattice with set union as the join and set intersection as the meet.
Recall the definitions of $\lambda(E)$ and $\mu(F)$ for $E \subseteq \cE$ and $F \subseteq \cF$ given in \eqref{eq:lambda-E} and \eqref{eq:mu-F}.
Recall the definition of admissible subsets from $\S$\ref{ss:lemmas}.

We  define the function $\cO:2^{\cE \sqcup \cF} \to \rr$ by
\begin{align}\label{eq:fv}
	\cO(E \sqcup F) \ := \
\fc\big(\lambda(E), \mu(F)\big)\cdot\fd\big(\lambda(\overline{E}), \mu(\overline{F})\big)  \, - \,  \fa\big(\lambda(E), \mu(F)\big)  \cdot \fb\big(\lambda(\overline{E}), \mu(\overline{F})\big)
\end{align}
if  \ts $E \sqcup F$ \ts is admissible, and let \ts $\cO(E \sqcup F) := 0$ \ts otherwise.
We now verify that this function $\cO$ satisfies the orchestra inequalities~\eqref{eq:Orch}.

Let
\. $\{B_1, \ldots, B_t\}$ \. be a set partition of  $\cE \sqcup  \cF$.
We define the functions \. $\fa', \fb', \fc', \fd' \colon 2^{[t]} \to \mathbb{R}_{\geq 0}$ \. as follows. For any subset $H \subseteq [t]$, let
$$
E_H \. := \. \bigcup_{h \in H} (B_h \cap \mathcal{E}) \quad \text{and}
\quad F_H \. :=  \. \bigcup_{h \in H} (B_h \cap \mathcal{F}).
$$
We define \. $\fa'\colon 2^{[t]} \to \mathbb{R}_{\geq 0}$ \. by
$$
\fa'(H) \, := \,
\begin{cases}
	\fa\big(\lambda(E_H), \mu(F_H)\big) & \text{if } \ \ \lambda(E_H), \mu(F_H) \in \cP, \\
	0 & \text{otherwise.}
\end{cases}
$$
The functions \ts $\fb'$, \ts $\fc'$ \ts and \ts $\fd'$ \ts are defined analogously.

We now check that \. $\fa',\fb',\fc',\fd'$ \. satisfies \eqref{eq:AD-cond}, i.e., for all \. $G,H \in 2^{[t]}$, we have:
\begin{equation}\label{eq:GH}
   \fa'(G)  \fb'(H) \ \leq  \ \fc'(G \vee H) \fd'(G \wedge H).
\end{equation}
Note that \eqref{eq:GH} holds trivially if any of the vectors \.
$\lambda(E_G), \mu(F_G), \lambda(E_H)$ \. or \. $\mu(F_H)$ \. is not a partition,
as the LHS equals zero. Therefore, we assume that all four are partitions.

This implies that the vectors \.
$\lambda(E_{G\vee H}), \mu(F_{G \vee H})$, \. $\lambda(E_{G \wedge H})$ \.
and \. $\mu(F_{G \wedge H})$ \. are also partitions. Indeed, observe that
\[
\lambda(E_{G \vee H}) \, = \, \lambda(E_G) \vee \lambda(E_H) \quad \text{and}
\quad \mu(F_{G \vee H}) \, = \, \mu(F_G) \vee \mu(F_H)\ts.
\]
Since $\mathcal{P}$ is closed under join and meet operations \ts $\vee$ \ts and \ts $\wedge$,
the resulting vectors are partitions. Analogous arguments apply to the other vectors.
Thus in this case, \eqref{eq:GH} reduces to checking that
\begin{equation*}
\aligned
& \fa\big(\lambda(E_G), \mu(F_G)\big) \. \cdot \. \fb\big(\lambda(E_H), \mu(F_H)\big) \\
& \qquad \leq  \  \fc\big(\lambda(E_G) \vee \lambda(E_H), \mu(F_G) \vee \mu(F_H)\big) \. \cdot \. \fd\big(\lambda(E_G) \wedge \lambda(E_H), \mu(F_G) \wedge \mu(F_H)\big),
\endaligned
\end{equation*}
which follows from the assumption \eqref{eq:AD-cond-skew}.
This completes the verification of \eqref{eq:GH}.

We can now apply the RLS inequality (Theorem~\ref{t:RLS}) to \. $\fa',\fb',\fc',\fd'$.
to obtain:
	\begin{align}\label{eq:AAD-1}
	\sum_{H \subseteq [t]}   \fc'(H)  \fd'(\overline{H})  \. - \. \fa'(H) \fb'(\overline{H})  \ \geq \ 0.
\end{align}
When \. $E_H \sqcup F_H$ \. is admissible, the corresponding term in the sum above are equal to
\begin{align*}
&\fc'(H)  \fd'(\overline{H})  \. - \. \fa'(H) \fb'(\overline{H})\\
&\quad = \ \fc\big(\lambda(E_H),\mu(F_H)\big)
	\fd\big(\lambda(\overline{E_H}),\mu(\overline{F_H})\big) \. - \.
	\fa\big(\lambda(E_H),\mu(F_H)\big)
	\fb\big(\lambda(\overline{E_H}),\mu(\overline{F_H})\big) \\
& \quad =_{\eqref{eq:fv}} \ \cO(E_H \sqcup F_H) \ = \  \cO\big( \bigsqcup_{h \in H} B_h\big).
\end{align*}
Similarly, when \. $E_H \sqcup F_H$ \. is not admissible, we have:
\begin{align*}
	&\fc'(H)  \fd'(\overline{H})  \. - \. \fa'(H) \fb'(\overline{H})  \ = \ 0 \ = \   \cO\big( \bigsqcup_{h \in H} B_h\big).
\end{align*}
Plugging these terms into summation \eqref{eq:AAD-1}, gives:
\begin{align*}
	\sum_{H \subseteq [t]}  \cO \big(\bigcup_{h \in H} B_h\big) \ \geq \ 0.
\end{align*}
In other words, this shows that \ts $\cO$ \ts satisfies the orchestra inequalities~\eqref{eq:Orch}.

It then follows from the Schur orchestra inequality (Theorem~\ref{thm:orch}), that
\[
	\sum_{E \subseteq \cE, \. F \subseteq \cF}  \cO \big(E \sqcup F\big)  \.  \fs_{\lambda(E)/\mu(F)}  \. \fs_{\lambda(\overline{E})/\mu(\overline{F})} \ \geqslant_{\fs} \ 0.
\]
Plugging \eqref{eq:fv} into the equation above, we obtain:
\[
\aligned
& \sum_{E \subseteq \cE, \. F \subseteq \cF}
\Big[\fc\big(\lambda(E), \mu(F)\big) \cdot
\fd\big(\lambda(\overline{E}), \mu(\overline{F})\big)  \ - \  \fa\big(\lambda(E), \mu(F)\big) \cdot \fb\big(\lambda(\overline{E}), \mu(\overline{F})\big) \Big] \, \times \\
& \hskip9.5cm \times \,
 \fs_{\lambda(E)/\mu(F)} \.  \fs_{\lambda(\overline{E})/\mu(\overline{F})} \ \geqslant_{\fs} \ 0.
\endaligned
\]
Finally, note that the LHS is equal to
\[ 2^{N} \. \Phi_{\fs}(\alpha,\beta,\gamma,\delta) \quad \text{where} \quad
N \. := \. |\{i \in [\ell] \. : \. \gamma_i=0 \}| \. + \. |\{i \in [\ell] \. : \. \delta_i=0 \}|,
\]
after the substitution \. $\lambda(E)\gets \lambda$, \ts $\mu(F)\gets \mu$,
\ts $\lambda(\overline{E})\gets \nu$, \ts and \. $\mu(\overline{F})\gets \rho$.
The completes the proof of Theorem~\ref{thm:SAAD}.
\qed

\smallskip

\subsection{Proof Theorem~\ref{thm:ADS}}\label{ss:proof-ADS}
We extend the domain of the functions $\fa,\fb,\fc,\fd: \cP \to \rr_{\geq 0}$
to  $\fa',\fb',\fc',\fd': \cP  \times \cP \to \rr_{\geq 0}$ by
\[ \fa'(\lambda,\mu ) \ = \
\begin{cases}
	\fa(\lambda) & \text{ if } \mu = (0),\\
	0 & \text{ otherwise.}
\end{cases}
 \]
The functions \ts $\fb',\fc',\fd'$ \ts are defined analogously.  It then follows that
\begin{align*}
\sum_{\lambda  \in \cP} \fa({\lambda}) \fs_{\lambda}  \ = \	\sum_{\lambda,\mu  \in \cP} \fa'({\lambda},\mu) \fs_{\lambda/\mu}\.,
\end{align*}
with  analogous identities for \ts $\fb,\fc,\fd$ \ts and \ts $\fb',\fc',\fd'$.
It is straightforward to verify that $\fa',\fb',\fc',\fd'$ satisfy \eqref{eq:AD-cond-skew}.
The theorem now follows from \eqref{eq:ADS-skew}.
\qed
\medskip

\section{Grothendieck polynomials}\label{s:Gro}

In this section we prove Theorem~\ref{thm:Gro-LPP} via its skew generalization.
We then generalize the Okounkov inequality \eqref{eq:Oko}
to this setting (Corollary~\ref{cor:Gro-LC}).
The proof uses log-supermodularity property for the numbers of increasing
tableaux (Lemma~\ref{lem:fg}), a result of independent interest.

\subsection{Two log-supermodularity lemmas}\label{ss:Gro-two-lemmas}
For partitions $\lambda,\mu  \in \cP$,
let $\ff_{\lambda,\mu}$ be the indicator function for the condition $\lambda \subseteq \mu \subseteq
\widehat{\lambda}$. Recall that \. $\widehat{\lambda}=(\wh\la_1,\wh\la_2,\ldots)$ \. is
given by
\begin{equation}\label{eq:widehat}
\widehat{\lambda}_i \ := \ \min \{\lambda_j + j - 1 \. : \. j \in [i]\}\ts.
\end{equation}

\smallskip

\begin{lemma}\label{lem:ff}
For all \. $\lambda,\mu,\nu,\rho \in \cP$, we have:
	\begin{equation}\label{eq:indic}
	 \ff_{\lambda, \mu} \. \cdot \. \ff_{\nu,\rho}   \ \leq \
\ff_{\lambda\vee \nu, \mu \vee \rho} \. \cdot \. \ff_{\lambda\wedge \nu, \mu \wedge \rho}.
	\end{equation}
\end{lemma}

\smallskip

\begin{proof}
The inequality holds trivially if the LHS is zero, so we
may assume it is equal to~$1$. By definition, this means that
\begin{equation}\label{eq:ff-1}
\lambda \subseteq \mu \subseteq \widehat{\lambda} \quad \text{and} \quad
\nu \subseteq \rho \subseteq \widehat{\nu}.
\end{equation}

\nin
Note that, \eqref{eq:ff-1} implies the following inclusions:
\begin{align*}
	\lambda \vee \nu  \ \subseteq \ \mu \vee \rho,\\
	\lambda \wedge  \nu  \ \subseteq \ \mu \wedge \rho,\\
	\mu \vee \rho \ \subseteq \ \widehat{\lambda} \vee \widehat{\nu} \ \subseteq \  \widehat{\lambda \vee \nu},\\
	\mu \wedge  \rho \ \subseteq \ \widehat{\lambda} \wedge \widehat{\nu} \ = \  \widehat{\lambda \wedge \nu}.
\end{align*}
Here the last equality follows from the fact that, for all \ts $i\geq 1$, we have
\begin{align*}
(\widehat{\lambda} \wedge \widehat{\nu})_i \ & = \  \min\big\{\widehat{\lambda}_i,  \widehat{\nu}_i \big\} \ =_{\eqref{eq:widehat}} \  \min \big\{\min_{j \in [i]} \{\lambda_j + j - 1\}, \.\min_{j \in [i]}\{ \nu_j+j-1\}\big\} \\
& = \  \min_{j \in [i]}\big\{ \min\{\lambda_j,\nu_j\big\} \. + \ts j \ts - \ts 1\big\} \ =_{\eqref{eq:widehat}} \ \big(\widehat{\lambda \wedge \nu}\big)_i\..
\end{align*}
This completes the proof of the lemma.
\end{proof}

\smallskip

Recall the definition of constants \ts $\fg_{\lambda,\mu}=|\IT(\mu/\la)|$ \ts
 the number of increasing tableaux, see~$\S$\ref{ss:notation-Groth}.

\smallskip

\begin{lemma}\label{lem:fg}
For all \. $\lambda,\mu,\nu,\rho \in \cP$ \. satisfying \. $\lambda \subseteq \mu$ \. and \. $\nu \subseteq \rho$,
	\begin{equation}\label{eq:fg}
	\fg_{\lambda, \mu}  \.\cdot \. \fg_{\nu,\rho}   \ \leq \ \fg_{\lambda\vee \nu, \mu \vee \rho} \.\cdot \.\fg_{\lambda\wedge \nu, \mu \wedge \rho}.
\end{equation}
\end{lemma}

\smallskip

The proof of Lemma~\ref{lem:fg} given below is an application of the AD inequality \eqref{eq:AD}.
It is based on the approach in \cite[$\S$8]{CP-multi}.

\smallskip

\begin{proof}[Proof of Lemma~\ref{lem:fg}]
Let \ts $\ell$ \ts and \ts $w$ \ts be the length and width of the partition \ts $\mu \vee \rho$,
respectively. Let \ts $\aL$ \ts be the set of functions \. $T \colon [\ell] \times [w]
\to \nn_{\ge 1} \cup \{-\infty, \infty\}$ \. satisfying
\begin{align*}
	& T(i+1, j) \ge T(i, j) + 1 \quad \text{whenever } (i, j) \text{ and } (i+1, j) \in [\ell] \times [w], \\
	& T(i, j+1) \ge T(i, j) + 1 \quad \text{whenever } (i, j) \text{ and } (i, j+1) \in [\ell] \times [w],\\
	&T(i,j) \in \{1,2,\ldots,i-1\} \cup \{-\infty,\infty\} \quad \text{ for all } (i,j) \in [\ell] \times [w].
\end{align*}
Note that here we adopt the convention that $\infty + 1 = \infty$ and
$-\infty + 1 = -\infty$. It is straightforward to check that $\aL$ is a
distributive lattice, where $\vee$ is given by entrywise maximum and $\wedge$
is given by entrywise minimum.

Let $\fa,\fb,\fc,\fd:\aL \to \{0,1\}$ be the function given by
{\small
\begin{align*}
    \fa(T) &:= \mathbf{1}_{\left\{
T(i,j)
		\right\}}, \quad \text{where} \quad T(i,j)
        \begin{cases}
		= -\infty & \text{if } (i,j) \in \lambda, \\
		\in \nn_{\ge 1} & \text{if } (i,j) \in \mu/\lambda, \\
		= \infty & \text{if } (i,j) \notin \mu
	\end{cases}\\
    \fb(T) &:= \mathbf{1}_{\left\{
T(i,j)
		\right\}}, \quad \text{where} \quad T(i,j)
        \begin{cases}
		= -\infty & \text{if } (i,j) \in \nu, \\
		\in \nn_{\ge 1} & \text{if } (i,j) \in \rho/\nu, \\
		= \infty & \text{if } (i,j) \notin \nu
	      \end{cases}\\
    \fc(T) &:= \mathbf{1}_{\left\{
T(i,j)
		\right\}}, \quad \text{where} \quad T(i,j)
        \begin{cases}
	= -\infty & \text{if } (i,j) \in \lambda \wedge \nu, \\
	\in \nn_{\ge 1} & \text{if } (i,j) \in (\mu \wedge \rho)/(\lambda \wedge \nu), \\
	= \infty & \text{if } (i,j) \notin \mu \wedge \rho
        \end{cases}\\
    \fd(T) &:= \mathbf{1}_{\left\{
T(i,j)
		\right\}}, \quad \text{where} \quad T(i,j)
        \begin{cases}
	    = -\infty & \text{if } (i,j) \in \lambda \vee \nu, \\
	\in \nn_{\ge 1} & \text{if } (i,j) \in (\mu \vee \rho)/(\lambda \vee \nu). \\
	= \infty & \text{if } (i,j) \notin \mu \vee \rho
        \end{cases}
\end{align*}}
Note that the support of $\fa$ consists of functions $T$ that
correspond to tableaux of skew shape $\mu/\lambda$, obtained by
ignoring the entries equal to $\pm\infty$.
Analogous observations apply to $\fb,\fc,\fd$.
This implies that
\begin{align*}
	\sum_{T \in \aL} \fa(T) \ = \  \fg_{\lambda,\mu}\,, \quad
	\sum_{T \in \aL} \fb(T) \ = \ \fg_{\nu,\rho}\,, \quad
	\sum_{T \in \aL} \fc(T) \ = \  \fg_{\lambda \wedge \nu, \mu \wedge \rho}\,, \quad
	\sum_{T \in \aL} \fd(T) \ = \  \fg_{\lambda \vee \nu, \mu \vee \rho}.
	\end{align*}
It is also straightforward to verify that $\fa,\fb,\fc,\fd$ satisfy \eqref{eq:AD-cond}.
The conclusion of the lemma now follows from the classical AD inequality~(Theorem~\ref{thm:AD}).
\end{proof}

\smallskip

\subsection{Proof of Theorem~\ref{thm:Gro-LPP}}\label{ss:Gro-LPP}
It will be convenient for us to define a skew version of Grothendieck polynomials.
For \ts $\ell \geq 1$ \ts and \ts $\lambda,\mu \in \cP^{\ell}$, we define
\[  \tG_{\lambda/\mu}^{(\ell)}  \ :=  \ \sum_{\rho \ts \in\ts \cP^{\ell} \. :
\. \lambda \ts\subseteq \ts\rho \ts\subseteq \ts \widehat{\lambda}} \. \fg_{\lambda,\rho} \.\fs_{\rho/\mu},
\]
where the sum runs over all partitions $\rho$ of length at most $\ell$.\footnote{This
differs from the standard notion of the \emph{skew stable Grothendieck polynomial},
see e.g.~\cite{Buch02}.}
We now prove the following skew version of Theorem~\ref{thm:Gro-LPP}.

\smallskip

\begin{thm}\label{thm:Gro-LPP-skew}
	For all partitions \. $\lambda, \mu,\nu,\rho \in \cP^{\ell}$, we have:
\[
\tG_{\lambda/\mu}^{(\ell)} \. \tG_{\nu/\rho}^{(\ell)} \leqslant_{\fs} \tG_{(\lambda \vee \nu)/(\mu \vee \rho)}^{(\ell)} \. \tG_{(\lambda \wedge \nu)/ (\mu \wedge \rho)}^{(\ell)}.
\]
\end{thm}

\smallskip

Theorem~\ref{thm:Gro-LPP} follows from Theorem~\ref{thm:Gro-LPP-skew}, by substituting
\ts $\nu \gets \mu$; \ts $\mu,\rho \gets \emp$,  and letting \ts $\ell \to \infty$.

\smallskip

\begin{proof}[Proof of Theorem~\ref{thm:Gro-LPP-skew}]
	Let  \. $\fa,\fb,\fc, \fd: \cP \times \cP \to \nn$ \. be the functions given by
	\begin{align*}
		\fa(\alpha,\beta) \ &:= \  \ff_{\lambda,\alpha} \. \fg_{\lambda,\alpha} \. \mathbf{1}_{\{\alpha_{\ell+1}=0\}} \. \mathbf{1}_{\{\beta= \mu\}}\.,\\
		\fb(\alpha,\beta) \ &:= \  \ff_{\nu,\alpha} \. \fg_{\nu,\alpha} \. \mathbf{1}_{\{\alpha_{\ell+1}=0\}} \. \mathbf{1}_{\{\beta= \rho\}}\.,\\
		\fc(\alpha,\beta) \ &:= \  \ff_{\lambda \vee \nu,\alpha} \. \fg_{\lambda \vee \nu,\alpha} \. \mathbf{1}_{\{\alpha_{\ell+1}=0\}} \.
\mathbf{1}_{\{\beta= \mu \vee \rho\}}\.,\\
		\fd(\alpha,\beta) \ &:= \  \ff_{\lambda\wedge \nu,\alpha} \. \fg_{\lambda\wedge \nu,\alpha} \. \mathbf{1}_{\{\alpha_{\ell+1}=0\}} \.  \mathbf{1}_{\{\beta= \mu \wedge \rho\}}.
	\end{align*}
It follows from Lemma~\ref{lem:ff} and Lemma~\ref{lem:fg}, that four functions \. $\fa,\fb,\fc,\fd$ \.
satisfy the assumptions \eqref{eq:AD-cond-skew}.
It also follows from the definition of \. $\fa,\fb,\fc,\fd$, that
\begin{alignat*}{2}
	&\sum_{\alpha,\beta \in \cP } \. \fa(\alpha,\beta) \. \fs_{\alpha/\beta} \ = \
	\tG_{\lambda,\mu}^{(\ell)}\., \quad
	&&\sum_{\alpha,\beta \in \cP } \. \fb(\alpha,\beta)\. \fs_{\alpha/\beta} \ = \
\tG_{\nu,\rho}^{(\ell)}\., \\
	&\sum_{\alpha,\beta \in \cP } \. \fc(\alpha,\beta)\. \fs_{\alpha/\beta} \ = \
\tG_{\lambda \vee \nu,\mu \vee \rho}^{(\ell)}\., \quad
	&&\sum_{\alpha,\beta \in \cP }\. \fa(\alpha,\beta) \. \fs_{\alpha/\beta} \ = \
\tG_{\lambda \wedge \nu,\mu \wedge \rho}^{(\ell)}\ts.
\end{alignat*}
The result now follows from Theorem~\ref{thm:ADS-skew}.
\end{proof}

\smallskip

\subsection{{Ahlswede--Daykin--Grothendieck inequality}} \label{ss:Groth-ADG}
We now give an alternative proof of Theorem~\ref{thm:Gro-LPP}, via a different
generalization of the ADS inequality.  The proof structures are very similar and
both proofs use the same ingredients.

\smallskip

\begin{thm}[{\rm \defn{ADG inequality}\ts}{}]
\label{thm:ADS-Groth}
	Let \ts $\fa,\fb,\fc,\fd: \cP \to \rr_{\geq 0}$ \ts be functions satisfying \eqref{eq:AD-cond}.
	Then we have:
	\begin{equation}\label{eq:ADS-Groth}\tag{ADG}
		\bigg( \sum_{\lambda \ts \in \ts \cP} \. \fa({\lambda}) \. \tG_{\la}(\xb) \bigg)  	\bigg( \sum_{\lambda  \ts \in \ts  \cP} \. \fb({\lambda}) \. \tG_{\la}(\xb) \bigg)  \ \leqslant_{\fs} \ 	\bigg( \sum_{\lambda  \ts \in \ts  \cP} \. \fc({\lambda}) \. \tG_{\la}(\xb) \bigg)
		\bigg( \sum_{\lambda  \ts \in \ts  \cP} \. \fd({\lambda}) \. \tG_{\la}(\xb) \bigg).
	\end{equation}
\end{thm}

\begin{proof}
Definition \eqref{eq:tG} implies that
\begin{equation*}
\sum_{\lambda \ts \in \ts \cP} \. \fa({\lambda}) \. \tG_{\la}(\xb) \ = \ \sum_{\lambda \ts \in \ts \cP} \. f_a(\lambda) \. \fs_{\lambda}(\xb),
\end{equation*}
where
\begin{equation*}
f_a(\lambda) \, := \, \sum_{\mu \ts \in \ts \cQ(\la)} \. \fa({\mu})\.\fg_{\mu,\lambda} \quad \text{and} \quad
\cQ(\la) \, := \, \big\{\ts\mu\in \cP \. : \. \mu \subseteq \lambda \subseteq \widehat{\mu}\ts\big\}.
\end{equation*}
Define \. $f_b,f_c,f_d: \cP \to \rr_{\geq 0}$ \. in a similar way. Then  \eqref{eq:ADS-Groth} can be
rewritten as an instance of the ADS inequality:
\begin{equation*}
		\bigg( \sum_{\lambda  \ts \in \ts  \cP} f_a(\lambda)\fs_{\lambda}(\xb) \bigg)\bigg( \sum_{\lambda  \ts \in \ts  \cP} f_b(\lambda)\fs_{\lambda}(\xb) \bigg)   \ \leqslant_{\fs} \ 	\bigg( \sum_{\lambda  \ts \in \ts  \cP} f_c(\lambda)\fs_{\lambda}(\xb) \bigg)\bigg( \sum_{\lambda  \ts \in \ts  \cP} f_d(\lambda)\fs_{\lambda}(\xb) \bigg)  	.
	\end{equation*}
Therefore, it suffices to show that functions \ts $f_a,f_b,f_c,f_d$ \ts satisfy the assumptions \eqref{eq:AD-cond}:
\begin{equation}\label{eq:local-ADG}
    f_a(\lambda)\. f_b(\nu) \, \leq \, f_c(\lambda \vee \nu) \. f_d(\lambda \wedge \nu) \quad\text{for all} \quad \la,\nu\in \cP.
\end{equation}
Indeed, in the notation above the inequality \eqref{eq:local-ADG} can be rewritten as
\begin{equation*}
\bigg(\sum_{\mu\in \cQ(\la)}  \. \fa({\mu})\fg_{\mu,\lambda}\bigg)
\bigg(\sum_{\mu\in \cQ(\nu)} \. \fb({\mu})\. \fg_{\mu,\nu}\bigg) \ \leq \
\bigg(\sum_{\mu\in \cQ(\la\vee\nu)} \. \fc({\mu})\. \fg_{\mu,\lambda \vee \nu}\bigg)
\bigg(\sum_{\mu\in \cQ(\la\wedge \nu)}  \. \fd({\mu})\. \fg_{\mu,\lambda \wedge \nu}\bigg).
\end{equation*}
The above inequality follows from AD inequality where the \eqref{eq:AD-cond} is given as the product of inequalities
\eqref{eq:indic}, \eqref{eq:fg} and \eqref{eq:AD-cond}.
This completes the proof.
\end{proof}

\smallskip

\subsection{Log-concavity}\label{ss:Gro-LC}
The following result is an immediate corollary of Theorem~\ref{thm:Gro-LPP-skew}
and Speyer's Theorem~\ref{t:speyer}, by setting \. $f: \lambda \mapsto \tG^{(\ell)}_{\lambda}$ \.
and taking the limit \ts $\ell \to \infty$.  It is the analogue of Okounkov inequality
(Corollary~\ref{cor:sh-Oko}) for stable Grothendieck polynomials.

\smallskip

\begin{cor}[{\rm \defn{Grothendieck log-concavity}\ts}{}]\label{cor:Gro-LC}
	For all partitions \. $\lambda, \mu \in \cP$, we have:
\[ \tG_{\lambda} \. \tG_{\mu} \ \leqslant_{\fs} \   \tG_{\left \lceil (\lambda+\mu)/2 \right\rceil}
\. \tG_{\left \lfloor (\lambda+\mu)/2 \right\rfloor}.
\]
\end{cor}

\smallskip

The corollary can be generalized to a skew version, which further generalizes
the {skew Okounkov inequality} (Theorem~\ref{thm:Oko-skew}), in a similar way as
Theorem~\ref{thm:Oko-skew-global}.

\smallskip

\begin{thm}\label{thm:Gro-LC}
	For all \. $\lambda,\mu, \nu,\rho \in \cP$, we have:
\[  \tG_{\lambda/\mu} \. \tG_{\nu/\rho} \ \leqslant_{\fs} \
\tG_{\left \lceil (\lambda+\nu)/2 \right\rceil /  \left \lceil (\mu+\rho)/2 \right\rceil} \.
\tG_{\left \lfloor (\lambda+\nu)/2 \right\rfloor /\left \lfloor (\mu+\rho)/2 \right\rfloor
}.   \]
\end{thm}

\smallskip

This proof closely follows that of Theorem~\ref{thm:Oko-skew-global}, and will also
be omitted.  As we mentioned earlier, these are not the ``right'' skew stable
Grothendieck polynomials.  Finding a generalization of Corollary~\ref{cor:Gro-LC}
to the latter remains an open problem, see~$\S$\ref{ss:finrem-Mih}.

\medskip

%%%%%%%%%%%%%%%%%%%%%%%%%%%%%%%%%%%%%%%%%%%%%%%%%%%%%%%%%%%%%%%%%%%%%%%%%%%%%%%%%%%%%%%%%%%
%%%%%%%%%%%%%%%%%%%%%%%%%%%%%%%%%%%%%%%%%%%%%%%%%%%%%%%%%%%%%%%%%%%%%%%%%%%%%%%%%%%%%%%%%%%
%%%%%%%%%%%%%%%%%%%%%%%%%%%%%%%%%%%%%%%%%%%%%%%%%%%%%%%%%%%%%%%%%%%%%%%%%%%%%%%%%%%%%%%%%%%
%%%%%%%%%%%%%%%%%%%%%%%%%%%%%%%%%%%%%%%%%%%%%%%%%%%%%%%%%%%%%%%%%%%%%%%%%%%%%%%%%%%%%%%%%%%
%%%%%%%%%%%%%%%%%%%%%%%%%%%%%%%%%%%%%%%%%%%%%%%%%%%%%%%%%%%%%%%%%%%%%%%%%%%%%%%%%%%%%%%%%%%
%%%%%%%%%%%%%%%%%%%%%%%%%%%%%%%%%%%%%%%%%%%%%%%%%%%%%%%%%%%%%%%%%%%%%%%%%%%%%%%%%%%%%%%%%%%
%%%%%%%%%%%%%%%%%%%%%%%%%%%%%%%%%%%%%%%%%%%%%%%%%%%%%%%%%%%%%%%%%%%%%%%%%%%%%%%%%%%%%%%%%%%

%\newpage

\section{Dual stable Grothendieck polynomials} \label{s:dual}

In this section we prove the dual Grothendieck LPP inequality
(Theorem~\ref{thm:Gro-LPP-dual}).  We then generalize the Okounkov
inequality \eqref{eq:Oko} to this setting (Corollary~\ref{cor:Gro-LC-dual}).
The proof uses log-supermodularity property for the numbers of elegant fillings
(Lemma~\ref{lem:fg*}), a result of independent interest.

\subsection{Two log-supermodularity lemmas}\label{ss:Gro-dual}
For partitions \ts $\lambda,\mu  \in \cP$, let \ts $\ff_{\lambda,\mu}^*$ \ts
be the indicator function for the condition \ts $\mu \subseteq \lambda$.

\smallskip
\begin{lemma}\label{lem:ff*}
	For all \ts $\lambda,\mu,\nu,\rho \in \cP$, we have:
\begin{equation*}
	\ff_{\lambda, \mu}^* \. \cdot \. \ff_{\nu,\rho}^*   \ \leq \
\ff_{\lambda\vee \nu, \mu \vee \rho}^*  \.\cdot \.\ff_{\lambda\wedge \nu, \mu \wedge \rho}^*.
\end{equation*}
\end{lemma}

\smallskip

The proof is straightforward.
Next, recall that, for the number \ts $\fg^*_{\lambda,\mu}=|\EF(\la/\mu)|$ \ts
is the number of elegant fillings of~$\la/\mu$, see~$\S$\ref{ss:notation-dual}.

\smallskip

\begin{lemma}\label{lem:fg*}
	For all \. $\lambda,\mu,\nu,\rho \in \cP$ \. such that
\. $\lambda \supseteq  \mu$ \. and \. $\nu \supseteq \rho$, we have:
	\begin{equation*}
		\fg_{\lambda, \mu}^* \.\cdot\. \fg_{\nu,\rho}^*   \ \leq \
\fg_{\lambda\vee \nu, \mu \vee \rho}^* \.\cdot\. \fg_{\lambda\wedge \nu, \mu \wedge \rho}^*.
	\end{equation*}
\end{lemma}

\smallskip

The proof of Lemma~\ref{lem:fg*} closely parallels the proof of Lemma~\ref{lem:fg}.
Since the necessary modifications involve subtle technical details,
we include the full argument for completeness.

\smallskip

\begin{proof}[Proof of Lemma~\ref{lem:fg*}]
	Let $\ell$ and $w$ be the length and width of the partition $\lambda \vee \nu$,
	respectively. Let \ts $\aL$ \ts be the set of functions \. $T \colon [\ell] \times [w]
	\to \nn_{\ge 1} \cup \{-\infty, \infty\}$ \. satisfying
	\begin{align*}
		& T(i+1, j) \ge T(i, j) + 1 \quad \text{for all } \ (i, j)\ts, \. (i+1, j) \in [\ell] \times [w], \\
		& T(i, j+1) \ge T(i, j)  \quad \text{for all } \ (i, j)\ts, \. (i, j+1) \in [\ell] \times [w],\\
		&T(i,j) \in \{1,2,\ldots,i-1\} \cup \{-\infty,\infty\} \quad \text{ for all } \ \, (i,j) \in [\ell] \times [w].
	\end{align*}
	Note that here we adopt the convention that $\infty + 1 = \infty$ and
	$-\infty + 1 = -\infty$. It is straightforward to check that \ts $\aL$ \ts is a
	distributive lattice, where $\vee$ is given by entrywise maximum and $\wedge$
	is given by entrywise minimum.
	
	Let $\fa,\fb,\fc,\fd:\aL \to \{0,1\}$ be the function given by
{\small
	\begin{align*}
		\fa(T) &:= \mathbf{1}_{\left\{
		T(i,j)
		\right\}}, \quad \text{where} \quad T(i,j)
        \begin{cases}
			= -\infty & \text{if } (i,j) \in \mu, \\
			\in \nn_{\ge 1} & \text{if } (i,j) \in \lambda/\mu, \\
			= \infty & \text{if } (i,j) \notin \lambda
		\end{cases}\\
        \fb(T) &:= \mathbf{1}_{\left\{
		T(i,j)
		\right\}}, \quad \text{where} \quad T(i,j)
        \begin{cases}
			= -\infty & \text{if } (i,j) \in \rho, \\
			\in \nn_{\ge 1} & \text{if } (i,j) \in \nu/\rho, \\
			= \infty & \text{if } (i,j) \notin \nu
		\end{cases}\\
        \fc(T) &:= \mathbf{1}_{\left\{
		T(i,j)
		\right\}}, \quad \text{where} \quad T(i,j)
        \begin{cases}
			= -\infty & \text{if } (i,j) \in \mu \wedge \rho, \\
			\in \nn_{\ge 1} & \text{if } (i,j) \in (\lambda \wedge \nu)/(\mu \wedge \rho), \\
			= \infty & \text{if } (i,j) \notin \lambda \wedge \nu
		\end{cases}\\
        \fd(T) &:= \mathbf{1}_{\left\{
		T(i,j)
		\right\}}, \quad \text{where} \quad T(i,j)
        \begin{cases}
			= -\infty & \text{if } (i,j) \in \mu \vee \rho, \\
			\in \nn_{\ge 1} & \text{if } (i,j) \in (\lambda \vee \nu)/(\mu \vee \rho), \\
			= \infty & \text{if } (i,j) \notin \lambda \vee \nu
		\end{cases}
	\end{align*}
}

	Note that the support of $\fa$ consists of functions $T$ that
	correspond to tableaux of skew shape~$\la/\mu$, obtained by
	ignoring the entries equal to $\pm\infty$.
	Analogous observations apply to $\fa,\fb,\fc,\fd$.
	This implies that
	\begin{align*}
		\sum_{T \in \aL} \fa(T) \ = \  \fg_{\lambda,\mu}^*\., \quad
		\sum_{T \in \aL} \fb(T) \ = \ \fg_{\nu,\rho}^*\., \quad
		\sum_{T \in \aL} \fc(T) \ = \  \fg_{\lambda \wedge \nu, \mu \wedge \rho}^*\., \quad
		\sum_{T \in \aL} \fd(T) \ = \  \fg_{\lambda \vee \nu, \mu \vee \rho}^*\..
	\end{align*}
	It is also straightforward to verify that \ts $\fa,\fb,\fc,\fd$ \ts satisfies \eqref{eq:AD-cond}.
	The conclusion of the lemma now follows from the classical AD inequality~(Theorem~\ref{thm:AD}).
\end{proof}

\smallskip

\subsection{Proof of Theorem~\ref{thm:Gro-LPP-dual}}
The proofs of Theorem~\ref{thm:Gro-LPP-dual} closely follow the proofs
of the Grothendieck LPP inequality (Theorem~\ref{thm:Gro-LPP}),
with the only difference being Lemma~\ref{lem:ff} and Lemma~\ref{lem:fg}
being substituted with two lemmas above. We omit the details. \qed

\smallskip

\subsection{Log-concavity}\label{ss:dual-further}
We can now obtain a natural counterpart of Corollary~\ref{cor:Gro-LC} for
the dual stable Grothendieck polynomials.

\smallskip

\begin{cor}[{\rm \defn{dual Grothendieck log-concavity}\ts}{}]\label{cor:Gro-LC-dual}
	For any partitions $\lambda, \mu \in \cP$, we have:
	\[ \fG_{\lambda}^* \. \fG_{\mu}^* \ \leqslant_{\fs} \
    \fG^*_{\left \lceil (\lambda+\mu)/2 \right\rceil}
    \.	\fG^*_{\left \lfloor (\lambda+\mu)/2 \right\rfloor}.
	\]
\end{cor}

\smallskip

The proofs of Corollary~\ref{cor:Gro-LC-dual} closely follows the proof of the Grothendieck log-concavity~(Corollary~\ref{cor:Gro-LC}),
with the only difference being Lemma~\ref{lem:ff} and Lemma~\ref{lem:fg}
being substituted with two lemmas above.
It is the analogue of Okounkov inequality
(Corollary~\ref{cor:sh-Oko}) for the dual stable Grothendieck polynomials.

\medskip

\section{Further results and applications}\label{s:further}

In this section we present four largely unrelated stories of
further applications and generalizations of the ADS inequality.
We present them without proofs or details, and with only a brief explanation
of their background and significance.\footnote{In the ideal world, each story
would be expanded to a separate section, where we would include complete proofs
and further explore these variations.  Alas, we do not live in that world.
The details are available upon request.}

\subsection{Schur measures}\label{ss:Schur measures}
The ADS inequality (Theorem~\ref{thm:ADS}) can be further generalized
to positivity results for functions on Schur measures.  Namely,
let \ts $\xb := (x_1, x_2, \ldots)$ \ts and \ts
$\yb := (y_1, y_2, \ldots)$ \ts be sets of variables, and let
\ts $\fs_{\lambda}(\xb)$ \ts and \ts $\fs_{\mu}(\yb)$
\ts denote Schur polynomials in these sets of variables.
Recall that \ts $\La(\xb) \otimes \La(\yb)$ \ts has a basis \.
$\cD:=\{\fs_{\lambda}(\xb)\fs_{\mu}(\yb) \.:\. \lambda, \mu \in \mathcal{P}\}$.
For \ts $f,g \in \La(\xb) \otimes \La(\yb)$, we write \. $f \geqslant_\otimes g$ \.
if \ts $(f-g)$ \ts has non-negative coefficients when expanded in~$\cD$.

\smallskip

\begin{thm}[{\rm \defn{Ahlswede--Daykin inequality for Schur measures}}]\label{thm:ADS-Oko}
	Let \. $\fa,\fb,\fc,\fd: \cP \to \rr_{\geq 0}$ \. be four functions satisfying \eqref{eq:AD-cond}.
	Then we have:
{\small
	\begin{equation*}
		\bigg( \sum_{\lambda \in \cP} \fa({\lambda}) \fs_{\lambda}(\xb) \fs_{\lambda}(\yb)\bigg)  	\bigg( \sum_{\lambda \in \cP} \fb({\lambda}) \fs_{\lambda}(\xb) \fs_{\lambda}(\yb)\bigg)  \, \leqslant_{\otimes} \, 	\bigg( \sum_{\lambda \in \cP} \fc({\lambda}) \fs_{\lambda}(\xb) \fs_{\lambda}(\yb)\bigg)
			\bigg( \sum_{\lambda \in \cP} \fd({\lambda}) \fs_{\lambda}(\xb) \fs_{\lambda}(\yb)\bigg).
	\end{equation*}
}\end{thm}

\smallskip

The proof of this theorem is completely analogous to that of the ADS
inequality (Theorem~\ref{thm:ADS}), following from a consecutive application
of Theorem~\ref{thm:orch-ext} and  Theorem~\ref{thm:orch}. We omit the details.

The theorem has a natural application to the setting of \emph{Schur measures}.
Fix two vectors \. $\aai=(a_1,a_2,\ldots)$, \.
$\bbi=(b_1,b_2,\ldots) \in \rr_{\geq 0}^\infty$ \. with finite support.
The corresponding \defnb{Schur measure} \ts on $\cP$, introduced in \cite{Oko01}, is defined by
\begin{align*}
	\Pb_{\aai,\bbi} [\lambda] \ := \ \frac{1}{Z_{\ab,\bb}} \fs_{\lambda}(\ab)
	\fs_{\lambda}(\bb)\ts,
\end{align*}
where \ts $Z_{\aai,\bbi}$ \ts is given by the \emph{Cauchy identity} (see, e.g., \cite{Mac95,EC}):
\[ Z_{\aai,\bbi} \ := \
\sum_{\lambda \in \cP}\. \fs_{\lambda}(\aai) \.\fs_{\lambda}(\bbi) \ = \
\prod_{i, j\geq 1} \frac{1}{1 - a_i b_j}\..
\]
The following result by Baslingker--Krishnapur--Madiman shows
that the marginal distribution of each part  is log-concave:

\smallskip

\begin{thm}[{\cite[Thm.~5]{BKM}}]\label{thm:BKM}
For all $i \geq 1$ and $k \geq 1$ and any $\ab,\bb \in \rr_{\geq 0}^\infty$,
\[   \Pb_{\ab,\bb}[\lambda_i=k+1] \.  \Pb_{\ab,\bb}[\lambda_i=k-1] \ \leq \
\big(\Pb_{\ab,\bb}[\lambda_i=k]\big)^2.
\]
\end{thm}

\smallskip

By Theorem~\ref{thm:ADS-Oko}, we can strengthen the inequality
in Theorem~\ref{thm:BKM} to Schur positivity:

\smallskip

\begin{thm}\label{prop:BKM}
	For all \ts $i \geq 1$ \ts and \ts $k\geq 1$, we have:
	\begin{align*}
		\bigg( \sum_{\lambda \in \cP, \.  \lambda_i =k+1} \fs_{\lambda}(\xb) \fs_{\lambda}(\yb) \bigg)
				\bigg( \sum_{\lambda \in \cP, \.  \lambda_i =k-1} \fs_{\lambda}(\xb) \fs_{\lambda}(\yb) \bigg)
				\ \leq_{\otimes}
			\bigg( \sum_{\lambda \in \cP, \.  \lambda_i =k} \fs_{\lambda}(\xb) \fs_{\lambda}(\yb) \bigg)^2.
	\end{align*}
\end{thm}

\smallskip

The proof of Theorem~\ref{prop:BKM} proceeds analogously to the proof of
the Okounkov inequality (Theorem~\ref{thm:Oko}) given in \cite{LPP07},
but with Theorem~\ref{thm:ADS-Oko} playing the role of the ADS inequality
(cf.\ Corollary~\ref{cor:sh-Oko}).
We omit the details.

\smallskip

\subsection{Klartag--Lehec inequality}\label{ss:HKS}
We now consider the following \defn{Klartag--Lehec {\rm (KL)} inequality},
which was first discovered as a discrete version of a result in convex geometry,
see a discussion in~\cite{MM24}.  It can be viewed as version of the AD inequality
(Theorem~\ref{thm:AD}), where the log-supermodularity assumption is replaced by log-concavity.

\smallskip

\begin{thm}[{\rm \defn{KL~inequality{}\ts}~\cite[Thm.~1.4]{KL}}{}]\label{thm:HKS}
Let \. $\fa,\fb,\fc,\fd:\zz^\ell \to \rr_{\geq 0}$ \. be functions satisfying
\begin{equation}\label{eq:HKS-cond}\tag{KL-cond}
	\fa(\lambda) \, \fb(\mu) \ \leq \ \fc (\left \lceil (\lambda+\mu)/2\right \rceil) \.
	\fd (\left\lfloor(\lambda+\mu)/2\right \rfloor).
\end{equation}
Then we have:
\begin{equation}\label{eq:KL}\tag{KL}
\bigg( \sum_{\lambda \in \zz^\ell} \fa({\lambda})\bigg)
\bigg( \sum_{\lambda \in \zz^\ell} \fb({\lambda}) \bigg)
\ \leq \
\bigg( \sum_{\lambda \in \zz^\ell} \fc({\lambda}) \bigg)
\bigg( \sum_{\lambda \in \zz^\ell} \fd({\lambda}) \bigg).
\end{equation}
\end{thm}

\smallskip

It is natural to ask whether this inequality has a Schur positive version:

\smallskip

\begin{conj}\label{conj:SHKS}
	Let \. $\fa,\fb,\fc,\fd:\cP^\ell \to \rr_{\geq 0}$ \. be functions with finite support,
satisfying \eqref{eq:HKS-cond}.  Then we have:
\begin{equation}\label{eq:KLS}\tag{KLS}
\bigg( \sum_{\lambda \in \cP^\ell} \fa({\lambda}) \fs_\lambda\bigg)
	\bigg( \sum_{\lambda \in \cP^\ell} \fb({\lambda}) \fs_{\lambda} \bigg)
	\ \leqslant_{\fs} \
	\bigg( \sum_{\lambda \in \cP^\ell} \fc({\lambda}) \fs_\lambda \bigg)
	\bigg( \sum_{\lambda \in \cP^\ell} \fd({\lambda}) \fs_\lambda\bigg).
\end{equation}	
\end{conj}

\smallskip

We checked the conjecture computationally, on all partitions of length at most~$3$
and size at most~$10$.
The conjecture is best understood in the context of discrete convex analysis,
see \cite{Mur03}, by comparing it with Speyer's framework in \cite{Spe26}. There,
Speyer presents four successively stronger
positivity properties that happen to be equivalent for single real functions
\cite[Thm.~4.4]{Spe26}, but not for single Schur positivity for symmetric functions
\cite[Rem.~4.19]{Spe26}.  In this context, Speyer's Theorem~\ref{t:speyer}
says that log-supermodularity implies log-concavity for single symmetric functions.

It would be interesting to investigate if Speyer's framework extends to
four functions inequalities, i.e.\  whether \eqref{eq:ADS} is stronger
than \eqref{eq:KLS}?  What makes Conjecture~\ref{conj:SHKS} so challenging,
is that it is neither weaker nor stronger than the AD inequality.
If true, the conjecture would imply Schur differential log-concavity
(Conjecture~\ref{conj:SLC} below).  It is also worth noting that  the KL inequality
(Theorem~\ref{thm:HKS}) admits a unified proof via the AD inequality \cite{HKS}.
This lends further credence to the conjecture (cf.~$\S$\ref{ss:finrem-four}).

\smallskip

\smallskip

\subsection{Differential operators}\label{ss:derivatives}
Let \ts $\xb = (x_1, x_2, \ldots)$ \ts and \ts $\zb = (z_1, z_2, \ldots)$ \ts
be two sets of variables.  Let \ts $\La(\xb)[[\zb]]$ \ts denote the ring of
formal power series in variables~$\ts\zb$, with coefficients in the ring of
symmetric functions \ts $\La(\xb)$.
For a sequence \ts $\la\in \nn^\infty$, define the \defnb{differential operator}
\. $\partial^{\lambda}$ \. with respect to~$\zb$,  by
$$ \partial^{\lambda} \ :=  \ \prod_{i\geq 1} \left(\tfrac{\partial}{\partial z_i}\right)^{\lambda_i}. $$
We denote by $\eb_i$ the  $i$-th standard unit vector in $\rr^\infty$.

Differential operators provide a natural framework for the theory of \emph{stable}
and \emph{Lorentzian polynomials}, see, e.g., \cite{BBL, BH}. They
act as fundamental operations that preserve the underlying geometric
properties. Motivated by these connections, we define the
\defnb{normalized Schur generating function} \ts by
\begin{equation}\label{eq:nS}
	\mathfrak{S}(\zb) \ := \ \sum_{\lambda \ts\in \ts\cP}  \, \fs_{\lambda}(\xb) \,  \prod_{i\geq 1} \frac{z_i^{\lambda_i}}{\lambda_i!} \ \in \. \La(\xb)[[\zb]]\ts.
\end{equation}

\smallskip

\begin{thm}[{\rm \defn{Schur differential total positivity}\ts}{}]\label{thm:STP}
For  all sequences \. $\lambda \in \rr^{\infty}_{\geq 0}$, \. $\zb \in \rr_{\geq 0}^\infty$ \.
with finite support, and for all distinct \ts $i,j\geq 1$, we have:
\begin{equation}\label{eq:STP}\tag{STP}
	\big(\partial^{\lambda+\eb_i}\mathfrak{S}(\zb)\big) \.  \big( \partial^{\lambda+\eb_j} \mathfrak{S}(\zb) \big) \ \leqslant_{\fs} \  \big(\partial^{\lambda+\eb_i+\eb_j} \mathfrak{S}(\zb)\big) \. \big(\partial^{\lambda}\mathfrak{S}(\zb)\big).
\end{equation}
\end{thm}

\smallskip
Without Schur positivity, the inequality
in \eqref{eq:STP} is known in the literature as
\defng{multivariate total positivity} \ts of order~$2$, see e.g.\ \cite{KR}.
When the inequality is reversed, it is known as the
\defng{Rayleigh property}, see e.g.\ \cite{Bra15}.
The proof of Theorem~\ref{thm:STP} follows the same argument as the
proof of Theorem~\ref{thm:LPP-global},
except that the four functions are modified to keep track of the
weight \. $\prod_{i\geq 1} ({z_i^{\lambda_i}})/{\lambda_i!}$.
We omit the details.

\smallskip

\begin{conj}[{\rm \defn{Schur differential log-concavity}\ts}{}]\label{conj:SLC}
	For  all  \. $\lambda \in \rr^{\infty}_{\geq 0}$, $\zb \in \rr_{\geq 0}^\infty$ \. with finite support and any  $i\geq 1$,
	\begin{equation}\label{eq:SLC}\tag{SLC}
		\big(\partial^{\lambda+2\eb_i}\mathfrak{S}(\zb)\big) \.  \big( \partial^{\lambda} \mathfrak{S}(\zb) \big) \ \leqslant_{\fs} \  \big(\partial^{\lambda+\eb_i} \mathfrak{S}(\zb)\big)^2.
	\end{equation}
\end{conj}

\smallskip

Unfortunately, our approach via the skew ADS inequality (Theorem~\ref{thm:ADS})
no longer applies for certain ranges of parameters.  Nevertheless,
there is a good reason to believe that Conjecture~\ref{conj:SLC} holds (cf.~$\S$\ref{ss:Lor}).
First, without Schur positivity, the inequality holds for  evaluations
at \ts $\xb = \ab$, for all finitely supported \ts $\ab \in \rr_{\geq 0}^\infty$.
Indeed, this follows by combining the KL inequality (Theorem~\ref{thm:HKS})
and the LPP inequality~(Theorem~\ref{thm:LPP}).  Second, relatedly, Conjecture~\ref{conj:SLC}
follows from Conjecture~\ref{conj:SHKS}, the Schur positive version of the KL inequality.

\smallskip

\subsection{Supersymmetric Schur functions}\label{ss:Super}
As before, let \ts $\xb := (x_1, x_2, \ldots)$ \ts and \ts
$\yb := (y_1, y_2, \ldots)$ \ts be two sets of variables.
We say that a symmetric function \ts $p(\xb, \yb)\in \La(\bx)\otimes \La(\by)$ \ts
is \defnb{supersymmetric}, if under the substitution \ts $x_1 = t$,  $y_1 = -t$,
the resulting series is independent of~$t$.  We denote by \ts $\La(\xb/\yb)$ \ts
the ring  of supersymmetric polynomials in~$\xb$ and~$\yb$.  See \cite[$\S$6]{Mac92}
for an overview.

For a partition \ts $\lambda \in \cP$, the \defnb{supersymmetric Schur function} \.
$\fs_{\lambda}(\xb/\yb)$, defined by Berele and Regev~\cite[p.~152]{BR87}, is given by
\begin{equation*}
	\fs_{\lambda}(\xb/\yb) \ := \  \sum_{\mu \subseteq \lambda} \. \fs_{\mu}(\xb) \. \fs_{\lambda'/\mu'}(\yb),
\end{equation*}
Note that we recover the classical Schur functions \ts $\fs_{\lambda}(\xb)$ \ts
by setting \ts $\yb = \mathbf{0}$.

Supersymmetric Schur functions \. $\{\fs_{\lambda}(\xb/\yb)\}$ \. form a basis for the ring \ts $\La(\xb/\yb)$, see~\cite{Ste}.
For elements \ts $f,g \in \La(\xb/\yb)$, we write \ts $f \geqslant_{\oslash} g$ \ts if \ts
$(f-g)$ \ts nonnegative coefficients when expanded in this basis.
The ADS inequality (Theorem~\ref{thm:ADS}) has the following generalization to supersymmetric functions.

\smallskip

\begin{thm}[{\rm \defn{supersymmetric ADS inequality}\ts}{}]\label{thm:ADS-Super}
	Let \. $\fa,\fb,\fc,\fd: \cP \to \rr_{\geq 0}$ \. be functions satisfying \eqref{eq:AD-cond}.
	Then we have:
{\small
	\begin{equation}\label{eq:ADS-Super}
		\bigg( \sum_{\lambda \in \cP} \fa({\lambda}) \fs_{\lambda}(\xb/\yb) \bigg)  	\bigg( \sum_{\lambda \in \cP} \fb({\lambda}) \fs_{\lambda}(\xb/\yb) \bigg)  \ \leqslant_\oslash \ 	\bigg( \sum_{\lambda \in \cP} \fc({\lambda}) \fs_{\lambda}(\xb/\yb) \bigg)
		\bigg( \sum_{\lambda \in \cP} \fd({\lambda}) \fs_{\lambda}(\xb/\yb) \bigg).
	\end{equation}
}
\end{thm}

\smallskip

\begin{proof}
	Write the defect of \eqref{eq:ADS-Super}, i.e.\ the difference of two sides, in the basis \. $\{\fs_{\lambda}(\xb/\yb)\}$ \. as
\[ \sum_{\lambda \in \cP} \. \ff_{\lambda} \cdot \fs_{\lambda}(\xb/\yb)\ts.
\]
	From above, it suffices to show that \ts $\ff_{\lambda}\ge 0$, for all \ts $\la\in \cP$.
	Substituting \ts $\yb \gets \0$ \ts into \eqref{eq:ADS-Super}, gives the inequality
\eqref{eq:ADS}, which
is Schur positive by Theorem~\ref{thm:ADS}.
On the other hand, this is also equal to
		\[ \sum_{\lambda \in \cP} \. \ff_{\lambda} \cdot \fs_{\lambda}(\xb)\ts. \]
Thus, all \ts $\ff_{\lambda}$'s are nonnegative, as desired.
\end{proof}

\medskip
\section{Final remarks and open problems} \label{s:finrem}

\subsection{Correlation inequalities}\label{ss:finrem-more-corr}
From the ocean of other correlation inequalities in probability and
statistical physics, let us single out the \defng{Griffiths inequality} \cite{Gri67}
for the Ising model, and the \defng{van den Berg–Kesten {\rm (BK)} inequality}
\cite{BK85} for events occurring with disjoint evidence. These
inequalities are also proved by induction and have a number of
advanced generalizations, see \cite{KS68,Gin70,GHS70} and \cite{BCR99,Smy13,Rei00},
respectively.
When the inductive approach fails, such as for the monotonicity of percolation
\cite[$\S$8.4]{GPZ25}, establishing the inequality becomes a major challenge.
See also \cite{Lig85} for applications of correlation inequalities
to other problems in statistical physics.

\subsection{Equality conditions}\label{ss:finrem-eq}
There is long ongoing effort to characterize the equality conditions
for more and more general correlation inequalities, from  HK inequality
to the FKG inequality, and all the way to the AD inequality, see
\cite{DKW79,Beck90,MR94,AK95}.  The problem was ultimately resolved in a
forthcoming paper by the first and third authors \cite{CP-ADeq},
which also characterized equality condition of the skew LPP and skew
Okounkov inequalities.

It would be an interesting and difficult challenge to characterize
equality conditions for the ADS inequality, as even the equality conditions
for the RLS inequality remain open.  Note that there is no reason
these equality conditions should have a tractable characterizations,
since many inequalities do not, see \cite{IP22,CP-coinc,CP-AF}.

\subsection{Combinatorial interpretations}\label{ss:finrem-CI}
In a forthcoming paper \cite{PS26}, the last two authors used
a generalized Littlewood--Richardson rule by Nguyen--Pylyavskyy \cite{NP25},
to give a combinatorial interpretation for the defect \ts $\de(\la,\mu,\nu)$ \ts
of Schur multiplicities in \eqref{eq:LPP}:
$$
\de(\la,\mu,\nu) \, := \, c^\la_{\mu\vee\nu,\ts \mu\wedge\nu} \. - \. c^\la_{\mu, \nu}\..
$$
This resolved the open problem in \cite[$\S$7.4]{Pak-OPAC}.

It would be interesting to extend the result to the defect of Schur multiplicities
of \eqref{eq:ADS}.  Note that one cannot use our proof of ADS inequality as a black
box, since it uses \eqref{eq:AD}, among other ingredients.  However, the defect of the
AD inequality \emph{does not} \ts have a combinatorial interpretation, understood
as being in $\SP$, under certain complexity assumptions \cite[Ch.~6]{Gla-thesis}.
In a positive direction, the defect of the HK inequality \eqref{eq:HK} does
have a combinatorial interpretation \cite[Prop.~5.8]{Pak-OPAC}.

\subsection{Representation theory aspects}\label{ss:finrem-Oko}
In \cite{Oko01,Oko03}, Okounkov studied log-concavity of Kostka numbers and
LR coefficients from representation theoretic point of view.  Although his most
ambitious conjecture was disproved in~\cite{CDW07}, the LPP inequality
(Theorem~\ref{thm:LPP}) and the M-convexity for Kostka numbers \cite{HMMS22}
can be viewed as special cases.  One can follow Okounkov's approach to
restate \eqref{eq:ADS} in the language of certain $S_n$-modules being
submodules of others; we avoid the temptation.

Let us also mention a conjecture by Gui \cite[Conj.~5.1]{Gui25},
generalizing the Okounkov inequality \eqref{eq:Oko} to the reduced
(stable) Kronecker coefficients.  By the result in \cite{AS20},
this conjecture can be interpreted as log-concavity for
a certain basis of symmetric functions.

\subsection{Speyer's theorem}\label{ss:finrem-Speyer}
In \cite{Spe26}, Speyer proved a Schur positive inequality conjectured
previously by Lam--Postnikov--Pylyavskyy, see \cite{DP07}.
\defng{Speyer's inequality} \ts extended \eqref{eq:LPP}
in the direction not covered in this paper.
Curiously, \eqref{eq:AD} is a key technical tool in the proof of Speyer's theorem.
It would be interesting to find a skew version of this result.\footnote{After this
paper was written, we learned of a forthcoming paper \cite{LN26} which proves a
skew version of Speyer's inequality.} Even better would be to obtain a
four functions version of his inequality, see below.

\subsection{Skew stable Grothendieck polynomials}\label{ss:finrem-Mih}
It would be interesting to find a skew version of supermodularity of
Grothendieck polynomials (Theorem~\ref{thm:Gro-LPP})
for the definition by Buch~\cite{Buch02}.  As emphasized previously,
the polynomials used in Theorem~\ref{thm:Gro-LPP-skew} are a
simplified variant tailored specifically to our proof,
rather than the conventional definition.
One potential approach to proving this general skew inequality would be to
combine the skew ADS inequality with
combinatorial  formulas for skew stable Grothendieck polynomials
developed by Chan--Pflueger \cite{CP-not-us}.

\subsection{Lorentzian Schur positivity}\label{ss:Lor}
In view of discussion in~$\S$\ref{ss:derivatives}, it would be interesting
to develop Schur positive analogues of stable and Lorentzian polynomials.
This could allow one to translate the wealth of existing intuition and
techniques from the geometry of polynomials directly into the realm of
Schur positivity.
This could lead to many new Schur positive inequalities,
potentially resolving Conjecture~\ref{conj:SLC}.

\subsection{Four functions style discrete convex analysis}\label{ss:finrem-four}
Continuing the discussion above and in~$\S$\ref{ss:HKS}, it would be interesting
to develop a comprehensive study of the discrete convex analysis on distributive
lattices in the style of \eqref{eq:AD}, where convexity conditions such as
log-supermodularity \eqref{eq:lsm} is substituted with four functions conditions
such as \eqref{eq:AD-cond}.  A good roadmap would be four conditions in \cite[Rem.~4.19]{Spe26}.
One can view both \eqref{eq:AD}, \eqref{eq:RLS} and \eqref{eq:KL} as basic result
in this theory, all leading to their Schur positive analogues: \eqref{eq:ADS} in Theorem~\ref{thm:ADS},
\eqref{eq:LSS} in Theorem~\ref{thm:SAAD}, and \eqref{eq:KLS} in Conjecture~\ref{conj:SHKS}.

\vskip.7cm

\section*{Acknowledgements}
We are grateful to Leonardo Mihalcea for introducing us to conjectures in \cite{Mih} that led to Theorem~\ref{thm:Gro-LPP} and
Corollary~\ref{cor:Gro-LC}.
We also thank Nikita Gladkov, June Huh, Tom Hutchcroft, Son Nguyen, Pasha Pylyavskyy,
Colleen Robichaux and Linus Setiabarata for inspiring discussions and helpful remarks.
In preparation of this paper, we extensively consulted with Per Alexandersson's
online catalogue \cite{Ale20}, which we found to be a rich source of
results and helpful references.

The first author (SHC) was partially supported by the NSF grant DMS-2246845,
and would like to thank the National University of Singapore for their
warm hospitality during his sabbatical in the Spring of 2026. 
The second author (HC) was partially supported by the Rutgers University 
SAS Fellowship.  The third author (IP) was partially supported by the 
NSF grant CCF-2302173.

\end{document}